\documentclass[a4paper,11pt]{article}

\usepackage{amsmath, amsthm, amsfonts, amssymb}
\usepackage[all,cmtip]{xy}
\usepackage{enumerate}

\usepackage{tikz-cd} 

\usepackage{color} 
\usepackage{graphicx}
\usepackage{mathrsfs}

\numberwithin{equation}{section}

\usepackage[nosort,nocompress,noadjust]{cite}

\usepackage[bookmarks=false,hyperfootnotes=false,colorlinks,
    linkcolor={red!60!black},
    citecolor={blue!50!black},
    urlcolor={blue!80!black}]{hyperref}

\renewcommand{\eqref}[1]{\hyperref[#1]{(\ref{#1})}}

\newtheorem{letterthm}{Theorem}

\newtheorem{letterconj}[letterthm]{Conjecture}

\newtheorem*{haagerup-theorem}{Relative Bicentralizer Theorem}

\newtheorem{thm}{Theorem}[section]
\newtheorem{conj}{Conjecture}[section]
\newtheorem{lem}[thm]{Lemma}

\newtheorem{cor}[thm]{Corollary}
\newtheorem{prop}[thm]{Proposition}

\theoremstyle{definition}
\newtheorem{rem}[thm]{Remark}
\newtheorem{remark}[thm]{Remarks}

\newtheorem{exam}{Example}

\newtheorem{df}[thm]{Definition}

\newtheorem{claim}[thm]{Claim}

\newcommand{\R}{\mathbb{R}}
\newcommand{\C}{\mathbb{C}}
\newcommand{\B}{\mathbf{B}}

\newcommand{\Z}{\mathbb{Z}}

\newcommand{\Q}{\mathbb{Q}}
\newcommand{\N}{\mathbb{N}}

\newcommand{\ba}{\mathbf{a}}
\newcommand{\bb}{\mathbf{b}}

\newcommand{\cA}{\mathcal{A}}

\newcommand{\cE}{\mathcal{E}}
\newcommand{\cF}{\mathcal{F}}
\newcommand{\cP}{\mathcal{P}}
\newcommand{\cI}{\mathcal{I}}
\newcommand{\cH}{\mathcal{H}}
\newcommand{\cZ}{\mathcal{Z}}
\newcommand{\cU}{\mathcal{U}}
\newcommand{\cD}{\mathcal{D}}

\newcommand{\cM}{\mathcal{M}}
\newcommand{\cN}{\mathcal{N}}

\newcommand{\cV}{\mathcal{V}}

\newcommand{\Ad}{\operatorname{Ad}}
\newcommand{\id}{\text{\rm id}}

\newcommand{\Aut}{\mathord{\text{\rm Aut}}}

\newcommand{\Inn}{\mathord{\text{\rm Inn}}}
\newcommand{\UCP}{\mathord{\text{\rm UCP}}}
\newcommand{\CCP}{\mathord{\text{\rm CCP}}}

\newcommand{\rL}{\mathord{\text{\rm L}}}
\newcommand{\rB}{\mathord{\text{\rm B}}}

\newcommand{\rb}{\mathord{\text{\rm b}}}
\newcommand{\rC}{\mathord{\text{\rm C}}}

\newcommand{\conv}{\overline{\mathord{\text{\rm conv}}} \,}

\newcommand{\rd}{\mathord{\text{\rm d}}}

\newcommand{\rE}{\mathord{\text{\rm E}}}

\newcommand{\Tr}{\mathord{\text{\rm Tr}}}
\newcommand{\Ball}{\mathord{\text{\rm Ball}}}

\newcommand{\ot}{\otimes}
\newcommand{\ovt}{\mathbin{\overline{\otimes}}}

\newcommand{\op}{\text{\rm op}}
\newcommand{\ri}{\text{\rm i}}
\newcommand{\AC}{\mathord{\text{\rm AC}}}

\newcommand{\om}{\omega}

\newcommand{\I}{{\rm I}}
\newcommand{\II}{{\rm II}}
\newcommand{\III}{{\rm III}}

\begin{document}

\begin{center}
{\boldmath\LARGE\bf Kadison's problem for type $\III$ subfactors\\ and the bicentralizer conjecture}

\bigskip

{\sc by Amine Marrakchi\footnote{UMPA, CNRS ENS de Lyon, Lyon (France). E-mail: amine.marrakchi@ens-lyon.fr}}
\end{center}

\begin{abstract}\noindent

In 1967, Kadison asked ``if $N$ is a subfactor of the factor $M$ for which $N' \cap M$ consists of scalars, will some maximal abelian *-subalgebra of $N$ be a maximal abelian subalgebra of $M$?". Generalizing a theorem of Popa in the type $\II$ case (1981), we solve Kadison's problem for all subfactors with expectation $N \subset M$ where $N$ is either a type $\III_\lambda$ factor with $0 \leq \lambda  < 1$ or a type $\III_1$ factor that satisfies Connes's bicentralizer conjecture. Our solution is based on a new explicit formula for the bicentralizer algebras of arbitrary inclusions. This formula implies a type $\III$ analog of Popa's local quantization principle. We generalize Haaegrup's theorem from 1984 by connecting the relative bicentralizer conjecture to the Dixmier property. Finally, we prove this conjecture for a large class of inclusions and we prove an ergodicity theorem for the bicentralizer flow. 
\end{abstract}

\section*{Introduction}
In 1967, at the Baton Rouge conference \cite{Ka67}, Kadison asked the following question : \emph{``if $N$ is a subfactor of the factor
$M$ for which $N' \cap M$ consists of scalars, will some maximal abelian *-subalgebra of $N$ be a maximal abelian subalgebra of $M$?"}. An inclusion of factors that satisfies $N' \cap M=\C$ is called \emph{irreducible}. This condition is clearly necessary for $N$ to contain a maximal abelian subalgebra of $M$ and Kadison asked if this condition is sufficient. 

In 1981, using an inductive construction, Popa \cite{Po81} solved Kadison's question affirmatively for all irreducible inclusions $N \subset M$ such that $N$ is of type $\II$ and the inclusion $N \subset M$ is \emph{with expectation}, meaning that there exists a faithful normal conditional expectation from $M$ onto $N$. Note that this second condition is automatically satisfied when $M$ is of type $\II_1$. Therefore, Popa's theorem answers Kadison's question affirmatively for \emph{all} irreducible inclusions of type $\II_1$ factors. 

On the other hand, when $M$ is infinite, Ge and Popa showed \cite[Corollary 5.1]{GP98} that the answer to Kadison's question can be \emph{negative} for irreducible inclusions $N \subset M$ \emph{without} expectation, even when $N$ is a $\II_1$ factor and $M$ is a $\II_\infty$ factor! Taking this counter-example into account, Ge and Popa reformulated Kadison's question for arbitrary factors, possibly of type $\III$, by adding this extra assumption that $N$ is with expectation in $M$. See the discussion after \cite[Corollary 5.1]{GP98} (see also \cite[Problem 7.6]{Po21}). The main result of this paper confirms that Ge and Popa's intuition was correct,  by almost solving their question.
\begin{letterthm} \label{main Kadison}
Let $N \subset M$ be an irreducible inclusion of factors with expectation and with separable preduals. Suppose that one of the following condition holds:
\begin{enumerate}[\rm (1)]
\item $N$ is not of type $\III_1$.
\item $N$ is a type $\III_1$ factor with trivial bicentralizer.
\end{enumerate}
Then there exists a maximal abelian subalgebra $A \subset N$ that is also maximal abelian in $M$.
\end{letterthm}
We note that Theorem \ref{main Kadison} is new even when $N$ is amenable of type $\III_\lambda$ with $0 \leq \lambda <1$. But of course, the type $\III_1$ case is the most difficult one, as always, when one tries to generalize results from the type $\II_1$ world to the type $\III$ world. A famous conjecture of Connes, the so-called \emph{Connes's Bicentralizer Problem}, claims that \emph{every} type $\III_1$ factor has trivial bicentralizer (we will recall the definitions and the precise statement below). This conjecture has been verified for several families of type $\III_1$ factors such as amenable factors \cite{Ha85}, factors with a Cartan subalgebra, free products \cite{HU15}, semi-solid factors \cite{HI15}, $q$-deformed Araki-Woods factors \cite{HI20}, \cite{Bi24}. In this paper, we will also add to this list all type $\III_1$ factors that are tensor products of two type $\III_1$ factors (see Theorem \ref{intro conrete class} below). Therefore, Theorem \ref{main Kadison} solves Kadison's problem for a large class of type $\III$ factors, and conjecturally, for all of them.

We emphasize the fact that, unlike Popa's solution in the type $\II$ case, the maximal abelian subalgebra we construct in Theorem \ref{main Kadison} may \emph{not} be with expectation in $N$. Indeed, this can't be achieved in general because of \emph{modular theoretic obstructions} as was already observed in \cite{AHHM18}. The following conjecture gives an effective if and only if criterion for the existence of such a maximal abelian subalgebra with expectation. In this paper, we will also prove this conjecture for a large class of inclusions.

\begin{letterconj} \label{conj Kadison expectation}
Let $N \subset M$ be an irreducible inclusion of factors with expectation and with separable preduals. Then the following are equivalent :
\begin{enumerate}[\rm (1)]
\item There exists a maximal abelian subalgebra \emph{with expectation} $A \subset N$ that is also maximal abelian in $M$.
\item $N' \cap c(M) \subset c(N)$.
\end{enumerate}
\end{letterconj}
It is easy to check that (1) $\Rightarrow$ (2). So conjecture \ref{conj Kadison expectation} is all about the converse implication. Here, $c(M)$ denotes the \emph{core} of $M$ \cite{FT01}. For every faithful normal state $\varphi$ on $M$, the core $c(M)$ is canonically isomorphic to $M \rtimes_{\sigma^\varphi} \R$ where $\sigma^\varphi : \R \curvearrowright M$ is the modular automorphism group of $\varphi$. A fundamental invariant of $M$ is the so-called \emph{flow of weights} $M' \cap c(M)=\cZ(c(M))$. Similarly, the relative commutant $N' \cap c(M)$ is an important invariant of the inclusion $N \subset M$ that we call the \emph{relative flow weights}. It encodes information about the position of $N$ inside $M$ with respect to modular theory. 

When $N$ is with expectation in $M$, we have a natural inclusion $c(N) \subset c(M)$. The condition $N' \cap c(M) \subset c(N)$ means that $N' \cap c(M)=\cZ(c(N))$ and it implies in particular that $\cZ(c(M)) \subset \cZ(c(N))$, i.e.\ the flow of weights of $N$ is an extension of the flow of weights of $M$. This shows that condition (2) in Conjecture \ref{conj Kadison expectation} is quite restrictive. Here are few examples.

\begin{exam}
Let $N \subset M$ be an irreducible subfactor with expectation. If $N$ is of type $\II$, then the condition $N' \cap c(M) \subset c(N)$ is automatically satisfied, and indeed, Popa's theorem shows that we can find a maximal abelian subalgebra with expectation $A \subset N$ that is maximal abelian in $M$.
\end{exam}
\begin{exam}
Let $N$ be a factor and $\alpha : G \curvearrowright N$ an action of a discrete group $G$ on $N$. Let $M=N \rtimes_\alpha G$. The inclusion $N \subset M$ is irreducible if and only if $\alpha$ is \emph{outer,} i.e.\ $\alpha_g \notin \Inn(N)$ for all $g \in G \setminus \{1\}$. However, we will have $N' \cap c(M) \subset c(N)$ if and only if the extended action $c(\alpha) : G \curvearrowright c(N)$ is outer. For this kind of inclusions, we will show that Conjecture \ref{conj Kadison expectation} holds when $N$ is not of type $\III_1$ or when $N$ is of type $\III_1$ with trivial bicentralizer.
\end{exam}
\begin{exam}
Let $N \subset M$ be an irreducible subfactor with expectation. Let $R_\infty$ be the Araki-Woods factor of type $\III_1$ (or any type $\III_1$ factor with trivial bicentralizer) and let $\cN=N \ovt R_\infty$ and $\cM=M \ovt R_\infty$. In this stabilized inclusion $\cN\subset \cM$, all modular obstructions disapear: we always have $\cN' \cap c(\cM)=\C$. In accordance with Conjecture \ref{conj Kadison expectation}, we will show, without making any assumption on $N$ or $M$, that we can always find a maximal abelian subalgebra with expectation $A \subset \cN$ that is maximal abelian in $\cM$.
\end{exam}

In order to prove Theorem \ref{main Kadison} or Conjecture \ref{conj Kadison expectation}, a fundamental tool is Connes's bicentralizer algebra and its relative version for inclusions. Let us recall the definition of these algebras. 

Let $N$ be a von Neumann algebra with a faithful normal state $\varphi \in N_*$. Following \cite{Co80, Ha85}, we define the {\em asymptotic centralizer} of $\varphi$ by 
$$\AC(N, \varphi) = \left \{ (a_{n})_{n} \in \ell^{\infty}(\N, N) \mid \lim_{n} \| \varphi(a_{n} \cdot) - \varphi(\cdot a_{n})\| = 0\right\}$$ 
and the \emph{bicentralizer} of $M$ with respect to $\varphi$ by
$$\rB(N, \varphi) = \left\{ x \in N \mid x a_{n} - a_{n} x \to 0 \text{ strongly}, \text{ for all } (a_{n})_{n} \in \AC(N, \varphi)\right\}.$$
Connes classified all amenable factors that are not of type $\III_1$ in his groundbreaking paper \cite{Co75b}. It is during his efforts to tackle the remaining type $\III_1$ case that he encountered the bicentralizer problem. Indeed, Connes proved  that an amenable type $\III_1$ factor $N$ is isomorphic to the unique Araki-Woods factor of type $\III_1$ if and only if $\rB(N,\varphi)=\C$ for some faithful state $\varphi \in N_*$ (see \cite{Co85}). Moreover, he conjectured that we always have $\rB(N,\varphi)=\C$ for \emph{every} type $\III_1$ factor $N$ and \emph{every} state $\varphi \in N_*$. This conjecture, known as \emph{Connes' Bicentralizer Problem}, was famously solved by Haagerup in the case where $N$ is amenable \cite{Ha85}. This result, together with Connes's work, settled the classification problem for amenable factors of all types. However, Connes's Bicentralizer Problem is still, to this day, widely open when $N$ is not amenable, despite some recent progress. When $N$ is not of type $\III_1$, the bicentralizer $\rB(N,\varphi)$ is usually nontrivial but it has attracted very little attention so far except for a recent work of Okayasu \cite{Ok21}. 

When $N \subset M$ is an inclusion of von Neumann algebras with expectation, one can define the bicentralizer of the inclusion $N \subset M$ with respect to $\varphi$ by 
\begin{equation} \label{eq intro bicentralizer}\rB(N \subset M, \varphi) = \left\{ x \in M \mid x a_{n} - a_{n} x \to 0 \text{ strongly}, \text{ for all } (a_{n})_{n} \in \AC(N, \varphi)\right\}.
\end{equation}
This relative version of the bicentralizer algebra appears both implicitly in \cite{Po95} and explicitly in \cite{Mas03}, playing an important role in the classification of type $\III$ subfactors. Recently, in \cite{AHHM18}, the relative bicentralizer algebra was studied in a more systematic way and was used to solve Kadison's problem for some families of type $\III$ subfactors. But the relative bicentralizer algebra could only be computed for a very specific class of inclusions, see \cite{Mas20}.

In the present paper, we propose, for an arbitrary inclusion with expectation $N \subset M$ with an arbitrary faithful state $\varphi \in N_*$, a formula to compute the bicentralizer algebra $\rB(N \subset M,\varphi)$ that relates it directly to the relative flow of weights $N' \cap c(M)$. Actually, it is easier and more natural to give a formula for $\rB(N \subset c(M),\varphi)$ rather then $\rB(N \subset M,\varphi)$ itself. Here, $\rB(N \subset c(M),\varphi)$ is a subalgebra of $c(M)$ defined exactly as in \ref{eq intro bicentralizer}.
\begin{letterconj}[Relative Bicentralizer Conjecture] \label{conj relative bicentralizer}
Let $N \subset M$ be an inclusion of von Neumann algebras with expectation. Let $\varphi \in N_*$ be a faithful normal state. Then we have
\begin{equation} \label{eq intro bicentralizer conj}
\rB(N \subset c(M),\varphi)= \{ \varphi^{\ri t} \mid t \in \R \}'' \vee (N' \cap c(M)).
\end{equation}
where $(\varphi^{\ri t})_{t \in \R}$ is the one-parameter group of unitaries in $c(N)=N \rtimes_{\sigma^\varphi} \R$ that implements the modular action $\sigma^\varphi : \R \curvearrowright N$.
\end{letterconj}

A large section of this paper (Section \ref{section algebraic}) is devoted to the study of the right hand side of \ref{eq intro bicentralizer conj}, which we call the \emph{algebraic bicentralizer} of $\varphi$. A particularly useful property of this algebraic bicentralizer is that it can be defined more generally for abitrary \emph{weights} and not only states, thus allowing us to use the \emph{integrable weight} machinery of Connes and Takesaki \cite{CT76}.

By the relative commutant theorem \cite{CT76}, we have $N' \cap c(N)=\cZ(c(N))$ for every von Neumann algebra $N$. Thus Conjecture \ref{conj relative bicentralizer} would imply, in the case $N=M$, that $\rB(N,\varphi)$ is always \emph{abelian}, a property that one would not expect a priori from the analytic definition of $\rB(N,\varphi)$. In fact, Conjecture \ref{conj relative bicentralizer} tells us that $\rB(N,\varphi)$ is precisely the \emph{strong center} of $\varphi$ that was already introduced, in a different context, by Falcone and Takesaki in \cite{FT01}. In particular, if $N=M$ is a type $\III_1$ factor (recall that this is equivalent to $N' \cap c(N)=\C$), then Conjecture \ref{conj relative bicentralizer} reduces to Connes's bicentralizer problem $\rB(N,\varphi)=\C$.

We now move to the central result of this paper. Recall that an inclusion of von Neumann algebras $N \subset M$ has the \emph{weak Dixmier property} if
$$\forall x \in M, \quad \conv \{ uxu^* \mid u \in \mathcal{U}(N) \} \cap (N' \cap M)  \neq \emptyset$$
where $\conv$ denotes the weak*-closed convex hull. In his celebrated work on the bicentralizer problem \cite{Ha85}, Haagerup proved that a type $\III_1$ factor has trivial bicentralizer if and only if the inclusion $N_\psi \subset N$ has the weak Dixmier property for some dominant weight $\psi$ on $N$. Using Takesaki's duality for crossed products, this is easily seen to be equivalent to the weak Dixmier property for the inclusion $N \subset c(N)$. In the same paper,  by using an inductive construction similar to Popa's construction from \cite{Po81}, Haagerup proved that a type $\III_1$ factor with separable predual has trivial bicentralizer if and only if it contains a maximal abelian subalgebra with expectation.

The following theorem is therefore a relative version of Haagerup's theorem.

\begin{letterthm} \label{intro bicentralizer dixmier}
Let $N \subset M$ be an inclusion of countably decomposable von Neumann algebras with expectation. Then the following are equivalent:
\begin{enumerate}
\item The inclusion $N \subset M$ satisfies the relative bicentralizer conjecture with respect to some faithful state $\varphi \in N_*$.
\item The inclusion $N \subset M$ satisfies the relative bicentralizer conjecture with respect to every faithful state $\varphi \in N_*$.
\item The inclusion $N \subset c(M)$ has the weak Dixmier property :
$$\forall x \in c(M), \quad \conv \{ uxu^* \mid u \in \mathcal{U}(N) \} \cap (N' \cap c(M))  \neq \emptyset.$$
\end{enumerate}
If $M$ has separable predual, these properties are also equivalent to the following :
\begin{enumerate}
\setcounter{enumi}{3}
\item There exists an amenable subalgebra with expectation $P \subset N$ such that $$P' \cap c(M)=N' \cap c(M).$$
\item There exists a maximal abelian subalgebra $A \subset N$ such that $$A' \cap c(M) =A \vee (N' \cap c(M)).$$
\item There exists a maximal abelian subalgebra with expectation $A \subset N$ such that $$A' \cap c(M) =c(A) \vee (N' \cap c(M))$$
\end{enumerate}
\end{letterthm}
Note that relative versions of Haagerup's theorem were previously obtained for finite index inclusions \cite{Po95} and quasiregular inclusions \cite{AHHM18} by adapting Haagerup's original proof. The proof of Theorem \ref{intro bicentralizer dixmier} is different and does not involve any maximality argument.

In Theorem \ref{intro bicentralizer dixmier}, item (4) shows that $N$ contains a large amenable subalgebra with expectation $P$ that has the same relative flow of weights as $N$ inside $M$. 

Item (5) solves in particular Kadison's problem for the inclusion $N \subset M$ when $N$ is an irreducible subfactor of $M$. Item (5) also solves Kadison's problem for the inclusion $N \subset c(N)$ when $N$ is a type $\III_1$ factor with trivial bicentralizer.

Item (6) solves Conjecture \ref{conj Kadison expectation}. Indeed, if we take $A$ as in item (6) and we assume that $N' \cap c(M) \subset c(N)$, then we obtain $A' \cap c(M)=A' \cap c(N)$, hence $A' \cap M=A' \cap N=A$.  

Items (4), (5) and (6) are proved in the last section of this paper (Section \ref{section masa}) by using Popa's \emph{local quantization principle} \cite[A.1.2]{Po95}. In fact, we show that the relative bicentralizer conjecture is equivalent to a type $\III$ analog of Popa's local quantization principle.

Our next theorem provides a large class of inclusions for which we can verify this conjecture.
\begin{letterthm} \label{intro conrete class}
Let $N \subset M$ be an inclusion of countably decomposable von Neumann algebras with expectation. Then $N \subset M$ satisfies the relative bicentralizer conjecture in the following cases.
\begin{enumerate}[\rm (1)]
\item $N$ is amenable.
\item $N$ has no type $\III_1$ summand.
\item $N=N_1 \ovt N_2$ where each $N_i$ is of type $\III_1$.
\item $N$ has a Cartan subalgebra.
\item $N$ satisfies the bicentralizer conjecture and $N \subset M$ is quasiregular (e.g.\ $N \subset M$ has finite index or $M$ is a crossed product of $N$ by a discrete group).
\item $N$ is generated by a family of globally $\sigma^\varphi$-invariant subalgebras $(N_i)_{i \in I}$ for some faithful state $\varphi \in N_*$ and the inclusions $N_i \subset M$ satisfy the relative bicentralizer conjecture for every $i \in I$.
\end{enumerate}
\end{letterthm}
In Theorem \ref{intro conrete class}, item (3) is new even when $M=N$. It says that the tensor product of two type $\III_1$ factors always has trivial bicentralizer (compare this with \cite[Theorem D]{Ma18}). It also applies to Example 3. Item (5) applies to Example 2 below Conjecture \ref{conj Kadison expectation}. Yusuke Isono also informed us \cite{Is23} that he was able to apply our results, and specifically item (5), in order to solve the Haagerup-St\o rmer conjecture \cite{HS88} for all type $\III_1$ factors with trivial bicentralizer : any pointwise inner
automorphism of such a factor is the composition of an inner automorphism and a modular automorphism.

We saw that Theorem \ref{intro bicentralizer dixmier} solves Kadison's question, but only for those subfactors that satisfy the relative bicentralizer conjecture. In order to prove Theorem \ref{main Kadison} when $N$ is an arbitrary type $\III_1$ factor with trivial bicentralizer, we make a reduction of Theorem \ref{main Kadison} to the case where the subfactor $N$ is amenable (for which the relative bicentralizer conjecture holds by Theorem \ref{intro conrete class}). This reduction is achieved by the following theorem.

\begin{letterthm} \label{intro ergodic}
Let $N \subset M$ an irreducible subfactor with expectation. Suppose that $N$ is of type $\III_1$ and has trivial bicentralizer. Let $\varphi \in N_*$ be a faithful state. Then the relative bicentralizer flow $\beta^\varphi : \R^*_+ \curvearrowright \rB(N \subset M,\varphi)$ is ergodic. In particular, if $M_*$ is separable, there exists an amenable subfactor with expectation $P \subset N$ such that $P$ is irreducible in $M$.
\end{letterthm}

We refer to \cite{AHHM18} (see also Section \ref{section analytic}) for the definition of the bicentralizer flow $\beta^\varphi : \R^*_+ \curvearrowright \rB(N \subset M,\varphi)$. This canonical flow is only defined when $N$ is of type $\III_1$. When $M_*$ is separable, \cite[Theorem C]{AHHM18} shows that $\beta^\varphi$ is ergodic if and only if $N$ contains an amenable subfactor with expectation $P$ that is irreducible in $M$. This is much weaker than what we would get from item (4) in Theorem \ref{intro bicentralizer dixmier}, but this is sufficient to prove Theorem \ref{main Kadison}.

The proof of Theorem \ref{intro ergodic} is similar to the proof  of the implication $(3) \Rightarrow (2)$ in Theorem \ref{intro bicentralizer dixmier}. It is based on a key technical novelty that we call \emph{ultrapower implementation of binormal states} and which is introduced in Section \ref{section ultrapower binormal}. In Section \ref{section II1} we give applications of this technique to $\II_1$ factors, including a new bimodule characterization of Ozawa's property (AO) for von Neumann algebra. After this interlude, we study the weak Dixmier property in Section \ref{section dixmier}. In Section \ref{section ucp map} we introduce a class of well-behaved ucp maps that are compatible with modular theory in the sense that they have a unique extension to the core.
Next, in Section \ref{section algebraic}, we introduce the \emph{algebraic bicentralizer}, we establish its fundamental properties and we show how to compute it explicitely for lacunary weights and integrable weights. In Section \ref{section analytic}, we recall the definition of the \emph{analytic} bicentralizer algebra and its basic properties and we establish a key lemma based on the ultrapower implementation of binormal states. In Section \ref{section bicentralizer conjecture}, we combine all our efforts from the previous sections to obtain our main results on the bicentralizer conjecture, including the first three items of Theorem  \ref{intro bicentralizer dixmier} as well as Theorem \ref{intro conrete class} and Theorem \ref{intro ergodic}. Section \ref{section quasiregular} is dedicated to the specific case of finite index and quasiregular inclusions. Finally, in Section \ref{section masa}, we study Popa's local quantization property and we prove the three last items of Theorem \ref{intro bicentralizer dixmier}.

\bigskip

{\bf Acknowledgment.} We thank Cyril Houdayer, Sorin Popa and Stefaan Vaes for their useful comments. We are also very grateful to the anonymous referee for his numerous comments and corrections.

\tableofcontents

\pagebreak

\section{Preliminaries}
\subsection{Weights and operator valued weights}
We use the theory of operator valued weights and the notations from \cite{Ha77a} and \cite{Ha77b}. Let $M$ be a von Neumann algebra. We denote by $\cP(M)$ the set of all normal faithful semifinite weights on $M$ and by $\cE(M) \subset \cP(M)$ the set of all normal faithful states. We say that $M$ is \emph{countably decomposable} (or \emph{$\sigma$-finite}) if $\cE(M)$ is nonempty. The predual of $M$, denoted by $M_*$, is separable if and only if $M$ is countably decomposable and countably generated.

 If $N$ is a von Neumann subalgebra of $M$, we denote by $\cP(M,N)$ the set of all normal faithful semifinite operator valued weights from $M$ onto $N$ and by $\cE(M,N) \subset \cP(M,N)$ the set of all normal faithful conditional expectations. If $\tau_N \in \cP(N)$ and $\tau_M \in \cP(M)$, then there exists a unique $T \in \cP(M,N)$ such that $\tau_M=\tau_N \circ T$.

We say that the subalgebra $N$ (or that the inclusion $N \subset M$) is \emph{with expectation} if $\cE(M,N)$ is nonempty. We say that $N$ is \emph{with expectations} in $M$ if there exists a faithful family of normal conditional expectations from $M$ onto $N$. By \cite[Theorem 6.6]{Ha77b}, if $T \in \cP(M,N)$, then $N$ is with expectations in $M$ if and only if $T|_{N' \cap M}$ is semifinite, in which case $T|_{N' \cap M} \in \cP(N' \cap M, \cZ(N))$. The inclusion $N \subset M$ is with expectation if and only if it is with expectations and $N' \cap M$ is countably decomposable.

\subsection{Amalgamated tensor products}
The reference for this section is \cite{SZ99}.

Let $M$ be a von Neumann algebra and $N_1$ and $N_2$ two commuting von Neumann subalgebras of $M$ such that $M=N_1 \vee N_2$. Then $A:=N_1 \cap N_2=\cZ(N_1) \cap \cZ(N_2) \subset \cZ(M)$. We say that $M$ is the \emph{amalgamated tensor product} of $N_1$ and $N_2$ if there exists two Hilbert spaces $H_1$ and $H_2$ and a normal embedding 
$\pi : M \hookrightarrow \B(H_1) \ovt A \ovt \B(H_2)$ such that :
\begin{itemize}
\item $\pi(N_1) \subset \B(H_1) \ovt A$ and $\pi(N_2) \subset A \ovt \B(H_2)$.
\item  $\pi(x)=1 \otimes x \otimes 1$ for every $x \in N_1 \cap N_2$.
\end{itemize}

The following theorem garanties the existence and uniqueness of amalgamted tensor products.
\begin{thm}[{\cite[Section 5.7]{SZ99} }]  \label{ATP exist}
Let $A$ be an abelian von Neumann algebra and $N_1$ and $N_2$ be two Neumann algebras with two normal embeddings $\iota_i : A \hookrightarrow \cZ(N_i)$. 
\begin{enumerate}
\item There exists a von Neumann algebra $M$ and two normal embeddings $s_i  : N_i \hookrightarrow M$ such that $s_1 \circ \iota_1=s_2 \circ \iota_2$ and
$M$ is an amalgamated tensor product of $s_1(N_1)$ and $s_2(N_2)$.
\item If $(\tilde{M},\tilde{s}_1,\tilde{s}_2)$ is another triple satisfying the conditions of item 1, then there exists a unique isomorphism $\theta  : M \rightarrow \tilde{M}$ such that $\theta \circ s_i=\tilde{s}_i$ for $i=1,2$.
\end{enumerate}
\end{thm}
If $(M,s_1,s_2)$ is the amalgamated tensor product, we will use the notation $M=N_1 \ovt_A N_2$ when $\iota_1$ and $\iota_2$ are obvious from the context. In that case we will view $A$ as a subalgebra of $N_1 \ovt_A N_2$ and the $A$-bilinear map $N_1 \times N_2 : (x,y) \mapsto s_1(x)s_2(y)$ will be denoted $(x,y) \mapsto x \otimes_A y$. Then we have $(ax) \otimes_A y= x \otimes_A (ay)=a( x \otimes_A y)$ for all $a \in A$.

The second theorem gives us a useful characterization of amalgamated tensor products.

\begin{thm}[{\cite[Section 5.5]{SZ99} }] 
Let $M$ be a von Neumann algebra and $N_1$ and $N_2$ two commuting von Neumann algebras such that $M=N_1 \vee N_2$. Then $M$ is the amalgamated tensor product of $N_1$ and $N_2$ if and only there exists a faithful family of normal conditional expectations $E : M \rightarrow N_1 \cap N_2$ such that $E(x_1x_2)=E(x_1)E(x_2)$ for all $x_i \in N_i, \: i=1,2$.
\end{thm}

\begin{cor} \label{corollary expectation tensor product}
Let $N \subset M$ be an inclusion of von Neumann algebras with expectations. Then 
$$N \vee (N' \cap M) \cong N \ovt_{\cZ(N)} (N' \cap M).$$
\end{cor}

\subsection{Crossed products by locally compact groups and split extensions} \label{prelim crossed product}
A general reference for this section is \cite[Chapter X]{Ta03}.

Let $G$ be a locally compact group. A $G$-von Neumann algebra is a pair $(M,\alpha)$ where $M$ is a von Neumann algebra and $\alpha : G \curvearrowright M$ is a continuous action, i.e.\ a group homomorphism $\alpha : G \rightarrow \Aut(M)$ such that $g \mapsto \phi \circ \alpha_g$ is norm continuous for every $\phi \in M_*$. We denote by the same letter $\alpha$ the natural \emph{co-action} map $\alpha : M \rightarrow M \ovt \rL^\infty(G) \cong \rL^\infty(G,M)$ given by
$$\alpha : x \mapsto  (g \mapsto \alpha_{g^{-1}}(x))$$
This is a $G$-equivariant embedding when $M \ovt \rL^\infty(G)$ is equipped with the left translation action on the second coordinate. The \emph{fixed point algebra} of the action $\alpha$ is $$M^\alpha :=\{ x \in M \mid \forall g \in G, \: \alpha_g(x)=x\}.$$ The \emph{Haar operator valued weight} from $M$ to $M^\alpha$ associated to a left Haar measure $\mu$ is defined by
$$ I^\alpha(x)=\int_G \alpha_g(x) \: \rd\mu(g), \quad x \in M_+.$$
We say that the action $\alpha$ is \emph{integrable} if $I^\alpha$ is semifinite. In that case, we have $I^\alpha \in \cP(M,M^\alpha)$.

A \emph{split $G$-extension} is a triple $(M,N,u)$ where $M$ is a von Neumann algebra, $N$ is a subalgebra of $M$ and $u : G \rightarrow \cU(M)$ is a continuous representation of $G$ such that $u_gNu_g^*=N$ for all $g \in G$ and $N \cup \{ u_g \mid g \in G \}$ generates a dense subalgebra of $M$. If $(M,N,u)$ is a $G$-extension then $(N,\Ad(u)|_N)$ is a $G$-von Neumann algebra.

Given a $G$-von Neumann algebra $(N,\alpha)$, there is a natural split $G$-extension $(M,N,u)$ such that $\alpha=\Ad(u)|_N$. This is the \emph{crossed product} construction $(N \rtimes_\alpha G, N, u)$. It is characterized concretely by the property that the $G$-equivariant co-action map $\alpha : N \rightarrow N \ovt \rL^\infty(G)$ extends to an embedding of $G$-extensions $$\alpha \rtimes G : (N \rtimes_\alpha G,N,u) \rightarrow (N \ovt \B(\rL^2(G)), N \ovt \rL^\infty(G), 1\otimes \lambda_G)$$
where $\lambda_G : G \rightarrow \cU(\rL^2(G))$ is the left regular representation.

The following theorem is essential.
\begin{thm}[Digernes-Takesaki]
Let $G$ be a locally compact group and $(N,\alpha)$ a $G$-von Neumann algebra. Then the image of $\alpha \rtimes G$ is exactly the fixed point subalgebra of $N \ovt \B(\rL^2(G))$ under the diagonal action $\alpha \otimes \Ad(\rho_G) : G \curvearrowright N \ovt \B(\rL^2(G))$ where $\rho_G$ is the right regular representation.
\end{thm}

Unfortunately, not every split $G$-extension $(M,N,u)$ is a crossed product extension\footnote{However, it is interesting to notice that the crossed product construction is the unique \emph{natural} extension. This means that the functor $(M,N,u) \rightarrow (N,\Ad(u)|_N)$ from the category of split $G$-extensions to the category of $G$-von Neumann algebras has a unique right inverse, up to natural isomorphism. Indeed, observe first that a right inverse functor is determined by the image of the $G$-von Neumann algebra $\rL^\infty(G)$ thanks to the co-action maps. Then use the Stone-von Neumann-Mackey theorem \cite{Mac49} to show that the image of $\rL^\infty(G)$ by this functor must be $(\B(\rL^2(G)), \rL^\infty(G),\lambda_G)$.}, even when the natural map from the algebraic crossed product into $M$ is injective. We will encounter such examples in our study of the algebraic bicentralizer. When $G$ is abelian, we have the following useful characterization of crossed product extensions that we will use many times.
\begin{thm}[Takesaki] \label{dual action crossed product}
Let $(M,N,u)$ be a split $G$-extension for some locally compact abelian group $G$. Let $\widehat{G}$ be the Pontryagin dual of $G$. Then $(M,N,u)$ is a crossed product extension if and only if there exists an action $\beta : \widehat{G} \curvearrowright M$ such that $\beta_p(u_g)=\langle g,p \rangle u_g$ for all $(g,p) \in G \times \widehat{G}$ and $\beta_p(x)=x$ for all $x \in \N$ and $p \in \widehat{G}$. In that case, we have $N=M^\beta$.

Conversely, if $M$ is a von Neumann algebra with a unitary representation $u : G \rightarrow \cU(M)$ of $G$ and there exists an action $\beta : \widehat{G} \curvearrowright M$ such that $\beta_p(u_g)=\langle g,p \rangle u_g$ for all $(g,p) \in G \times \widehat{G}$, then $(M,M^\beta,u)$ is a crossed product extension.
\end{thm}
If $(N,\alpha)$ is a $G$-von Neumann algebra, the action $\beta : \widehat{G} \curvearrowright N \rtimes_\alpha G$ of the theorem is called the \emph{dual action} of $\alpha$ and denoted $\beta=\widehat{\alpha}$.

\begin{prop} \label{criterion expectation crossed product}
Let $(M,N,u)$ be a split $G$-extension for some locally compact group $G$ (not necessarily abelian). Let $M_0 \subset M$ and $N_0 \subset N$ be von Neumann subalgebras such that :
\begin{itemize}
\item $(M_0,N_0,u)$ is a crossed product extension.
\item There exists a faithful family of normal conditional expectations $E$ from $M$ to $M_0$ such that $E(N)=N_0$.
\end{itemize}
Then $(M,N,u)$ is a crossed product extension.
\end{prop}
\begin{proof}
Let $\alpha : G \curvearrowright N$ be the action induced by $u$. Let $\Sigma$ be the set of all normal conditional expectations from $M$ to $M_0$ that send $N$ on $N_0$. Note that for each $E \in \Sigma$, the normal conditional expctation $E|_N$ from $N$ to $N_0$ is $G$-equivariant :
$$ E(\alpha_g(x))=E(u_gxu_g^*)=u_gE(x)u_g^*=\alpha_g(E(x))$$
for all $x \in N$ and $g \in G$.
Therefore $E$ induces a normal conditional expectation $\tilde{E} : N \rtimes_\alpha G \rightarrow N_0 \rtimes_{\alpha} G$. Since $\Sigma$ is a faithful family, then so is $\{ \tilde{E} \mid E \in \Sigma\}$.

Let $A$ be the algebraic crossed product of $N$ by the action $\alpha$ and $A_0 \subset A$ the algebraic crossed product of $N_0$ by the same action. Let $\pi : A \rightarrow M$ and $\rho : A \rightarrow N \rtimes_\alpha G$ be the natural *-morphisms. To get the desired conclusion, we have to show that for ever net $(x_i)_{i \in I}$ in $A$ we have $\lim_i \pi(x_i) =0$ in the strong topology if and only if $\lim_i \rho(x_i)=0$ in the strong topology. Indeed, if we prove this, then $\pi$ and $\rho$ extend to isomorphisms defined on the same topological completion of $A$, and $\pi \circ \rho^{-1}$ will then define an isomorphism from $N \rtimes_{\alpha} G$ onto $M$.

Let $\theta$ be the isomorphism from $M_0$ to $N_0 \rtimes_\alpha G$ such that $\theta \circ \pi|_{A_0}=\rho|_{A_0}$. Observe that $\theta \circ E \circ \pi = \tilde{E} \circ \rho$ for every $E \in \Sigma$.

Now we see that
\begin{align*}
&\pi(x_i) \to 0 \text{ strongly} \\
\Leftrightarrow \quad  &\pi(x_i^*x_i) \to 0 \text{ weakly} \\
\Leftrightarrow \quad & E(\pi(x_i^*x_i)) \to 0 \text{ weakly for every } E \in \Sigma \\
\Leftrightarrow \quad & \theta(E(\pi(x_i^*x_i))) \to 0 \text{ weakly for every } E \in \Sigma \\
\Leftrightarrow \quad & \tilde{E}(\rho(x_i^*x_i))) \to 0 \text{ weakly for every } E \in \Sigma \\
\Leftrightarrow \quad  &\rho(x_i^*x_i) \to 0 \text{ weakly} \\
\Leftrightarrow \quad  &\rho(x_i) \to 0 \text{ strongly}.
\end{align*}
\end{proof}

\subsection{The modular theory} \label{prelim modular theory}
A \emph{trace scaling flow} is a triple $(\cM, \tau, \theta)$ where $\cM$ is a von Neumann algebra, $\tau \in \cP(\cM)$ is a semifinite trace and $\theta : \R^*_+ \curvearrowright \cM$ is an action that scales the trace, i.e.\ $\tau \circ \theta_{\lambda}=\frac{1}{\lambda} \tau$ for all $\lambda \in \R^*_+$.

Let $\mathcal{C}$ be the category whose objects are trace scaling flows and morphisms are normal trace-preserving $\R^*_+$-equivariant embeddings. 

Let $\mathcal{D}$ be the category whose objects are von Neumann algebras and morphisms are normal embeddings with expectation, i.e.\  pairs $(i,E)$ where $i : N \hookrightarrow M$ is a normal embedding and $E \in \cE(M,i(N))$. 

We define the \emph{fixed point functor} $F$ from $\mathcal{C}$ to $\mathcal{D}$ as follows. If $(\cM, \tau, \theta)$ is a trace scaling flow, we let $F(\cM,\tau,\theta)=\cM^\theta$. If $(\cN, \tau, \theta)$ and $(\cM, \tau', \theta')$ are two trace scaling flows and $\iota : \cN \hookrightarrow \cM$ is a normal trace-preserving $\R^*_+$-equivariant embedding, then there exists a unique trace-preserving conditional expectaion $E \in \cE(\cM,\iota(\cN))$. We let $F(\iota)=(\iota|_{\cN^\theta}, E|_{\cM^{\theta'}})$. 

The following theorem is the great achievement of the modular theory developped by Tomita, Takesaki and Connes.
\begin{thm}
The fixed point functor $F : \mathcal{C} \rightarrow \mathcal{D}$ is an equivalence of categories.
\end{thm}
We select an inverse functor $c : \mathcal{D} \rightarrow \mathcal{C}$. Such an inverse functor $c$ exists by abstract nonsense, but it can also be constructed naturally as explained in \cite{FT01}. If $M$ is a von Neumann algebra, the trace scaling flow associated to it by the functor $c$ is denoted $(c(M), \tau_M,\theta_M)$ or simply $(c(M), \tau,\theta)$ when this is not ambiguous. It is called the \emph{noncommutative flow of weights} of $M$. The von Neumann algebra $c(M)$ is called the \emph{core} of $M$ and we have a natural identification $M=c(M)^\theta$. Here is a list of facts that we will use constantly throughout this article.
\begin{itemize}
\item The action $\theta : \R^*_+ \curvearrowright c(M)$ is integrable and the Haar operator valued weight 
$I^\theta = \int_{\R^*_+} \theta_\lambda \:\frac{\rd \lambda}{\lambda}$
belongs to $\cP(c(M),M)$. There is an order preserving bijection from $\cP(M)$ to $\cP(c(M))^\theta=\{ \psi \in \cP(c(M)) \mid \psi \text{ is } \theta\text{-invariant} \}$ that sends $\varphi \in \cP(M)$ to $\widehat{\varphi}= \varphi \circ I^\theta$. 
\item Following \cite{FT01}, it is convenient to view each $\varphi \in \cP(M)$ as a positive operator affiliated with $c(M)$ by identifying it with the Radon-Nikodym derivative $\frac{\rd \widehat{\varphi}}{\rd \tau}$. Then, by definition, we have $\widehat{\varphi}(x)=\tau(\varphi \cdot x)$ for every $x \in c(M)_+$, where we use the notation of \cite[Proposition 1.11]{Ha77a}. Note that $\theta_\lambda(\varphi)=\lambda^{-1} \varphi$ for every $\lambda \in \R^*_+$ and $(\varphi^{\ri t})_{t \in \R}$ is a one-parameter group of unitaries in $c(M)$ that satisfy $\theta_{\lambda}(\varphi^{\ri t})=\lambda^{- \ri t} \varphi^{\ri t}$ for every $(\lambda,t) \in \R^*_+ \times \R$.
\item Take $\varphi, \psi \in \cP(M)$. For every $x \in M$ and $t \in \R$, we have $\psi^{\ri t} x \varphi^{-\ri t} \in M$ and this defines a one-parameter group of isometries $\sigma^{\psi,\varphi} : \R \curvearrowright M$ given by $\sigma_t^{\psi,\varphi} : x \mapsto \psi^{\ri t} x \varphi^{-\ri t}$. Moreover, $\sigma^\varphi = \sigma^{\varphi,\varphi}$ is a one-parameter group of \emph{automorphisms} of $M$. This is the modular automorphism group of Tomita-Takesaki. Connes's Radon-Nikodym cocycle is given by $t \mapsto u_t=\sigma^{\psi,\varphi}_t(1)=\psi^{\ri t} \varphi^{-\ri t}$. We define the fixed points algebra $M_{\psi,\varphi}:=\{ x \in M \mid \forall t \in \R, \; \sigma_t^{\psi,\varphi}(x)=x\}$. Note that $M_\varphi :=M_{\varphi,\varphi}$ is the centralizer of $\varphi$. We also define the \emph{spectral subspaces} $M(\sigma^{\psi,\varphi}, [-\varepsilon, \varepsilon])$ for $\varepsilon >0$ as the set of all elements $x \in M$ such that the map $t \mapsto \sigma_t^{\psi,\varphi}(x)$ extends to an entire analytic function on $\C$ satisfying
$$ \forall z \in \C, \; \| \sigma_z^{\psi,\varphi}(x)\| \leq e^{\varepsilon |\mathrm{Im}(z)|} \|x\|.$$
\item The triple $(c(M), M, (\varphi^{\ri t})_{t \in \R})$ is a crossed product extension. We thus have an isomorphism $\pi_\varphi : c(M) \rightarrow M \rtimes_{\sigma^\varphi} \R$ that intertwines the scaling action $\theta$ with the dual action of $\sigma^\varphi$.
\item By the relative commutant theorem of Connes and Takesaki \cite[Theorem 5.1]{CT76}, we have $M' \cap c(M)=\cZ(c(M))$. The restriction $\theta|_{\cZ(c(M))} : \R^*_+ \curvearrowright \cZ(c(M))$ of the trace scaling flow is the \emph{(commutative) flow of weights} of $M$. We have $\cZ(c(M))^\theta=\cZ(M)$.
\item The flow of weights $\R^*_+ \curvearrowright \cZ(c(M))$ extends to a continuous action $\rL^0(\cZ(M),\R^*_+) \curvearrowright \cZ(c(M))$ where $\rL^0(\cZ(M),\R^*_+)$ is the topological group of all unbounded positive functions affiliated with $\cZ(M)$. We say that a type $\III$ von Neumann algebra $M$ is of type :
\begin{enumerate}
\item [- $\III_0$] if the action $\rL^0(\cZ(M),\R^*_+) \curvearrowright \cZ(c(M))$ is faithful.
\item [- $\III_{\Lambda}$] with $\Lambda \in \cZ(M)$, $0< \Lambda < 1$, if the kernel of $\rL^0(\cZ(M),\R^*_+) \curvearrowright \cZ(c(M))$ is of the form $\{ \Lambda^N \mid N \in \rL^0(\cZ(M),\Z) \}$.
\item [- $\III_1$] if its flow of weights is trivial, i.e.\ if $\cZ(c(M))=\cZ(M)$.
\end{enumerate}
Every type $\III$ von Neumann algebra $M$ decomposes as a direct sum of von Neumann algebras of the above types
 $$M = M_{\III_0} \oplus M_{\III_\Lambda} \oplus M_{\III_1}.$$
 
 \item Let $N \subset M$ be an inclusion with expectation. To every $E \in \cE(M,N)$ is associated a normal trace preserving $\R^*_+$-equivariant embedding $ i :c(N) \rightarrow c(M)$ with $i|_{N}=\id_N$ and a trace preserving faithful normal conditional expectation $c(E) \in \cE(c(M),i(c(N)))$ such that $c(E)|_M=E$. The image of this embedding $i$ depends on the choice of $E$. We will denote it by $c_E(N)$ when we want to keep track of this dependence. When no confusion can arise, we will simply denote it $c(N)$.
 \item More generally, if $N \subset M$ is an inclusion and $T \in \cP(M,N)$, then there exists a similar embedding $ i :c(N) \rightarrow c(M)$ with $i|_{N}=\id_N$ and $c(T) \in \cP(c(M),i(c(N)))$ such that $\tau_N \circ T=\tau_M$ and $c(T)|_M=T$. The image of this embedding will be denoted $c_T(N) \subset c(M)$. 
 \item Let $N \subset M$ be an inclusion and $T \in \cP(M,N)$. Take $\varphi \in \cP(N)$. Then the embedding $i_T : c(N) \rightarrow c_T(N) \subset c(M)$ induced by $T$, sends $\varphi^{\ri t}$ to $(\varphi \circ T)^{\ri t}$ for all $t \in \R$. In particular, if $\psi \in \cP(N)$ is another weight then $(\varphi \circ T)^{\ri t} (\psi \circ T)^{-\ri t}=\varphi^{\ri t} \psi^{-\ri t} \in N$. As a consequence, the restriction of $\sigma^{\varphi \circ T}$ to $N' \cap M$ does not depend on the choice of $\varphi$ and will simply be denoted $\sigma^T$. It is called the \emph{modular automorphism group} of $T$. See \cite{Ha77b} for more on this.
\end{itemize}

\subsection{The standard form}
Let $M$ be a von Neumann algebra and $(c(M),\tau, \theta)$ its core.
\begin{itemize}
\item Let $\cA$ be the algebra of $\tau$-measurable operators affiliated with $c(M)$. Following \cite{Ha79}, for every $p \in [1,+\infty)$, we define the noncommutative $\rL^p$-space $$\rL^p(M)=\{ x \in \cA \mid \forall \lambda \in \R^*_+, \; \theta_{\lambda}(x)=\lambda^{-1/p} x \}.$$
Each $\rL^p$-space is a $M$-bimodule. The space $\rL^1(M)$ is identified with $M_*$ by the linear order-preserving $M$-bimodular bijection that sends $\phi \in M_*^+$ to $\frac{\rd \widehat{\phi}}{\rd \tau}$ where $\widehat{\phi}=\phi \circ I^\theta$. On $\rL^1(M) \cong M_*$, we denote by $\int$ the linear form that sends $\phi \in M_*$ to $\phi(1)$.
\item The space $\rL^2(M)$ is a hilbert space when equipped with the scalar product $\langle \xi, \eta \rangle := \int \xi \eta^*$, where $\int$ is the linear form on $\rL^1(M) \cong M_*$ that sends $\phi$ to $\phi(1)$. The left and right regular representations $\lambda,\rho : M \rightarrow \B(\rL^2(M))$ are defined by $\lambda(x)\xi = x\xi$ and $\rho(x)=\xi x$ for every $x \in M$ and $\xi \in \rL^2(M)$. Note that $\rho$ is an anti-representation rather then a representation. We have $\lambda(M)'=\rho(M)$ and the conjugate linear involution $J : \xi \mapsto \xi^*$ satisfies $J \lambda(x) J=\rho(x^*)$ for every $x \in M$. The set 
$$\rL^2(M)_+=\rL^2(M) \cap c(M)_+= \{ \phi^{1/2} \mid \phi \in M_*^+=\rL^1(M)_+ \}$$ is a \emph{self-dual positive cone} in $\rL^2(M)$. See \cite{Ha75} for more details on the standard form.
\item Take $\varphi \in \cP(M)$ and let $\mathfrak{n}_\varphi=\{ x \in M \mid \varphi(x^*x) < +\infty \}$. Then for every $x \in \mathfrak{n}_\varphi$, the product $x\varphi^{1/2}$ is a well-defined element of $\rL^2(M)$ and $\| x \varphi^{1/2} \|^2=\varphi(x^*x)$. We have $(ax)\varphi^{1/2}=a(x\varphi^{1/2})$ for all $a \in M$ and $x \in \mathfrak{n}_\varphi$. We have similar facts for $\varphi^{1/2}x$ when $x \in \mathfrak{n}_\varphi^*$.
\item Take $\varphi, \psi \in \cP(M)$. There exists an closed densely defined positive operator $\Delta_{\psi,\varphi}$ on $\rL^2(M)$ such that $\Delta_{\psi,\varphi}^{\ri t}(\xi)=\psi^{\ri t} \xi \varphi^{-\ri t}$ for every $\xi \in \rL^2(M)$ and $t \in \R$. We have 
$$ \Delta_{\psi,\varphi}^{\ri t}\lambda(x)\Delta_{\psi,\varphi}^{-\ri t} = \lambda(\sigma_t^\psi(x)) \text{ and } \Delta_{\psi,\varphi}^{\ri t}\rho(x)\Delta_{\psi,\varphi}^{-\ri t} = \rho(\sigma_{-t}^\varphi(x))$$
for all $x \in M$ and $t \in \R$. Moreover, the graph of $\Delta_{\psi,\varphi}^{1/2}$ is the closure of 
$$\{ (x\varphi^{1/2}, \psi^{1/2}x) \mid  x \in \mathfrak{n}_\varphi \cap \mathfrak{n}_\psi^* \}.$$
\end{itemize}

\subsection{The Ocneanu ultrapower}
The reference for ultrapowers of von Neumann algebras is \cite{AH12}. We only recall here the definition of the Ocneanu ultrapower and its basic properties.

Let $I$ be any nonempty directed set and $\omega$ any {\em cofinal} ultrafilter on $I$, i.e.\ $\{i \in I : i \geq i_0\} \in \omega$ for every $i_0 \in I$. When $I = \N$, $\omega$ is cofinal if and only if $\omega$ is {\em nonprincipal}, i.e.\ $\omega \in \beta(\N) \setminus \N$.  Let $M$ be a von Neumann algebra. Define
\begin{align*}
\mathcal I_\omega(M) &= \left\{ (x_i)_i \in \ell^\infty(I, M) \mid x_i \to 0\ \ast\text{-strongly as } i \to \omega \right\} \\
\mathcal M^\omega(M) &= \left \{ (x_i)_i \in \ell^\infty(I, M) \mid  (x_i)_i \, \mathcal I_\omega(M) \subset \mathcal I_\omega(M) \text{ and } \mathcal I_\omega(M) \, (x_i)_i \subset \mathcal I_\omega(M)\right\}.
\end{align*}
The multiplier algebra $\cM^\omega(M)$ is a $\rC^*$-algebra and $\mathcal I_\omega(M) \subset \cM^\omega(M)$ is a norm closed two-sided ideal. Following \cite[\S 5.1]{Oc85}, the quotient $\rC^{*}$-algebra $M^\omega := \cM^\omega(M) / \mathcal I_\omega(M)$ is a von Neumann algebra, known as the {\em Ocneanu ultrapower} of $M$. We denote the image of $(x_i)_i \in \cM^\omega(M)$ by $(x_i)^\omega \in M^\omega$. We can view $M$ as a von Neuman subalgebra of $M^\omega$ by identifying $x \in M$ with the constant net.

If $N \subset M$ is with expectations, then $\cM^\omega(N) \subset \cM^\omega(M)$, hence we can view $N^\omega$ as a subalgebra of $M^\omega$ with expectations.

 There is a canonical conditional expectation $E \in \cE(M^\omega, M)$ such that $E((x_i)^\omega)$ is the weak*-limit of $(x_i)_{i \in I}$ along the ultrafilter $\omega$, for every $(x_i)_i \in \mathfrak M^\omega(M)$. Using this conditional expectation, we can extend every $\varphi \in \cP(M)$ to $\varphi^\omega=\varphi \circ E \in \cP(M^\omega)$. We then have the following relation for the modular flow
$$\sigma_t^{\varphi^\omega}((x_i)^\omega)= (\sigma_t^{\varphi}(x_i))^\omega$$
for every $t \in \R$ and every $(x_i)_i \in \cM^\omega(M)$.

Associated to $E : M^\omega \rightarrow M$, we have a natural inclusion $c(M)=c_E(M) \subset c(M^\omega)$ and a natural inclusion $\rL^2(M) \subset \rL^2(M^\omega)$. 

We also have $\cM^\omega(M) \subset \cM^\omega(c(M))$, so we can view $M^\omega$ as a subalgebra of $c(M)^\omega$. In fact we even have natural inclusion $c(M^\omega)=c(M) \vee M^\omega \subset c(M)^\omega$.

It is important to observe that $\rL^2(M^\omega)$ is \emph{not} the ultrapower Hilbert space $\rL^2(M)^\omega$. In fact, $\rL^2(M^\omega)$ can be identified with a  subspace of $\rL^2(M)^\omega$ as follows. The $\rC^*$-algebra $\ell^\infty(I,M)$ acts naturally on the right and on the left on the utrapower Hilbert space $\rL^2(M)^\omega$. Then we have an identification 
\begin{equation}
\rL^2(M^\omega)=\{ \xi \in \rL^2(M)^\omega \mid\forall x \in \cI_\omega(M), \; x\xi=\xi x=0\}. 
\end{equation}
This subspace is clearly stable under left and right multiplication by elements of $\cM^\omega(M)$, and this gives the left and right action of the quotient $M^\omega=\cM^\omega(M)/\cI^\omega(M)$. The positive cone of $\rL^2(M^\omega)$ is simply $\rL^2(M^\omega) \cap (\rL^2(M)_+)^\omega$. The subspace $\rL^2(M) \subset \rL^2(M^\omega)$ coincides with the subspace of constant nets in $\rL^2(M)^\omega$. Further details can be found in \cite{AH12} where $M^\omega$ is identified with a corner of the Groh-Raynaud ultrapower of $M$ and $\rL^2(M)^\omega$ is identified with the standard form of the Groh-Raynaud ultrapower, so that $\rL^2(M^\omega)$ is nothing more than the corner $p\rL^2(M)^\omega p$ where $p$ is the support projection of $M^\omega$ inside the Groh-Raynaud ultrapower.

Finally, following \cite[Proposition 1.3.1]{Co75b}, we observe that when $M$ admits a faithful normal finite trace $\tau$ and we identify $\rL^2(M)$ with $\rL^2(M,\tau)$, then $\rL^2(M^\omega,\tau^\omega)$ is the subspace of $\rL^2(M,\tau)^\omega$ consisting of $\omega$-\emph{equi-integrable} vectors.
\begin{prop}
Let $(M,\tau)$ be a von Neumann algebra with a faithful normal finite trace. Identify $\rL^2(M)$ with $\rL^2(M,\tau)$.  Let $\xi=(\xi_i)^\omega \in \rL^2(M)^\omega= \rL^2(M,\tau)^\omega$. Then $\xi \in \rL^2(M^\omega)=\rL^2(M^\omega,\tau^\omega)$ if and only if  it is $\omega$-equi-integrable : for every $\varepsilon > 0$, there exists $a > 0$ such that $\lim_{i \to \omega} \| 1_{[a,+\infty)}(|\xi_i|)|\xi_i| \|_2 \leq \varepsilon$. 
\end{prop}

\section{Ultrapower implementation of binormal states} \label{section ultrapower binormal}
The following is the key technical result in this paper. It is inspired by an averaging argument that goes back to \cite{Ha84} and was used more recently in \cite{IV15} and \cite[Lemma 5.1]{Ma17}.
\begin{thm} \label{ultrapower implement tracial}
Let $M$ be a finite von Neumann algebra. For any state $\Psi \in \B(\rL^2(M))$, the following are equivalent:
\begin{enumerate}
\item [$(\rm i)$] $\Psi \circ \lambda$ and $\Psi \circ \rho$ are normal states.
\item [$(\rm ii)$] There exists an abelian von Neumann algebra $A$, a cofinal ultrafilter $\omega$ on some directed set $I$ and a vector $\xi \in \rL^2((A \ovt M)^\omega)$ such that $\Psi(T)=\langle (1 \otimes T)^\omega \xi, \xi \rangle$ for all $T \in \B(\rL^2(M))$. 
\end{enumerate} 
\end{thm}
\begin{proof}
$(\rm ii) \Rightarrow (\rm i)$ Clear.
$(\rm i) \Rightarrow (\rm ii)$ Fix a faithful normal tracial state $\tau$. Let $E$ be an orthonormal basis of $\rL^2(M,\tau)$. Let $A=L(\Z^{\oplus E})$ and for every $e \in E$, let $u_e \in \mathcal{U}(A)$ be the generator of the copy of $\Z$ indexed by $e$. Let $V : \rL^2(M,\tau) \rightarrow \rL^2(A,\mu)$ be the unique isometry that sends $e$ to $u_e$. Define a map $\theta : \B(\rL^2(M,\tau))^+_* \rightarrow \rL^2(A \ovt M,\mu \otimes \tau)$ by $\theta(\Phi)=(V \otimes 1)(\Phi^{1/2})$, where we view the Hilbert-Schmidt operator $\Phi^{1/2}$ as a vector in $\rL^2(M,\tau) \otimes \rL^2(M,\tau)$. For every $\kappa > 0$, let $\Omega_{\kappa}$ be the set of all $\Phi \in \B(\rL^2(M,\tau))^+_*$ such that $\Phi \circ \lambda \leq \kappa \tau$ and $\Phi \circ \rho \leq \kappa \tau$.

\begin{claim} \label{Omega equi-integrable}
For every $\kappa > 0$, the set $\theta(\Omega_{\kappa}) \subset \rL^2(A \ovt M,\mu \otimes \tau)$ is equi-integrable.
\end{claim}

\begin{proof}[Proof of Claim \ref{Omega equi-integrable}]
Take $\Phi \in \Omega_{\kappa}$. Write $\Phi^{1/2}=\sum_{e \in E} e \otimes a_e$ with $a_e \in \rL^2(M,\tau)$. Then we have $a_e \in M$ for all $e \in E$ and $\sum_{e \in E} a_e a_e^* \leq \kappa$ and $\sum_{e \in E} a_e^*a_e \leq \kappa$. 

Let $x=\theta(\Phi^{1/2}) = \sum_{e \in E} u_e \otimes a_e$.  Then we have
$$ x^*x = \sum_{ e,f \in E} u_e^{*}u_f \otimes a_e^* a_f $$
 and 
 $$|x |^4= (x^*x)^2 = \sum_{e,f,g,h} u_e^{*}u_f u_g^{*}u_h\otimes a_e^* a_f a_g^* a_h. $$
Now, observe that $$ \tau(u_e^*u_fu_g^*u_h)= \begin{cases}
1&\text{if $\{e,g\}=\{f,h\}$}\\
0 &\text{otherwise.}
\end{cases} $$ 
Hence, we have
$$ (\mu \otimes \tau)(|x|^4)  
\leq \sum_{e,g \in E} \tau(a_e^*a_e a_g^*a_g) + \sum_{e,g \in E} \tau(a_e^* a_g a_g^* a_e) \leq 2\kappa^2.$$
This shows that $\theta(\Omega_\kappa)$ is bounded in $\rL^4(A \otimes M, \mu \otimes \tau)$. A fortiori, it is equi-integrable in $\rL^2(A \otimes M, \mu \otimes \tau)$.

\end{proof}

Now, take $(\Psi_i)_{i \in I}$ a net of normal states on $\B(\rL^2(M))$ that converges to $\Psi$ in the weak$^*$ topology. By the Hahn-Banach theorem, up to taking convex combination, we can suppose that $$\lim_i \| \Psi_i \circ \lambda - \Psi \circ \lambda \|=\lim_i \| \Psi_i \circ \rho - \Psi \circ \rho \|=0.$$

\begin{claim}
For every $\varepsilon > 0$, there exists $\kappa > 0$ and $i_0 \in I$ such that for all $i \geq i_0$, we have $d(\Psi_i, \Omega_{\kappa}) < \varepsilon$ where $d$ is the distance comming from the norm of $\B(\rL^2(M,\tau))_*$.
\end{claim}
\begin{proof}
Fix $\varepsilon > 0$. Write $\Psi_i \circ \lambda = \tau(h_i \cdot)$ and $\Psi_i \circ \rho = \tau(k_i \cdot)$ with $h_i,k_i \in \rL^1(M,\tau)_+$. Note that the nets $(h_i)_{i \in I}$ and $(k_i)_{i \in I}$ are convergent in $\rL^1(M,\tau)$, and in particular, they are equi-integrable. Thus, we can find $\kappa > 0$ and $i_0 \in I$ large enough such that $\tau(e_ih_i) \geq 1- \varepsilon$ and $\tau(f_i k_i) \geq 1-\varepsilon$ for all $i \geq i_0$ where $e_i=1_{[0,\kappa]}(h_i)$ and $f_i=1_{[0,\kappa]}(k_i)$. Let $q_i=\lambda(e_i)\rho(f_i)$. Then, for all $i \geq i_0$, we have $$\Psi_i(q_i) \geq \Psi_i(\lambda(e_i))+\Psi_i(\rho(f_i))-1 \geq 1-2\varepsilon,$$ hence $\| \Psi_i - \Psi_i(q_i \cdot q_i) \| \leq 4 \varepsilon^{1/2}$. Moreover, by construction, we have $\Psi_i(q_i \cdot q_i) \in \Omega_{\kappa}$. This is what we wanted, up to replacing $\varepsilon$ by $4\varepsilon^{1/2}$.
\end{proof}

By using the claim and the Powers-St\o rmer inequality, we get $d \left( \theta(\Psi_i), \theta(\Omega_{\kappa}) \right)^2 \leq \varepsilon$ for all $i \geq i_0$ where $d$ is the distance comming from the $\rL^2$-norm. By Claim \ref{Omega equi-integrable}, $\theta(\Omega_{\kappa})$ is equi-integrable in $\rL^2(A \ovt M, \mu \otimes \tau))$. Therefore, the net $(\theta(\Psi_i))_{i \in I}$ is $\omega$-equi-integrable for any cofinal ultrafilter $\omega$ on $I$. We thus have $\xi=(\theta(\Psi_i))^\omega \in \rL^2((A \ovt M)^\omega, (\mu \otimes \tau)^\omega)$. By construction, for all $T \in \B(\rL^2(M,\tau))$, we have 
$$\langle (1 \otimes T)^\omega \xi, \xi \rangle=\lim_{i \to \omega} \langle (1 \otimes T) \theta(\Psi_i), \theta(\Psi_i) \rangle = \lim_{i \to \omega} \langle (1 \otimes T) \Psi_i^{1/2}, \Psi_i^{1/2} \rangle= \Psi(T)$$
as we wanted.
\end{proof}

\begin{thm} \label{ultrapower implement semifinite}
Let $M$ be a countably decomposable semifinite von Neumann algebra. For a state $\Psi \in \B(\rL^2(M))^*$, the following are equivalent:
\begin{enumerate}
\item [$(\rm i)$] $\Psi \circ \lambda$ and $\Psi \circ \rho$ are normal states.
\item [$(\rm ii)$] There exists an abelian von Neumann algebra $A$, a cofinal ultrafilter $\omega$ on some directed set $I$ and a vector $\xi \in \rL^2((A \ovt M)^\omega)$ such that $\Psi(T)=\langle (1 \otimes T)^\omega \xi, \xi \rangle$ for all $T \in \B(\rL^2(M))$. 
\end{enumerate} 
\end{thm}
\begin{proof}
$(\rm ii) \Rightarrow (\rm i)$ Clear.

$(\rm i) \Rightarrow (\rm ii)$ Fix a trace $\tau \in \cP(M)$ and take $(p_n)_{n \in \N}$ an increasing sequence of finite trace projections in $M$ such that $p_n$ converges strongly to $1$. Let $E$ be an orthonormal basis of $\rL^2(M,\tau)$ that is compatible with the increasing sequence of subspaces $p_n\rL^2(M,\tau)p_n=\rL^2(p_nMp_n,\tau)$. By this we mean that there exists an increasing sequence of subsets $F_n \subset E, \: n \in \N$ such that $F_n$ is an orthonormal basis of $\rL^2(p_nMp_n,\tau)$ for every $n \in \N$.

 Let $A=L(\Z^{\oplus E})$ and for every $e \in E$, let $u_e \in \mathcal{U}(A)$ be the generator of the copy of $\Z$ indexed by $e$. Let $V : \rL^2(M,\tau) \rightarrow \rL^2(A,\mu)$ be the unique isometry that sends $e$ to $u_e$. Define a map $\theta : \B(\rL^2(M,\tau))^+_* \rightarrow \rL^2(A \ovt M,\mu \otimes \tau)$ by $\theta(\Phi)=(V \otimes 1)(\Phi^{1/2})$, where we view the Hilbert-Schmidt operator $\Phi^{1/2}$ as a vector in $\rL^2(M,\tau) \otimes \rL^2(M,\tau)$. For every $\kappa > 0$ and every $n \in \N$, let $\Omega_{\kappa,n}$ be the set of all $\Phi \in \B(\rL^2(M))^+_*$ such that $\Phi \circ \lambda \leq \kappa \tau$, $\Phi \circ \rho \leq \kappa \tau$ and $\Phi$  is supported on $\lambda(p_n)\rho(p_n)$ (or equivalently, $\Phi \in \B(\rL^2(p_nMp_n,\tau))^+_*$).

Having replaced $\Omega_\kappa$ by $\Omega_{\kappa,n}$, we obtain the following claim exactly as in the finite case.
\begin{claim} \label{Omega equi-integrable 2}
For every $\kappa > 0$ and every $n \in \N$, the set $\theta(\Omega_{\kappa,n}) \subset (1 \otimes p_n)\rL^2(A \ovt M,\mu \otimes \tau)(1 \otimes p_n)$ is equi-integrable.
\end{claim}

Now, take $(\Psi_i)_{i \in I}$ a net of normal states on $\B(\rL^2(M,\tau))$ that converge to $\Psi$ in the weak$^*$ topology. By the Hahn-Banach theorem, up to taking convex combination, we can suppose that $$\lim_i \| \Psi_i \circ \lambda - \Psi \circ \lambda \|=\lim_i \| \Psi_i \circ \rho - \Psi \circ \rho \|=0.$$

\begin{claim}
For every $\varepsilon > 0$, there exists $\kappa > 0$, $n \in \N$ and $i_0 \in I$ such that for all $i \geq i_0$, we have $d(\Psi_i, \Omega_{\kappa,n}) \leq \varepsilon$.
\end{claim}
\begin{proof}
Fix $\varepsilon > 0$. Write $\Psi_i \circ \lambda = \tau(h_i \cdot)$ and $\Psi_i \circ \rho = \tau(k_i \cdot)$ with $h_i,k_i \in \rL^1(M,\tau)_+$. Note that the nets $(h_i)_{i \in I}$ and $(k_i)_{i \in I}$ are convergent in $\rL^1(M,\tau)$, and in particular, they are equi-integrable. Thus, we can find $\kappa > 0$, $n \in \N$ and $i_0 \in I$ large enough such that $\tau(e_ih_i) \geq 1- \varepsilon$ and $\tau(f_i k_i) \geq 1-\varepsilon$ for all $i \geq i_0$ where $e_i=1_{(0,\kappa]}(p_nh_ip_n)$ and $f_i=1_{(0,\kappa]}(p_nk_ip_n)$. Let $q_i =\lambda(e_i)\rho(f_i)$. Then, for all $i \geq i_0$, we have $$\Psi_i(q_i) \geq \Psi_i(\lambda(e_i))+\Psi_i(\rho(f_i))-1 \geq 1-2\varepsilon,$$  hence we get easily $\| \Psi_i - \Psi_i(q_i \cdot q_i) \| \leq 4 \varepsilon^{1/2}$. Moreover, by construction, we have $\Psi_i(q_i \cdot q_i) \in \Omega_{\kappa,n}$. This is what we wanted, up to replacing $\varepsilon$ by $4\varepsilon^{1/2}$.
\end{proof}

Let $\omega$ be a cofinal ultrafiler on $I$ and let $\xi=(\theta(\Psi_i))^\omega \in \rL^2(M)^\omega$. By using the previous claim and the Powers-St\o rmer inequality, we get $d \left( \theta(\Psi_i), \theta(\Omega_{\kappa,n}) \right)^2 \leq \varepsilon$ for all $i \geq i_0$. Thus $d \left( \xi, \theta(\Omega_{\kappa,n})^\omega \right)^2 \leq \varepsilon$. By Claim \ref{Omega equi-integrable 2}, $\theta(\Omega_{\kappa,n})$ is equi-integrable in $\rL^2(A \ovt M, \mu \otimes \tau))$, hence $\theta(\Omega_{\kappa,n})^\omega \subset \rL^2((A \ovt M)^\omega, (\mu \otimes \tau)^\omega)$. Therefore, we have  $d \left( \xi,  \rL^2((A \ovt M)^\omega, (\mu \otimes \tau)^\omega) \right)^2 \leq \varepsilon$.

Since $\varepsilon > 0$ was arbitrary, we conclude that $\xi=(\theta(\Psi_i))^\omega \in \rL^2((A \ovt M)^\omega, (\mu \otimes \tau)^\omega)$. By construction, for all $T \in \B(\rL^2(M,\tau))$, we have 
$$\langle (1 \otimes T)^\omega \xi, \xi \rangle=\lim_{i \to \omega} \langle (1 \otimes T) \theta(\Psi_i), \theta(\Psi_i) \rangle = \lim_{i \to \omega} \langle (1 \otimes T) \Psi_i^{1/2}, \Psi_i^{1/2} \rangle= \Psi(T)$$
as we wanted.
\end{proof}

\begin{df}[\cite{Ma18}]
Let $X$ be an unbounded self-adjoint operator on a Hilbert space $H$. Take $\lambda \in \R$. We say that a state $\Psi \in \B(H)^*$ is a \emph{$\lambda$-eigenstate} for $X$ if $\Psi(h(X))=h(\lambda)$ for every bounded continuous function $h \in C_b(\R)$.
\end{df}
\begin{remark} \label{remark eigenstate}
Note that $\Psi$ is a $\lambda$-eigenstate for $X$ if and only if $\Psi(1_U(X))=1$ for every open set $U \subset \R$ containing $\lambda$. Indeed, this condition implies that $|\Psi(h(X))|=|\Psi(1_U(X) h(X))| \leq \| 1_U h\|_\infty$ and if $h(\lambda)=0$, then $\inf_U \| 1_U h\|_\infty=0$, thus $\Psi(h(X))=0$.
\end{remark}

\begin{thm}  \label{ultrapower all type}
Let $M$ be a countably decomposable von Neumann algebra. Let $\Psi \in \B(\rL^2(M))^*$ be a state such that
\begin{enumerate}
\item $\Psi \circ \lambda$ and $\Psi \circ \rho$ are normal states.
\item There exists $\varphi,\phi \in \cP(M)$ such that $\Psi$ is a $1$-eigenstate of $\Delta_{\varphi,\phi}$.
\end{enumerate}
Then there exists an abelian von Neumann algebra $A$, a cofinal ultrafilter $\omega$ on some directed set $I$ and a vector $\xi \in \rL^2((A \ovt M)^\omega)$ such that $\Psi(T)=\langle (1 \otimes T)^\omega \xi, \xi \rangle$ for all $T \in \B(\rL^2(M))$. 
\end{thm}
\begin{proof}
Let $N=c(M)$. We can identify $\rL^2(N)$ with $\rL^2(M) \otimes \rL^2(\R)$ in such a way that :
\begin{itemize}
\item $\lambda_N(\varphi^{it})=1 \otimes \lambda_\R(t)$ for all $s \in \R$.
 \item $\lambda_N(x)$ is the desintegrable operator on $\rL^2(N)=\rL^2(\R, \rL^2(M))$ given by $t \mapsto \lambda_M(\sigma_{-t}^\varphi(x))$.
\item $\rho_N(\phi^{it})=  \Delta_{ \varphi,\phi }^{i t}  \otimes \lambda_\R(t)$ for all $t \in \R$.
\item $\rho_N(x)= \rho_M(x)\otimes 1 $ for all $x \in M$.
\end{itemize}
Choose any unit vector $f \in \rL^2(\R)$ and let $E : \B(\rL^2(N)) \rightarrow \B(\rL^2(M))$ be the normal conditional expectation given by $E(T \otimes S)=\langle S f, f\rangle T$ for all $S \in \B(\rL^2(\R))$ and $T \in \B(\rL^2(M))$. Define a state $\Phi \in \B(\rL^2(N))^*$ by $\Phi = \Psi \circ E$.
\begin{claim}
$\Phi$ is a binormal state on $\B(\rL^2(N))$, i.e.\ $\Phi \circ \lambda_N$ and $\Phi \circ \rho_N$ are both normal states.
\end{claim}
\begin{proof}
Since $N$ is generated by $M$ and the unitaries $(\varphi^{it})_{t \in \R}$, we clearly have $\lambda_N(N) \subset \rho_N(M)'= \lambda_M(M) \ovt \B(\rL^2(\R))$. Therefore $E(\lambda_N(N)) \subset \lambda_M(M)$. Since $E \circ \lambda_N$ is normal and $\Psi$ is normal on $\lambda_M(M)$, we conclude that $\Phi \circ \lambda_N=\Psi \circ E \circ \lambda_N$ is also normal. 

Now we want to show that $\Phi \circ \rho_N$ is normal. Let $U$ be the desintegrable unitary operator on $\rL^2(N)=\rL^2(\R, \rL^2(M))$ given by $t \mapsto \Delta_{\varphi,\phi}^{i t}$.  Observe that
$$ \Phi(U)=\Psi(E(U))=\Psi \left( \int_{t \in \R} \Delta_{\varphi,\phi}^{it} \; |f(t)|^2 \: \rd t \right)=\Psi(h(\Delta_{\varphi,\psi}))$$
where $h \in C_0(\R^*_+)$ is given by $h(\lambda)=\int \lambda^{\ri t} |f(t)|^2 \: \rd t$. Since $\Psi$ is a $1$-eigenstate for $\Delta_{\varphi,\phi}$, we obtain $\Phi(U)=h(1)=1$. In particular, this implies that $\Phi=\Phi \circ \Ad(U)=\Phi \circ \Ad(U^*)$.

Now, observe that for all $x \in M$, $(\Ad(U^*) \circ \rho_N)(x)=U^*(\rho_M(x) \otimes 1)U$ is the desintegrable operator on $\rL^2(N)=\rL^2(\R, \rL^2(M))$ given by $t \mapsto \rho_M(\sigma_{t}^\phi(x))$. Moreover $(\Ad(U^*) \circ \rho_N)(\phi^{it})=U^*(\Delta_{\varphi,\phi}^{i t}  \otimes \lambda_\R(t))U=1 \otimes \lambda_\R(t) $. Thus, as above, we obtain that $(\Ad(U^*) \circ \rho_N)(N) \subset  \rho_M(M) \ovt \B(\rL^2(\R)) $, hence $(E \circ \Ad(U^*) \circ \rho_N)(N) \subset \rho_M(M)$. Since $E \circ \Ad(U^*) \circ \rho_N$ is normal, and $\Psi$ is normal on $\rho_M(M)$, we conclude that 
$$ \Phi \circ \rho_N= \Phi \circ \Ad(U^*) \circ \rho_N = \Psi  \circ E \circ \Ad(U^*) \circ \rho_N$$
is also normal. This ends the proof of the claim.
\end{proof}
Since $N$ is semifinite and $\Phi$ is binormal, Theorem \ref{ultrapower implement semifinite} tells us that there exists an abelian von Neumann algebra $A$, a cofinal ultrafilter $\omega$ on some directed set $I$ and a vector $\xi \in \rL^2((A \ovt N)^\omega)$ such that $\Phi(T)=\langle (1 \otimes T)^\omega \xi, \xi \rangle$ for all $T \in \B(\rL^2(N))$. Let $p \in \B(\rL^2(\R))$ be the rank one projection on the unit vector $f \in \rL^2(\R)$. We have $\Phi(1 \otimes p)=1$, thus $(1 \otimes 1 \otimes p)^\omega \xi=\xi$. This means that we can write $\xi = \eta \otimes f$ where $\eta \in \rL^2(A \ovt M)^\omega$. We claim that $\eta \in \rL^2((A \ovt M)^\omega)$. For this, we have to check that for any bounded net $(x_i)_{i \in I}$ in $(A \ovt M)$ that converges $*$-strongly to $0$, we have $(\lambda_{A \ovt M}(x_i))^\omega \eta=0$ and $(\rho_{A \ovt M}(x_i))^\omega \eta=0$.

First, since $\xi$ is in $\rL^2((A \ovt N)^\omega)$, we have $(\rho_{A \ovt N}(x_i))^\omega \xi=0$. But $\rho_{A \ovt N}(x_i)= \rho_{A \ovt M}(x_i) \otimes 1$ and $\xi = \eta \otimes f$. We thus get
$(\rho_{A \ovt M}(x_i) \otimes 1)^\omega (\eta \otimes f)=0$, hence $(\rho_{A \ovt M}(x_i))^\omega \eta=0$.

Since $\xi$ is in $\rL^2((A \ovt N)^\omega)$, we also have $(\lambda_{A \ovt N}(x_i))^\omega \xi=0$. But $\lambda_{A \ovt N}(x_i)= (1 \otimes U^*)(\lambda_{A \ovt M}(x_i) \otimes 1)(1 \otimes U)$ hence
$$(\lambda_{A \ovt M}(x_i) \otimes 1)^\omega(1 \otimes U)^\omega \xi=0.$$
But we know that $\xi=(1 \otimes U)^\omega \xi$ because $\langle (1 \otimes U)^\omega \xi, \xi \rangle =\Phi(U)=1$. Thus, again, we conclude that $(\lambda_{A \ovt M}(x_i) \otimes 1)^\omega \xi=0$ hence $(\lambda_{A \ovt M}(x_i))^\omega \eta=0$.

We proved that $\eta \in \rL^2((A \ovt M)^\omega)$ and we clearly have
$$ \langle (1 \otimes T)^\omega \eta, \eta \rangle = \langle (1 \otimes T \otimes 1)^\omega (\eta \otimes f), \eta \otimes f \rangle= \Phi(T \otimes 1)=\Psi(T)$$
for all $T \in \B(\rL^2(M))$.
\end{proof}

\section{Applications to $\II_1$ factors} \label{section II1}
This section contains some applications of the ultrapower technique of the previous section to $\II_1$ factors. We start by giving new short proofs of some fundamental results from Connes's work on the classification of injective factors.
\begin{thm}[{\cite[Theorem 2.1]{Co75b}}] \label{connes 1}
Let $M$ be a full $\II_1$ factor with trace $\tau$. Then there exists a finite subset $F \subset \cU(M)$ such that
$$\forall x \in M, \quad \|x-\tau(x) \|_2 \leq \sum_{u \in F} \| xu-ux\|_2.$$
\end{thm}
\begin{proof}
Suppose on the contrary, that there exists a net $(x_i)_{i \in I}$ in $M$ such that $\tau(x_i)=0$ and $\| x_i \|_2=1$ for all $i \in I$ and $\lim_i \| x_iu -ux_i\|_2=0$ for all $u \in \cU(M)$. Let $\Psi \in \B(\rL^2(M))$ be a weak* accumulation point of the net of states $\omega_i : T \mapsto \langle T \hat{x_i}, \hat{x_i} \rangle$. Since $\tau(x_i)=0$ for all $i \in I$, we have $\Psi(p)=0$ where $p \in \B(\rL^2(M))$ is the rank one projection on $\C\hat{1}$. Observe also that $\Psi \circ \lambda$ and $\Psi \circ \rho$ are both invariant under conjugation by elements of $\cU(M)$. Since $\tau$ is the unique trace on $M$, we thus have $\Psi \circ \lambda = \Psi \circ \rho=\tau$. In particular, $\Psi$ is binormal. Therefore, by Theorem \ref{ultrapower implement tracial}, we can find some abelian von Neumann algebra, some cofinal ultrafilter $\omega$ on some directed set $J$ and some vector $\eta \in \rL^2((A \ovt M)^\omega)$ such that $\Psi(T)=\langle (1 \otimes T)^\omega \eta, \eta \rangle$ for all $T \in \B(\rL^2(M))$. Since $\Psi(| \lambda(a)-\rho(a)|^2)=0$ for all $a \in M$, we get $a \eta=\eta a$ for all $a \in M$. Since $M$ is full, this implies that $\eta \in \rL^2((A \otimes 1)^\omega)$ thanks to Lemma \ref{ultrapower commutation} below. In other words $(1 \otimes p)^\omega \eta=\eta$. On the other hand we have $0=\Psi(p)=\langle (1 \otimes p)^\omega \eta,\eta \rangle$. This yields the desired contradiction.
\end{proof}

\begin{lem} \label{ultrapower commutation}
Let $(M,\tau)$ be a von Neumann algebra with a faithful normal tracial state and $A$ be countably decomposable abelian von Neumann algebra. Let $P,Q \subset M$ be von Neumann subalgebras. Suppose that $P' \cap M^\omega \subset Q^\omega$ for every cofinal ultrafilter $\omega$. Then $(1 \otimes P)' \cap (A \ovt M)^\omega \subset (A \ovt Q)^\omega$ for every cofinal ultrafilter $\omega$.
\end{lem}
\begin{proof}
Let $\mu$ be faithful normal state on $A$ and use $\mu \otimes \tau$ to define a $2$-norm on $A \ovt M$.

The condition $P' \cap M^\omega \subset Q^\omega$ means that for every $\varepsilon > 0$, there exists a finite set $F \subset \cU(P)$ and $\delta > 0$ such that if $x \in \Ball(M)$ satisfies $\sum_{u \in F} \| ux-xu\|_2^2 \leq \delta$ for all $u \in F$, then $\|x-E_Q(x)\|_2^2 \leq \varepsilon$.

 Let $\mathcal{L}= \bigcup_{A_0 \subset A} A_0 \ovt M$ where the union runs overs finite dimensional subalgebras of $A$. For $x \in \Ball(\mathcal{L})$, write $x=\sum_{i \in I} p_i \otimes x_i$ where $A=\bigoplus_{i \in I} \C p_i$ and $x_i \in \Ball(M)$. Suppose that
 $$\sum_{u \in F} \| (1 \otimes u)x-x(1 \otimes u)\|_2^2 \leq \varepsilon \delta.$$
 Then 
 $  \sum_{i \in I} \mu(p_i) \sum_{u \in F} \| ux_i-x_iu \|_2^2 \leq \varepsilon \delta.$
 Therefore, $\sum_{i \in J} \mu(p_i) \leq \varepsilon$ where $J=\{i \in I \mid  \sum_{u \in F} \| ux_i-x_iu \|_2^2 \geq \delta \}$. For $i \in I \setminus J$, we have $\| x_i- E_Q(x_i)\|^2_2 \leq \varepsilon$. Thus we obtain $$\| x- E_{A \ovt Q}(x)\|_2^2=\sum_{i \in I} \mu(p_i) \| x_i-E_Q(x_i)\|_2^2 \leq \sum_{i \in J} 2\mu(p_i) + \sum_{i \in I \setminus J} \varepsilon \mu(p_i) \leq  3 \varepsilon.$$
 This holds for every $x \in \Ball(\mathcal{L})$, hence for all $x \in \Ball(A \ovt M)$ by density of $\Ball(\mathcal{L}) \subset \Ball(A \ovt M)$. We conclude that $(1 \ovt P)' \cap (A \ovt M)^\omega \subset (A \ovt Q)^\omega$.
\end{proof}

\begin{thm}[{\cite[Theorem 3.1]{Co75b}}] \label{connes 2}
Let $M$ be a $\II_1$ factor with trace $\tau$. Take $\theta \in \Aut(M) \setminus \overline{\Inn(M)}$. Then there exists a finite subset $F \subset \cU(M)$ such that
$$\forall x \in M, \quad \|x\|_2 \leq \sum_{u \in F} \|xu- \theta(u)x\|_2.$$
\end{thm}
\begin{proof}
Suppose on the contrary, that there exists a net $(x_i)_{i \in I}$ in $M$ such that $\| x_i \|_2=1$ for all $i \in I$ and $\lim_i \| x_iu- \theta(u)x_i\|_2=0$ for all $u \in \cU(M)$. Let $\Psi \in \B(\rL^2(M))$ be a weak* accumulation point of the net of states $\omega_i : T \mapsto \langle T \hat{x_i}, \hat{x_i} \rangle$. Observe that $\Psi \circ \lambda$ and $\Psi \circ \rho$ are both invariant under conjugation by elements of $\cU(M)$. Since $\tau$ is the unique trace on $M$, we thus have $\Psi \circ \lambda = \Psi \circ \rho=\tau$. In particular, $\Psi$ is binormal. Therefore, by Theorem \ref{ultrapower implement tracial}, we can find some abelian von Neumann algebra, some cofinal ultrafilter $\omega$ on some directed set $J$ and some vector $\eta \in \rL^2((A \ovt M)^\omega)$ such that $\Psi(T)=\langle (1 \otimes T)^\omega \eta, \eta \rangle$ for all $T \in \B(\rL^2(M))$. Since $\Psi(| \lambda(\theta(a))-\rho(a)|^2)=0$ for all $a \in M$, we get $\theta(a) \eta=\eta a$ for all $a \in M$. By Lemma \ref{ultrapower approx inner} below, this contradicts the assumption that $\theta \notin \overline{\Inn(M)}$.
\end{proof}

\begin{lem} \label{ultrapower approx inner}
Let $M$ be a $\II_1$ factor and $A$ be a countably decomposable abelian von Neumann algebra. Suppose that $\theta \in \Aut(M) \setminus \overline{\Inn(M)}$. Then for every cofinal ultrafilter $\omega$ and every $y \in (A \ovt M)^\omega$ such that $\theta(x)y=yx$ for all $x \in M$, we have $y=0$.
\end{lem}
\begin{proof}
Define $P \subset M_2(M)$ as the set of all matrices of the form $\begin{pmatrix}
\theta(x) & 0   \\
0 & x 
\end{pmatrix}$ for $x \in M$. Let $Q \subset M_2(M)$ be the set of all diagonal matrices. Then $\theta \in \Aut(M) \setminus \overline{\Inn(M)}$ means that $P' \cap M_2(M)^\omega \subset Q^\omega$. The conclusion follows from Lemma \ref{ultrapower commutation}.
\end{proof}

With an argument similar to the proof of Theorem \ref{connes 1} and \ref{connes 2}, one can also prove \cite[Theorem 4.3]{Co75b}. 

The following is a new characterization of \emph{weakly bicentralized subalgebras} (see \cite[Section 4]{IM19}) that provides a converse to \cite[Lemma 3.4]{BMO19}.
\begin{thm}
Let $N \subset M$ an inclusion of finite von Neumann algebras. Then the following are equivalent :
\begin{enumerate}
\item The $M$-bimodule $\rL^2(M) \otimes_N \rL^2(M)$ is weakly contained in $\rL^2(M)$.
\item $N=(N' \cap M^\omega)' \cap M$ for some cofinal ultrafilter $\omega$.
\end{enumerate}
\end{thm}
\begin{proof}
(1) $\Rightarrow$ (2). The $M$-bimodule $\rL^2(M) \otimes_N \rL^2(M)$ is generated by a vector $\xi$ such that $\langle a\xi b,\xi \rangle = \tau(E_N(a)E_N(b))$ for some faithful normal trace $\tau$ on $M$ and the $\tau$-preserving normal conditional expectation $E_N : M \rightarrow N$. Suppose that the $M$-bimodule $\rL^2(M) \otimes_N \rL^2(M)$ is weakly contained in $\rL^2(M)$. Then there exists a state $\Psi$ on $\B(\rL^2(M))$ such that $\Psi(\lambda(a)\rho(b))=\tau(E_N(a)E_N(b))$ for all $a,b \in M$. By Theorem \ref{ultrapower implement tracial}, there exists some abelian von Neumann algebra $A$ and a vector $\eta \in \rL^2((A \ovt M)^\omega)$ such that $\langle (1 \otimes a)\eta(1 \otimes b),\eta \rangle = \tau(E_N(a)E_N(b)) \rangle$ for all $a,b \in M$. In particular, $\langle (1 \otimes a)\eta(1 \otimes b),\eta \rangle = \tau(ab) \rangle$ for all $a,b \in N$. This shows that $\eta$ is $N$-central. Take $x \in (N' \cap M^\omega)' \cap M$. Then $x\eta=\eta x$. Therefore, we have
$$ \tau(x^*x)=\langle x\eta,x\eta\rangle=\langle x\eta,\eta x\rangle=\langle x\eta x^*,\eta \rangle =\tau(E_N(x)E_N(x)^*).$$
In particular, if $E_N(x)=0$ then $x=0$. We conclude that $(N' \cap M^\omega)' \cap M=N$.
\end{proof}

Finally, we give a last application to the Akemann-Ostrand property for von Neumann algebras introduced by Ozawa in his famous paper. Let us recall the definitions.

Let $M$ be a finite von Neumann algebra. Following \cite{Oz10}, we say that an operator $T \in \B(\rL^2(M))$ is \emph{$M$-compact} if $T(\Omega)$ is compact for every equi-integrable subset $\Omega \subset \rL^2(M)$. It is enough to check this property for $\Omega_\tau=\{ \xi \in \rL^2(M) \mid |\xi|^2 \leq \tau \}$ where $\tau$ is some fixed faithful normal trace on $M$.

 Let $\mathbb{K}^L_M \subset \B(\rL^2(M))$ be the set of all $M$-compact operators. It is a closed left ideal in $\B(\rL^2(M))$. Let $$\mathbb{K}_M=(\mathbb{K}^L_M)^* \cap \mathbb{K}^L_M =(\mathbb{K}^L_M)^* \cdot \mathbb{K}^L_M$$ be the associated hereditary subalgebra and let $$\cM(\mathbb{K}_M)=\{ T \in \B(\rL^2(M)) \mid T\mathbb{K}_M \subset \mathbb{K}_M \text{ ans } \mathbb{K}_M T \subset \mathbb{K}_M\}$$ be its multiplier algebra. Then $\cM(\mathbb{K}_M)$ contains $\lambda(M)$ and $\rho(M)$. Therefore, we can define a $*$-homomorphism 
$$ \kappa : M \otimes_{\rm alg} M^{\op} \ni a \otimes b^{\op} \mapsto  \lambda(a)\rho(b) + \mathbb{K}_M \in \cM(\mathbb{K}_M)/\mathbb{K}_M.$$
Following \cite[Definition 2.1]{Ca22}, we say that $M$ has the \emph{$W^*$-Akemann-Ostrand property} (or simpy $\rm W^*AO$) if $\kappa$ is continuous with respect to the minimal tensor norm, i.e.\ if $\kappa$ extends to a $*$-morphism from $M \otimes_{\min} M^{\op}$ to $\cM(\mathbb{K}_M)/\mathbb{K}_M$. Ozawa's main result in \cite{Oz10} is that this property is satisfied by the group von Neumann algebras of hyperbolic groups (and more generally exact groups with the Akemann-Ostrand property). See \cite[theorem 2.2]{Ca22}.

Here is a new characterization of the property $\rm W^*AO$.
\begin{thm} \label{W*AO}
Let $M$ be a finite von Neumann algebra. The following are equivalent :
\begin{enumerate}
\item $M$ has property $\rm W^*AO$.
\item For every binormal state $\Phi \in \B(\rL^2(M))^*$ that vanishes on $\mathbf{K}(\rL^2(M))$, the state 
$$ M \otimes_{\rm alg} M^{\op} \ni a \otimes b^{\op} \mapsto \Phi(\lambda(a) \rho(b))$$ is $\otimes_{\rm min}$-continuous.
\end{enumerate}
\end{thm}
\begin{cor} \label{cor AO}
Let $M$ be a finite von Neumann algebra with property $\rm W^*AO$. Then for every $M$-bimodule $\cH$ such that $\cH \prec_M \rL^2(M)$ and $\cH$ is disjoint from $\rL^2(M)$, we have $\cH \prec_M \rL^2(M) \otimes \rL^2(M)$.
\end{cor}
\begin{proof}
Take a unit vector $\xi \in \cH$. Since $\cH \prec_M \rL^2(M)$, we can find a state $\Phi \in \B(\rL^2(M))^*$ such that $\Phi(\lambda(a)\rho(b))=\langle a \xi b ,\xi \rangle$ for all $a,b \in M$. Since $\cH$ is disjoint from $\rL^2(M)$, the state $\Phi$ must be singular, i.e.\ it must vanish on the compacts $\mathbf{K}(\rL^2(M))$.  By Theorem \ref{W*AO}, we thus know that 
$$ M \otimes_{\rm alg} M^{\op} \ni a \otimes b^{\op} \mapsto \Phi(\lambda(a) \rho(b)) = \langle a \xi b ,\xi \rangle$$ is $\otimes_{\rm min}$-continuous. Since this holds for every unit vector $\xi \in \cH$, we conclude that $\cH \prec_M \rL^2(M) \otimes \rL^2(M)$.
\end{proof}

\begin{rem}
It is not clear wether the conclusion of Corollary \ref{cor AO} is equivalent to property $\rm W^*AO$. This is at least true when $M$ is full or when $M_*$ is separable.
\end{rem}

Theorem \ref{W*AO} follows from the following more general result.
\begin{thm} \label{compact binormal state}
Let $M$ be a finite von Neumann algebra. Take $T \in \B(\rL^2(M))$. The following are equivalent :
\begin{enumerate}
\item $T \in \mathbb{K}^L_M$.
\item $\Phi(T^*T)=0$ for every binormal state $\Phi \in \B(\rL^2(M))^*$ such that $\Phi|_{\mathbf{K}(\rL^2(M))}=0$.
\end{enumerate}
\end{thm}
\begin{proof}
Fix a faithful normal trace $\tau$ on $M$.

(1) $\Rightarrow$ (2). Let $A$ be an abelian von Neumann algebra and $\omega$ a cofinal ultrafilter and $\xi \in \rL^2((A \ovt M)^\omega)$ such that $\Phi(S)= \langle (1 \otimes S)^\omega \xi, \xi \rangle$ for all $S \in \B(\rL^2(M))$. Since $\Phi(Q)=0$, we have $(1 \otimes Q)^\omega \xi=0$ for every finite rank projection $Q \in \B(\rL^2(M))$. Let $P$ be the projection from $\rL^2(A \ovt M)^\omega$ onto $L^2(A)^\omega \ovt \rL^2(M)$. Since $P$ is the supremum of $(1 \otimes Q)^\omega$ over all finite rank projections $Q \in \B(\rL^2(M))$, we get $P \xi=0$. 

Fix a faithful normal trace $\tau$ on $M$ and some faithful probability measure $\mu$ on $A$. Recall the definition of the norm
$$ \| S\|_{\Omega_\tau} = \sup \{ \|S\xi\| \mid \xi \in \Omega_\tau \}$$
for $S \in \B(\rL^2(M))$. Then one verifies easily that
$$ \| 1 \otimes S \|_{\Omega_{\mu \otimes \tau}}=\| S\|_{\Omega_\tau}$$
Passing to ultrapowers and using the fact that $(\Omega_{\mu \otimes \tau})^\omega = \Omega_{(\mu \otimes \tau)^\omega}$, we obtain 

we obtain 
$$ \| (1 \otimes S)^\omega \|_{\Omega_{(\mu \otimes \tau)^\omega}} =\| S\|_{\Omega_\tau} .$$

Now, since $T$ is $M$-compact, we know that for every $\varepsilon > 0$, we can find a finite rank projection $Q \in \B(\rL^2(M))$ such that $\|T-QT\|_{\Omega_\tau}  \leq \varepsilon $. We then get
$$ \| (1 \otimes T)^\omega - (1 \otimes Q)^\omega(1 \otimes T)^\omega \|_{\Omega_{(\mu \otimes \tau)^\omega}}  \leq \varepsilon.$$
Since  $(1 \otimes Q)^\omega \leq P$, we obtain 
$$ \| (1 \otimes T)^\omega - P(1 \otimes T)^\omega \|_{\Omega_{(\mu \otimes \tau)^\omega}}  \leq \varepsilon.$$
Since $\varepsilon > 0$ is arbitrary, we get 
$$ (1 \ovt T)^\omega \xi =P(1 \otimes T)^\omega \xi$$
for all $\eta \in \Omega_{(\mu \otimes \tau)^\omega}$, hence for all 
$\eta \in \rL^2((A \ovt M)^\omega)$. In particular,
$$ (1 \ovt T)^\omega \xi =P(1 \otimes T)^\omega \xi.$$
Since $(1 \otimes T)^\omega$ commutes with $P$ and $P \xi=0$, this shows that $(1 \otimes T)^\omega \xi=0$, hence $\Phi(T^*T)=0$.

(2) $\Rightarrow$ (1). Take $(\xi_i)_{i \in I} \subset \Omega_\tau$ a net that converges weakly to $0$. Let $\Phi \in \B(\rL^2(M))^*$ be a weak*-accumulation point of the net $\langle \cdot \xi_i,\xi_i \rangle$. Then $\Phi$ is binormal and vanishes on the compacts. Thus $\Phi(T^*T)=0$. Since this holds for every accumulation point, it means that $\lim_i T \xi_i=0$. We conclude that $T$ is weak-norm continous on $\Omega_\tau$, hence $T(\Omega_\tau)$ is compact.
\end{proof}

\begin{proof}[Proof of Theorem \ref{W*AO}]
For every binormal state $\Phi \in \B(\rL^2(M))^*$ that vanishes on the compacts, let $\pi_{\Phi}$ be the GNS representation of $M \otimes_{\rm max} M^{\op}$ associated to the state $a \otimes b^{\op} \mapsto \Phi(\lambda(a)\rho(b))$. By Theorem \ref{compact binormal state}, we have $\ker \kappa = \bigcap_{\Phi} \ker \pi_{\Phi}$. Therefore, if we let $\pi_{\rm \min} : M \otimes_{\rm max} M^{\op} \rightarrow M \otimes_{\rm min} M^{\op}$, then we have $\ker \pi_{\rm min} \subset \ker \kappa$ if and only if $\ker \pi_{\rm min} \subset \ker \pi_{\Phi}$ for every $\Phi$.
\end{proof}

\begin{rem}
In \cite[Theorem 7.20]{DP23}, Ding and Peterson found an alternative proof of Theorem \ref{W*AO} and Corollary \ref{cor AO} which moreover works for arbitrary (possibly type $\III$) von Neumann algebras! Their proof is based on the machinery of \cite{DKEP23} which allows them to prove Theorem \ref{compact binormal state} for arbitrary von Neumann algebras. Indeed, \cite[Proposition 3.8]{DKEP23} states that $T \in \mathbb{K}_M^L$ if and only if $T^*T \in \mathbb{K}_M^{\infty, 1}$, and \cite[Theorem 5.6 (part 5)]{DKEP23} states that $T^*T \in \mathbb{K}_M^{\infty, 1}$ if and only if $\Phi(T^*T) = 0$ for every binormal state on $\Phi \in \B(\rL^2(M))^*$.
\end{rem}

Finally, we observe that when a von Neumann algebra $M$ has property $(\Gamma)$, we can get rid of the abelian von Neumann algebra $A$ in Theorem \ref{ultrapower implement tracial}.
\begin{thm}
Let $M$ be a finite von Neumann algebra that has property $(\Gamma)$. For any state $\Psi \in \B(\rL^2(M))$, the following are equivalent:
\begin{enumerate}
\item [$(\rm i)$] $\Psi \circ \lambda$ and $\Psi \circ \rho$ are normal states.
\item [$(\rm ii)$] There exists a cofinal ultrafilter $\omega$ on some directed set $I$ and a vector $\xi \in \rL^2(M^\omega)$ such that $\Psi(T)=\langle T^\omega \xi, \xi \rangle$ for all $T \in \B(\rL^2(M))$. 
\end{enumerate} 
\end{thm}
\begin{proof}
We only have to prove $(\rm i) \Rightarrow (\rm ii)$. By Theorem \ref{ultrapower implement tracial}, we can find an abelian von Neumann algebra $A$, a cofinal ultrafilter $\omega$ on some directed set $I$ and a vector $\xi \in \rL^2((A \ovt M)^\omega)$ such that $\Psi(T)=\langle (1 \otimes T)^\omega \xi, \xi \rangle$ for all $T \in \B(\rL^2(M))$. Since $M$ has property $(\Gamma)$, we can take $\eta$ a second cofinal ultrafilter on some directed set such that $M' \cap M^\eta$ contains a copy of $A$. Then we have a copy of $A \ovt M$ inside $M^\eta$. Thus we have a copy of $(A \ovt M)^\omega$ inside the iterated ultrapower $(M^\eta)^\omega=M^{\omega \otimes \eta}$ (see \cite[Section 2.3]{AHHM18}). This means that we can view $\xi$ as a vector in $\rL^2((A \ovt M)^\omega) \subset \rL^2(M^{\omega \otimes \eta})$.
\end{proof}
We obtain the following corollary which answers in particular \cite[Remark 1.2]{IT23} for von Neumann algebras with property $(\Gamma)$.
\begin{cor}
Let $M$ be a finite von Neumann algebra with property $(\Gamma)$ and let $\cH$ be an $M$-bimodule. The following are equivalent.
\begin{enumerate}
\item $\cH \subset \rL^2(M^\omega)$ for some cofinal ultrafilter $\omega$.
\item $\cH$ is in the closure of $\rL^2(M)$ in the Fell topology.
\item $\cH$ is weakly contained in $\rL^2(M)$. 
\end{enumerate}
\end{cor}

\section{The weak Dixmier property} \label{section dixmier}
Let $M$ be a von Neumann algebra. We denote by $\CCP(M)$ the set of all (not necessarily normal) contractive completely positive maps from $M$ into itself, equipped with the topology of pointwise weak* convergence. Then $\CCP(M)$ is a compact convex semigroup in the sense of \cite{Ma19}. We denote by $\UCP(M) \subset \CCP(M)$ the closed convex subsemigroup of unital completely positive maps.

Let $N \subset M$ be an inclusion of von Neumann algebras. We denote by $\cD(N \subset M) \subset \UCP(M)$ the closed convex hull of $\{ \Ad(u) \mid u \in \cU(N) \}$. We call it the \emph{weak Dixmier semigroup}. 

\begin{prop} \label{equivalences dixmier}
Let $N \subset M$ be an inclusion of von Neumann algebras. Then the following are equivalent.
\begin{enumerate}
\item $\cD(N \subset M)$ contains a conditional expectation onto $N' \cap M$.
\item For every $x \in M$, the weak*-closed convex hull of $\{ uxu^* \mid u \in \cU(N) \}$ intersects $N' \cap M$.
\item For every $\xi \in M_*$ such that $\xi|_{N' \cap M}=0$, the norm-closed convex hull of $\{ \xi \circ \Ad(u) \mid u \in \cU(N) \}$ contains $0$.
\end{enumerate}
\end{prop}
\begin{proof}
(1) $\Leftrightarrow$ (2) is clear.

(1) $\Rightarrow$ (3). Take $\Phi \in \cD(N \subset M)$ a conditional expectation onto $N' \cap M$. Then for every $\xi \in M_*$ such that $N' \cap M$, we have $\xi \circ \Phi=0$. This means that $0$ is in the weak closure of the convex hull of $\{ \xi \circ \Ad(u) \mid u \in \cU(N) \}$ which, by the Hahn-Banach theorem, is equal to its norm closure.

(3) $\Rightarrow$ (2). Suppose that for some $x \in M$, the weak*-closed convex hull of $\{ uxu^* \mid u \in \cU(N) \}$ does not intersect $N' \cap M$. Then by the Hahn-Banach theorem, we can find $\xi \in M_*$ such that $\xi|_{N' \cap M}=0$ and $\xi(uxu^*) \geq c$ for some constant $c > 0$ and every $u \in \cU(N)$. This would imply that $\eta(x) \geq c$ for every $\eta$ in the closed convex hull of $\{ \xi \circ \Ad(u) \mid u \in \cU(N) \}$ contradicting (3).
\end{proof}

When the above conditions are satisfied, we say that the inclusion $N \subset M$ has the \emph{weak Dixmier property}. The following theorem generalizes the main result of \cite{Ma19} with a much simpler, more natural and more direct proof, because we no longer need to use the results of \cite{Ma18} on the bicentralizer flow. Instead we combine the minimal idempotent technique of \cite{Ma19} with the Approximation Theorem of \cite[Section 3.2]{SZ99}.

\begin{thm} \label{expectation implies dixmier}
Let $N \subset M$ be an inclusion von Neumann algebras with expectations. Then $N \subset M$ has the weak Dixmier property. If $N$ is of type $\III$, then $\cD(N \subset M)$ contains a faithful family of normal conditional expectations from $M$ onto $N' \cap M$.
\end{thm}
\begin{proof}This is elementary and already known when $N$ is semifinite, so we may assume that $N$ is of type $\III$. We may also assume that $N$ and $N' \cap M$ are countably decomposable.
Take $E \in \cE(M,N)$ and $F \in \cE(N,\cZ(N))$. By the Approximation Theorem of \cite[Section 3.2]{SZ99}, we have $F \in \cD(N)$. Therefore, we can find $\Psi \in \cD(N \subset M)$ such that $\Psi|_N=F$. Take $\Phi_0 \in \cD(N \subset M)$ a minimal idempotant (see \cite{Ma19}) and let $\Phi=\Phi_0 \circ \Psi$. Then we have 
$$E \circ \Phi = \Phi \circ E=\Phi_0 \circ \Psi \circ E=\Phi_0 \circ F \circ E=F \circ E.$$
This shows in particular that $\Phi$ is faithful and by \cite[Proposition 2.8]{Ma19}, we conclude that $\Phi$ is a conditional expectation onto $N' \cap M$ (in fact, it is the unique faithful normal conditional expectation that preserves $F \circ E$).
\end{proof}

The von Neumann tensor product of two non-normal ccp maps is not well-defined in general, unless one of the two maps is normal. For a detailed account on these issue, we refer to \cite{NT81}. For our study, we will only need the following easy proposition.

\begin{prop} \label{tensor product ccp}
Let $A$ and $B$ be two von Neumann algebras and let $f : A \rightarrow B$ be a $\CCP$ map (not necessarily normal). Then for any von Neumann algebra $M$, there exists a unique $\CCP$ map $f \otimes \id  : A \ovt M \rightarrow B \ovt M$ such that the following diagram commutes for every $\varphi \in M_*$:
$$ \begin{tikzcd}
A \ovt M \arrow{r}{f \otimes \id} \arrow[swap]{d}{\id \otimes \varphi} & B \ovt M \arrow{d}{\id \otimes \varphi} \\%
A  \arrow{r}{f }& B 
\end{tikzcd}$$
Moreover, we have the following properties :
\begin{enumerate}
\item If $g \in \CCP(B,C)$, then $(g \circ f) \otimes \id=(g \otimes \id) \circ (f \otimes \id)$.
\item If $f$ is normal then $f \otimes \id$ is also normal.
\item The map $\mathrm{CCP}(A,B) \ni f \mapsto f \otimes \id \in \mathrm{CCP}(A \ovt M, B \ovt M)$ is continuous in the topologies of pointwise weak*-convergence.
\end{enumerate} 
\end{prop}

By directly applying the previous proposition, we obtain the following crossed product construction. Recall from Section \ref{prelim crossed product} that if $\alpha : G \curvearrowright A$ is an action of a locally compact group $G$ on a von Neumann algebra $A$, then we have a natural $G$-equivariant embedding of $A$ into $A \ovt \rL^\infty(G)$ which we denote with the same letter $\alpha$, and that $\alpha$ extends to an embedding of the crossed products $\alpha \rtimes G : A \rtimes_\alpha G \rightarrow A \ovt \B(\rL^2(G))$.

\begin{prop} \label{crossed product ccp}
Let $\alpha : G \curvearrowright A$ and $\beta : G \curvearrowright B$ be actions of a locally compact group $G$ on von Neumann algebras $A$ and $B$. Let $f \in \CCP(A,B)^G$ be a $G$-equivariant $\CCP$ map (not necessarily normal). Then there exists a unique $f \rtimes G  \in \CCP(A \rtimes_\alpha G , A \rtimes_\beta G)$ that makes the following diagram commute
$$ \begin{tikzcd}
A \rtimes_\alpha G \arrow{r}{f \rtimes G} \arrow[swap]{d}{\alpha \rtimes G} & B \rtimes_\beta G \arrow{d}{\beta \rtimes G} \\%
A \ovt \B(\rL^2(G)) \arrow{r}{f \otimes \id}& B \ovt \B(\rL^2(G))
\end{tikzcd}$$
Moreover, we have the following properties :
\begin{enumerate}
\item If $C$ is a third $G$-von Neumann algebras and $g \in \CCP(B,C)^G$, then $(g \circ f) \rtimes G=(g  \rtimes G) \circ (f \rtimes G)$.
\item If $f$ is normal then $f \rtimes G$ is also normal.
\item The map $\mathrm{CCP}(A,B)^G \ni f \mapsto f \rtimes G \in \mathrm{CCP}(A \rtimes_\alpha G, B \rtimes_\beta G)$ is continuous in the topologies of pointwise weak*-convergence.
\end{enumerate} 
\end{prop}

\begin{rem} \label{normal crossed product ucp}
A consequence of item (1) in Proposition \ref{crossed product ccp}, we have $(f \rtimes G)|_{A_0 \rtimes G}=(f|_{A_0}) \rtimes G$ for every $G$-invariant von Neumann subalgebra $A_0 \subset A$. Combining with item (2), it follows that if $f|_{A_0}$ is normal then $(f \rtimes G)|_{A_0 \rtimes G}$ is also normal.
\end{rem}

\begin{prop} \label{dixmier tensor product}
Suppose that an inclusion of von Neumann algebras $A \subset B$ has the weak Dixmier property. Then $A \otimes 1 \subset B \ovt C$ has the weak Dixmier property for every von Neumann algebra $C$.
\end{prop}
\begin{proof}
By Proposition \ref{tensor product ccp}, the tensor product map $ - \otimes \id_C : \CCP(B) \rightarrow \CCP(B \ovt C)$ sends $\cD(A \subset B)$ into $\cD(A \otimes 1 \subset B \ovt C)$ and sends any conditional expectation onto $A' \cap B$ to a conditional expectation onto $(A' \cap B) \ovt C=A' \cap (B \ovt C)$.
\end{proof}

\begin{prop} \label{dixmier in crossed product}
Let $\alpha : G \curvearrowright A$ be an action of a locally compact group $G$ on some von Neumann algebras $A$. Let $P \subset A^\alpha$ be a von Neumann subalgebra. Then $P' \cap (A \rtimes_\alpha G)=(P' \cap A) \rtimes_\alpha G$. Moreover, if $P \subset A$ has the weak Dixmier property, then $P \subset A \rtimes_\alpha G$ has the weak Dixmier property.
\end{prop}
\begin{proof}
By Proposition \ref{crossed product ccp}, the map $ - \rtimes G : \CCP(A)^G \rightarrow \CCP(A \rtimes_\alpha G)$ sends $\cD(P \subset A)$ into $\cD(P  \subset A \rtimes_\alpha G)$ and sends any conditional expectation onto $P' \cap A$ to a conditional expectation onto $(P' \cap A) \rtimes_\alpha G=P' \cap (A \rtimes_\alpha G)$.
\end{proof}

\begin{prop} \label{dixmier amenable extension}
Let $N \subset M$ be an inclusion of von Neumann algebras. Suppose that we have a von Neumann subalgebra $P \subset N$ and a unitary representation $u : G \rightarrow \cU(N)$ such that $(N,P,u)$ is a split $G$-extension. If $P \subset M$ has the weak Dixmier property and $G$ is amenable, then $N \subset M$ has the weak Dixmier property.
\end{prop}
\begin{proof}
Since $G$ is amenable, we can find a conditional expectation $E$ from $M$ onto $\{ u_g \mid g \in G\}' \cap M$ in the closed convex hull of $\{ \Ad(u_g) \mid g \in G \}$.  Let $F \in \cD(P \subset M)$ be a conditional expectation from $M$ onto $P' \cap M$. Then $E \circ F$ is a conditional expectation from $M$ onto $\{u_g \mid g \in G\}' \cap P' \cap M=N' \cap M$ and $E \circ F \in \cD(N \subset M)$.
\end{proof}

\begin{prop} \label{dixmier dual}
Let $A$ be a von Neumann algebra and $\alpha : G \curvearrowright A$ an action of a locally compact abelian group $G$ on $A$ and let $\widehat{\alpha} : \widehat{G} \curvearrowright A \rtimes_\alpha G$ be the dual action. Let $P \subset A$ be a von Neumann subalgebra that is globally invariant under the action $\alpha$. If the inclusion $P \subset A \rtimes_\alpha G$ has the weak Dixmier property then so does the inclusion $P \subset (A \rtimes_\alpha G) \rtimes_{\hat{\alpha}} \widehat{G}$ and the inclusion $P \rtimes_\alpha G \subset (A \rtimes_\alpha G) \rtimes_{\hat{\alpha}} \widehat{G}$.
\end{prop}
\begin{proof}
Since $P$ is fixed by $\widehat{\alpha}$, the inclusion $P \subset (A \rtimes_\alpha G) \rtimes_{\widehat{\alpha}} \widehat{G}$ has the weak Dixmier property by Proposition \ref{dixmier in crossed product}. Thanks to Proposition \ref{dixmier amenable extension}, we conclude that $P \rtimes_\alpha G \subset (A \rtimes_\alpha G) \rtimes_{\hat{\alpha}} \widehat{G}$ also has the weak Dixmier property.
\end{proof}

Let $M$ be a von Neumann algebra. Recall that a weight $\psi \in \cP(M)$ is said to have infinite multiplicity if $M_\psi$ is properly infinite. Following \cite{CT76}, we say that $\psi$ is \emph{dominant} if it has infinite multiplicity and there exists a continuous one-parameter unitary group $u : \R^*_+ \rightarrow \cU(M)$ such that $u_\lambda \psi u_\lambda^*=\lambda \psi$ (such a group $u$ is called $\psi$-scaling group). In that case, $(M,M_\psi,u)$ is a crossed product extension whose dual action is $\sigma^\psi$. Moreover, the crossed product extension  $(M,M_\psi,u)$ is isomorphic to the crossed product extension of $c(M)$ by the trace scaling action $\theta$. 

\begin{prop} \label{dixmier for dominant and core}
Let $N \subset M$ be an inclusion of von Neumann algebras with $\cP(M,N) \neq \emptyset$. Suppose that $N$ is properly infinite and let $\psi$ be a dominant weight on $N$. The following are equivalent :
\begin{enumerate}
\item The inclusion $N \subset c(M)$ has the weak Dixmier property.
\item The inclusion $N_\psi \subset M$ has the weak Dixmier property.
\item The inclusion $N_\psi \subset c(M)$ has the weak Dixmier property.
\end{enumerate}
\end{prop}
\begin{proof}
Take $T \in \cP(M,N)$ and let $\phi=\psi \circ T$. We can then find a trace scaling action $\theta : \R^*_+ \curvearrowright M_\phi$ that restricts to a trace scaling action on $N_\psi$ and such that the inclusion $N \subset M$ is isomorphic to $N_\psi \rtimes_\theta \R^*_+ \subset M
_\phi \rtimes_\theta \R^*_+$. Under this isomorphism $\sigma^\phi$ is identified with the dual action $\widehat{\theta}$ and $c(M)$ is identified with $(M
_\phi \rtimes_\theta \R^*_+) \rtimes_{\sigma^\phi} \R$. It follows from Proposition \ref{dixmier dual} that if $N_\psi \subset M$ has the weak Dixmier property, then $N \subset c(M)$ also has the weak Dixmier property.

Conversely, if $N \subset c(M)$ has the weak Dixmier property, then $c_T(N) \subset c(M) \rtimes_\theta \R^*_+$ also has the weak Dixmier property by Proposition \ref{dixmier dual}. But the inclusion $c_T(N) \subset c(M) \rtimes_\theta \R^*_+$ is isomorphic to the inclusion $N_\psi \subset M$. Hence $N_\psi \subset M$ has the weak Dixmier property.
\end{proof}

\begin{prop} \label{dixmier for no type III1}
Let $N \subset M$ be an inclusion of von Neumann algebras with expectations. Suppose that $N$ has no type $\III_1$ summand. Then $N \subset c(M)$ has the weak Dixmier property.
\end{prop}
\begin{proof}
If $N$ is semifinite, then $N \subset c(N)$ is with expectations, hence $N \subset c(M)$ is also with expectations. Therefore $N \subset c(M)$ has the weak Dixmier property.

If $N$ is of type $\III$ but has no type $\III_1$ summand, then it has a discrete decomposition $N=P \rtimes \Z$ (see for instance Corollary \ref{no III1 semifinite reduction}). By the first case we know that $P \subset c(M)$ has the weak Dixmier property and we conclude by Proposition \ref{dixmier amenable extension}.
\end{proof}

The proof of the following proposition is straightforward if we use the norm criterion from item (3) of Proposition \ref{equivalences dixmier}. We leave the details to the reader.
\begin{prop} \label{desintegration dixmier}
Let $(X,\mu)$ be a standard borel probability space. Let $x \mapsto M_x$ a measurable field of von Neumann algebras with separable predual and $x \mapsto N_x \subset M_x$ a measurable field of von Neumann subalgebras. Let 
$$ M= \int^\oplus_X M_x \:  \rd \mu(x) \quad \text{ and } \quad N=\int^\oplus_X N_x \: \rd \mu(x).$$
Then $N \subset M$ has the weak Dixmier property if and only if $N_x \subset M_x$ has the weak Dixmier property for $\mu$-almost every $x \in X$.
\end{prop}

\section{A class of well-behaved ucp maps} \label{section ucp map}
\subsection{Definition and examples}
Fix $M$ a von Neumann algebra. A completely positive map from $M$ to $M$ is \emph{inner} if it is of the form $x \mapsto \sum_{i } a_i x a_i^*$ for some sequence $(a_i)_{i \in \N}$ of elements of $M$ such that $\sum_{i } a_i a_i^*$ is bounded. Such families will be called \emph{rows} and viewed as infinite row matrices with entries in $M$. They will be denoted with a bold letter $\ba=(a_i)_{i \in \N}$ and we call the elements $a_i$ the \emph{entries} of $\ba$. We define the norm of $\ba$ by $\| \ba \| =\left \|\sum_{i } a_ia_i^* \right \|^{1/2}$. This must not be confused with the $\ell^\infty$-norm $\| \ba \|_\infty = \sup_{i} \| a_i\|$. 

The completely positive map $x \mapsto \sum_{i } a_i x a_i^*$ will be denoted $\Ad(\ba)$. If $\ba$ is a row, then $\ba^*$ is a column. If $\ba$ and $\bb$ are two rows, then $\ba \bb^*=\sum_{i } a_i b_i^*$. We also use the bimodule notation $x\ba y$ when $x,y \in M$ to denote the row $(xa_i y)_{i \in \N}$. Therefore, we have $\Ad(\ba)(x)=\ba x \ba^*=\sum_{i } a_i x a_i^*$. Note that $\Ad(\ba)$ is contractive if and only if $\ba\ba^* \leq 1$ and unital if and only if $\ba \ba^*=1$. Note that $\| \ba \| =\| \ba \ba^* \|^{1/2}$ is equal to the $C^*$-norm of the matrix 
\begin{equation*}
A = 
\begin{pmatrix}
a_0 & a_{1} & a_{2} & \cdots  \\
0 & 0 & 0 & \cdots  \\
0 & 0 & 0 & \cdots  \\
\vdots  & \vdots & \vdots  & \ddots  
\end{pmatrix}.
\end{equation*}
This can be used to obtain various inequalities, such as the Cauchy-Schwarz inequality $\| \ba \bb^*\| \leq \| \ba \| \| \bb \|$.

Let $X$ be a weakly closed subspace of $M$. We define the closed convex subset $\cV(X \subset M) \subset \UCP(M)$ by
$$ \cV(X \subset M)=\UCP(M) \cap \overline{ \{ \Ad(\ba) \in \CCP(M) \mid \ba \text{ is a row with entries in } X \}}.$$
Note that when $X=N$ is a von Neumann subalgebra of $M$, then $\cV(N \subset M)$ is a closed convex subsemigroup of $\UCP(M)$.

We will use the following extension principle all the time.
\begin{prop} \label{extension principle}
Let $M$ be a von Neumann algebra, $N \subset M$ a von Neumann subalgebra and $X \subset N$ a weakly closed subspace. For every $f \in \cV(X \subset N)$, there exists $g \in \cV(X \subset M)$ such that $f=g|_N$.
\end{prop}
\begin{proof}
Write $f=\lim_i \Ad(\ba_i)$ in $\cV(X \subset N)$ for some net of rows $(\ba_i)_{i \in I}$ with entries in $X$ and $\| \ba_i\| \leq 1$. Now view each $\Ad(\ba_i)$ as an element of $\cV(X \subset M)$ and take $g$ an accumulation point of $(\Ad(\ba_i))_{i \in I}$ in the compact space $\cV(X \subset M)$. 
\end{proof}

Let $\varphi, \psi \in \cP(M)$. For every $\varepsilon > 0$, define $\cV^{\psi,\varphi}_\varepsilon(M):=\cV(X \subset M)$ with $X=M(\sigma^{\psi,\varphi}, [-\varepsilon,\varepsilon])$.

Let $\cV^{\psi,\varphi}(M)= \bigcap_{\varepsilon > 0} \cV^{\psi,\varphi}_\varepsilon(M)$. Observe that for every $\varepsilon > 0$ and any triple of $\varphi, \phi, \psi \in \cP(M)$ we have $$\cV^{\psi,\phi}_\varepsilon(M) \cV^{\phi,\varphi}_\varepsilon(M) \subset \cV^{\psi,\varphi}_{2\varepsilon}(M).$$ This means that $$\cV^{\psi,\phi}(M) \cV^{\phi,\varphi}(M) \subset \cV^{\psi,\varphi}(M).$$ In particular, $\cV^{\varphi}(M):=\cV^{\varphi,\varphi}(M)$ is a closed convex subsemigroup of $\mathrm{UCP}(M)$.

If $N$ is von Neumann subalgebra of $M$ and $\varphi, \psi \in \cP(N)$, we define similarly $\cV^{\psi,\varphi}(N \subset M)$ by taking $X=N(\sigma^{\psi,\varphi}, [-\varepsilon,\varepsilon])$.

Let $M_\infty=\B(H) \ovt M$ where $H$ is a separable Hilbert space of infinite dimension. If $f \in \CCP(M_\infty)$ and $e \in \B(H)$ is a rank one projection, we define $f_e \in \CCP(M)$ by 
$$\forall x \in M, \quad e \otimes f_e(x)=(e\otimes 1)f(1 \otimes x) (e\otimes 1).$$
The map $f \mapsto f_e$ is continuous from $\CCP(M_\infty)$ onto $\CCP(M)$ and sends unital maps to unital maps. It is straighforward to check that if $f$ is inner, then $f_e$ is also inner. More precisely, if $f=\Ad(\ba)$ for some row $\ba$ with entries in $M_\infty$, then $f_e=\Ad(\ba')$ for some row $\ba'$ such that each entry of $\ba'$ is a matrix coefficient of an entry of $\ba$. Therefore, if $X$ is a weakly closed subspace of $M$ and $f \in \cV(X_\infty \subset M_\infty)$ where $X_\infty = \B(H) \ovt X$, then $f_e \in \cV(X \subset M)$. In particular, if we define $\varphi_\infty=\mathrm{Tr} \otimes \varphi \in \cP(M_\infty)$ for every $\varphi \in \cP(M)$, we obtain the following proposition.

\begin{prop} \label{amplification ucp map}
Let $N \subset M$ be an inclusion of von Neumann algebra. Take $\varphi, \psi \in \cP(N)$. Take $f \in \cV^{\psi_\infty,\varphi_\infty}(N_\infty \subset M_\infty)$ and $e \in \B(H)$ a rank one projection. Then $f_e \in \cV^{\psi,\varphi}(N \subset M)$.
\end{prop}

Here is a fundamental example that motivates all these lengthy definitions. This crucial observation will be the only tool at our diposal to connect the bicentralizer algebra of a state to the bicentralizer algebra of a dominant weight.

\begin{prop} \label{transition UCP weights}
Let $M$ be a von Neumann algebra of type $\III_1$. Take $\varphi, \psi \in \cP(M)$. Then $\cV^{\varphi,\psi}(M)$ is non-empty.
\end{prop}
\begin{proof}
By Proposition \ref{amplification ucp map}, it is enough to show that $\cV^{\psi_\infty,\varphi_\infty}( M_\infty)$ is non-empty. By \cite[Corollary 4.8]{CT76}, since the weights $\psi_\infty$ and $\varphi_\infty$ are of infinite multiplicity and $M$ is of type $\III_1$, there exists a sequence of unitaries $u_n \in \mathcal{U}(M_\infty)$ such that $\varepsilon_n =d_{\mathrm{CT}}(\psi_\infty,u_n\varphi_\infty u_n^*) \to 0$ where $d_{\mathrm{CT}}$ is the Connes-Takesaki distance. Then $u_n \in M_\infty(\sigma^{\psi_\infty,\varphi_\infty}, [-\varepsilon_n,\varepsilon_n])$. Indeed, by definition of the Connes-Takesai distance, $d_{\mathrm{CT}}(u\psi u^*,\varphi)=\varepsilon$ means that the map $t \mapsto \psi^{\ri t} u \varphi^{-\ri t} u^* = \sigma_t^{\psi,\varphi}(u)u^*$ extends to an entire analytic function on $\C$ whose norm is bounded by $e^{\varepsilon |\mathrm{Im}(z)|}$ and this means by definition that $u \in M(\sigma^{\psi,\varphi}, [-\varepsilon,\varepsilon])$.

We conclude that any accumulation point of $\Ad(u_n)$ is in $\cV^{\psi_\infty,\varphi_\infty}( M_\infty)$.
\end{proof}

We will also need the following example. See the Preliminaries \ref{prelim modular theory} for the definition of $\sigma^{\psi,\varphi}$ and $N_{\psi,\varphi}$.
\begin{prop} \label{ucp map from ultrapower}
Let $N \subset M$ be an inclusion with expectations. Take $\varphi, \psi \in \cP(M)$ and $u$ a unitary in $N^\omega_{\psi^{\omega},\varphi^{\omega}}$ for some cofinal ultrafilter $\omega$. Then the map $f \in \UCP(M)$ defined by $f(x)=E^\omega(uxu^*)$ belongs to $\cV^{\psi,\varphi}(N \subset M)$.
\end{prop}
\begin{proof}
Take $\varepsilon > 0$. Let $h$ be a positive function in $\rL^1(\R)$ with $\| h\|_1=1$ such that $\widehat{h}$ is supported on $[-\varepsilon,\varepsilon]$ (for example take the \emph{Fej\` er Kernel}). Write $u=(u_i)^\omega$ for some net of unitaries $(u_i)_i$ in $\cM^\omega(N)$. Then $u=\sigma_h^{\psi^\omega,\varphi^\omega}(u)=(a_i)^\omega$ where $a_i=\sigma_h^{\psi,\varphi}(u_i) \in N(\sigma^{\psi,\varphi},[-\varepsilon,\varepsilon])$. Then for all $x \in M$, we have $f(x)=\lim_{i \to \omega} a_ixa_i^*$ where the limit is taken in the weak*-topology. Moreover, $\|a_i\| \leq 1$ for all $i \in I$. We conclude that $f \in \cV^{\psi,\varphi}_\varepsilon(N \subset M)$, and since $\varepsilon > 0$ is arbitrary, $f \in \cV^{\psi,\varphi}(N \subset M)$.
\end{proof}

\subsection{Fundamental properties}
\begin{lem}
Let $M$ be a von Neumann algebra. Let $\varphi,\psi \in \cP(M)$.  Take $\varepsilon > 0$ and $a \in M(\sigma^{\psi,\varphi}, [-\varepsilon,\varepsilon])$. Then for all $t \in \R$, we have $\| \sigma^{\psi,\varphi}_t(a)-a\| \leq K\varepsilon |t| \| a\|$ for some universal constant $K > 0$.
\end{lem}
\begin{proof}
Take $h \in \rL^1(\R)$ with $h' \in \rL^1(\R)$ and such that its fourier transform $\widehat{h}$ is constant equal to $1$ on $[-1,1]$ (for example take the \emph{de la Vall\' ee Poussin Kernel}). For every $\varepsilon > 0$ let $h_\varepsilon=\varepsilon h(\varepsilon \cdot)$. Then $\widehat{h_\varepsilon}= \widehat{h}( \varepsilon^{-1} \cdot)$ is constant equal to $1$ on $[-\varepsilon,\varepsilon]$. Therefore, we have $ a=\sigma_{h_\varepsilon}^{\psi,\varphi}(a)$. Take $t \in \R$ and let $\tau_t$ be the translation operator on $\rL^1(\R)$. Then $\sigma_t^{\psi,\varphi}(a)=\sigma_{\tau_t(h_\varepsilon)}^{\psi,\varphi}(a)$. We thus have 
$$ \|\sigma_t^{\psi,\varphi}(a)-a\| \leq \| \tau_t(h_\varepsilon)-h_\varepsilon \|_1 \|a\|=\| \tau_{t\varepsilon}(h)-h\|_1 \|a\|.$$
Now, since $h' \in \rL^1(\R)$, we have $\| \tau_s(h)-h\|_1 \leq |s| \| h' \|_1$ for every $s \in \R$. We conclude that 
 $$\|\sigma_t^{\psi,\varphi}(a)-a\| \leq K\varepsilon |t| \|a\| $$
 where $K=\| h'\|_1$.
\end{proof}

\begin{lem}
Let $M$ be a von Neumann algebra. Let $\varphi,\psi \in \cP(M)$.  Take $\varepsilon > 0$ and $\ba$ a row in $ M(\sigma^{\psi,\varphi}, [-\varepsilon,\varepsilon])$. Then for all $t \in \R$, we have $\| \sigma_t^{\psi,\varphi}(\ba)-\ba\| \leq K \varepsilon |t| \| \ba\| $.
\end{lem}
\begin{proof}
Apply the previous lemma to the following matrix and the weights $\Phi=\varphi \otimes {\rm Tr}$ and $\Psi=\psi \otimes {\rm Tr}$ on $M \ovt \B(\ell^2(\N))$.
\begin{equation*}
A = 
\begin{pmatrix}
a_0 & a_{1} & a_{2} & \cdots  \\
0 & 0 & 0 & \cdots  \\
0 & 0 & 0 & \cdots  \\
\vdots  & \vdots & \vdots  & \ddots  
\end{pmatrix}.
\end{equation*}
\end{proof}

\begin{prop}
Let $M$ be von Neumann algebra with $\varphi, \psi \in \cP(M)$. Take $f \in \cV^{\psi,\varphi}(M)$. Then $f \circ \sigma^\varphi_h = \sigma_h^\psi \circ f$ for every $h \in \rL^1(\R)$.
\end{prop}
\begin{proof}
Take $\varepsilon > 0$ and $\ba$ a row in $M(\sigma^{\psi,\varphi}, [-\varepsilon,\varepsilon])$ with $\| \ba \| \leq 1$. Then for every $x \in M$, we have
$$ \sigma_t^\psi( \ba x \ba^*)=\sigma_t^{\psi,\varphi}(\ba)\sigma_t^\varphi(x) \sigma_t^{\psi,\varphi}(\ba)^*.$$
Therefore 
$$ \sigma_t^\psi( \ba x \ba^*)- \ba \sigma_t^\varphi(x) \ba^*=(\sigma_t^{\psi,\varphi}(\ba)-\ba )\sigma_t^\varphi(x) \sigma_t^{\psi,\varphi}(\ba)^*+\ba \sigma_t^\varphi(x) (\sigma_t^{\psi,\varphi}(\ba)-\ba)^*$$
hence 
$$ \| \sigma_t^\psi( \ba x \ba^*)- \ba \sigma_t^\varphi(x) \ba^* \| \leq \| \sigma_t^{\psi,\varphi}(\ba)-\ba \| \|x\| \|\ba\| + \|\ba\| \|x\| \| \sigma_t^{\psi,\varphi}(\ba)-\ba \| \leq 2K\varepsilon |t|.$$
This shows that if $h$ is supported in $[-T,T]$, then
$$ \| \sigma_h^\psi(\ba x \ba^*)-\ba \sigma_h^\varphi(x)\ba^* \| \leq 2K \varepsilon T \|h\|_1.$$
By letting $\varepsilon > 0$ converge to $0$, we obtain $ \sigma_h^\psi(f(x))=f( \sigma_h^\varphi(x))$. Since functions with compact support are dense in $\rL^1(\R)$, we get the same conclusion for all $h \in \rL^1(\R)$.
\end{proof}

\begin{thm} \label{extension ucp to core}
Let $M$ be a von Neumann algebra. Let $\varphi, \psi \in \cP(M)$. Take $f \in \cV^{\psi,\varphi}(M)$. Then there exists a unique $c(f) \in \cV^{\psi,\varphi}(M \subset c(M))$ such that $f=c(f)|_M$. Moreover, we have a commutative diagram
$$ \begin{tikzcd}
c(M) \arrow{r}{c(f)}
 \arrow[swap]{d}{ \pi_\varphi } & c(M) \arrow{d}{\pi_\psi} \\%
M \rtimes_{\sigma^\varphi} \R \arrow{r}{f \rtimes \R} & M \rtimes_{\sigma^\psi} \R
\end{tikzcd}$$
and the map $$\cV^{\psi,\varphi}(M)\ni f \mapsto c(f) \in \cV^{\psi,\varphi}(M \subset c(M))$$ is a homeomorphism that sends normal maps to normal maps.
\end{thm}
\begin{proof}
Take $f \in \cV^{\psi,\varphi}(M)$. Thanks to Proposition \ref{extension principle}, there exists $g \in \cV^{\psi,\varphi}(M \subset c(M))$ such that $g|_M=f$. What we need to prove is the uniqueness of $g$. For this, it is sufficient to show that the following diagram commutes (for the bottom square, this is already known by Proposition \ref{crossed product ccp}).
$$ \begin{tikzcd}
c(M) \arrow{r}{g}
 \arrow[swap]{d}{ \pi_\varphi } & c(M) \arrow{d}{\pi_\psi} \\%
M \rtimes_{\sigma^\varphi} \R \arrow{r}{f \rtimes \R}  \arrow[swap]{d}{ \sigma_\varphi \rtimes \R }  & M \rtimes_{\sigma^\psi} \R \arrow{d}{\sigma_\psi \rtimes \R}  \\%
M \ovt \B(\rL^2(\R)) \arrow{r}{f \otimes \id} & M \ovt \B(\rL^2(\R))
\end{tikzcd}.$$
We first set up the convenient notations. We view $M$ as a subalgebra of $\B(\rL^2(M))$, by identifying it with $\lambda(M)$, and we view $M \ovt \B(\rL^2(\R)) $ as a subalgebra of $\B(\rL^2(M) \otimes \rL^2(\R))$. In this way, we can view $\kappa_\varphi= (\sigma^\varphi \rtimes \R) \circ \pi_\varphi$ and $\kappa_\psi = (\sigma^\psi \rtimes \R) \circ \pi_\psi$ as two representations of $c(M)$ on $\rL^2(M) \otimes \rL^2(\R)$. We consider a third representation $\kappa$ of $c(M)$ on the same hilbert space given by 
\begin{itemize}
\item $\kappa(x)=x \otimes 1$ for all $x \in M$
\item $\kappa(\varphi^{\ri t})=\Delta_\varphi^{\ri t} \otimes \lambda_t$ for all $t \in \R$
\item $\kappa(\psi^{\ri t})=\Delta_{\psi,\varphi}^{\ri t} \otimes \lambda_t$ for all $t \in \R$.
\end{itemize}
Consider also the two desintegrable (they belong to $\B(\rL^2(M)) \ovt \rL^\infty(G)$) unitary operators $U : t \mapsto \Delta_\varphi^{\ri t}$ and $V : t \mapsto \Delta_{\psi,\varphi}^{\ri t}$. Then we have the following relations :
\begin{align}
\label{eq:intertwining}  \kappa_\varphi = \Ad(U) \circ \kappa \quad \text{and} \quad \kappa_\psi=\Ad(V) \circ \kappa .
\end{align}
Now, remember that our goal is to show that 
\begin{align}
\label{eq:goal} (f \otimes \id) \circ \kappa_\varphi = \kappa_\psi \circ g.
\end{align}
Take $(\ba_i)_{i \in I}$ a net of rows in $M(\sigma^{\psi,\varphi},[-\varepsilon_i,\varepsilon_i])$ with $\varepsilon_i \to 0$ such that $g(x)=\lim_i \ba_i x \ba_i^*$ weakly for all $x \in c(M)$. Up to extracting a subnet, we may assume that $T \mapsto \ba_i T \ba_i^*$ also converges weakly on $\B(\rL^2(M))$ to some ucp map $h \in \cV^{\psi,\varphi}(M \subset \B(\rL^2(M))$. Then for all $x \in c(M)$, we have 
 \begin{align}
\label{eq: kappa relation} \kappa(g(x))=\lim_i \kappa(\ba_i x \ba_i^*)=\lim_i (\ba_i \otimes 1)\kappa(x) (\ba_i^* \otimes 1)= (h \otimes \id)(\kappa(x)) 
\end{align}

\begin{claim}
We have $(h \otimes \id)(U)=V$.
\end{claim}
\begin{proof}[Proof of the claim]
It is sufficient to prove that $(\id \otimes \omega)(( h \otimes \id)(U))=(\id \ot \om)(V)$ for every $\omega \in \B(\rL^2(\R))_*$. By definition of $h \otimes \id$, we have $(\id \otimes \omega)( h \otimes \id)(U)=h( (\id \otimes \omega)(U))$. The restriction of $\omega$ to $\rL^\infty(G) \subset \B(\rL^2(G))$ is given by integration with respect to some $m \in \rL^1(G)$. Then we have
$$ (\id \ot \om)(U)=\int \Delta^{\ri t}_\varphi \: m(t) \rd t \quad  \text{ and }\quad  (\id \ot \om)(V)=\int \Delta^{\ri t}_{\psi,\varphi} \: m(t) \rd t .$$
Now, we can compute $h((\id \ot \om)(U))$ by using the relation $\Delta_{\psi,\varphi}^{\ri t} \ba_i \Delta_{\varphi}^{-\ri t} =\sigma^{\psi,\varphi}_t(\ba_i)$ which yields
$$\| \ba_i \Delta_\varphi^{\ri t} \ba_i^*-\Delta_{\psi,\varphi}^{\ri t}\ba_i\ba_i^* \| = \| \sigma_{-t}^{\psi,\varphi}(\ba_i)\ba_i^*-\ba_i\ba_i^*\| \leq \| \sigma_{-t}^{\psi,\varphi}(\ba_i)-\ba_i \| \| \ba_i\| \leq K \varepsilon_i |t| \| \ba_i \|^2.$$
Since $\| \ba_i \| \leq 1$ and $\ba_i\ba_i^*$ converges strongly to $1$, we obtain the strong convergence of 
$$ \int \ba_i \Delta_\varphi^{\ri t} \ba_i^* \: m(t) \rd t \rightarrow \int \Delta_{\psi,\varphi}^{\ri t} \: m(t) \rd t$$
whenever $m$ is compactly supported and this convergence can be extended to every $m \in \rL^1(G)$ by approximation. It follows that $h( (\id \otimes \omega)(U))=(\id \otimes \omega)(V)$ and this ends the proof of the claim.
\end{proof}
The claim tells us in particular that $U$ is in the multiplicative domain of $h \otimes \id$ hence $(h \ot \id) \circ \Ad(U)=\Ad(V) \circ (h \ot \id)$. Then by using \ref{eq:intertwining} and \ref{eq: kappa relation}, we can finally prove our goal \ref{eq:goal} as follows :
\begin{align*}
(f \otimes \id) \circ \kappa_\varphi & = (h \otimes \id) \circ \kappa_\varphi \\
&= (h \otimes \id) \circ \Ad(U) \circ \kappa \\
&= \Ad(V) \circ (h \otimes \id) \circ \kappa \\
&= \Ad(V) \circ \kappa \circ g \\
& = \kappa_\psi \circ g.
\end{align*}
\end{proof}

Finally, we show how to use the ucp maps we defined in this section to obtain binormal states satisfying the assumptions of Theorem \ref{ultrapower all type}.
\begin{prop} \label{binormal state vs ucp}
Let $N \subset M$ be an inclusion of von Neumann algebra with a faithful normal conditional expectation $E_N$. Let $\varphi$ be a faithful normal state on $M$ with $\varphi = \varphi \circ E_N$ and $\psi \in \cP(M)$ with $\psi=\psi \circ E_N$. Take $f \in \cV^{\varphi,\psi}(N \subset M)$. Then there exists a state $\Phi \in \B(\rL^2(M))^*$ such that:
\begin{itemize}
\item For all $a,b \in M$, $\Phi( \lambda(a)\rho(b)) = \langle a\varphi^{1/2}f(b),\varphi^{1/2}\rangle$.
\item  $\Phi(e_N)=1$ where $e_N \in \B(\rL^2(M))$ is the Jones projection associated to $E_N$.
\item $\Phi$ is a $1$-eigenstate of $\Delta_{\varphi,\psi}$. 
\end{itemize}
\end{prop}
\begin{proof}
Take $f_0$ an inner contractive completely positive map of the form $f_0(x)=\sum_{i \in \N} a_ixa_i^*$ with $a_i \in N(\sigma^{\varphi,\psi}, [-\varepsilon,\varepsilon])$ for some $\varepsilon > 0$. Let $\xi_i=\varphi^{1/2}a_i$ and define $\Phi_0 \in \B(\rL^2(M))^*_+$ by $\Phi_0(T)=\sum_{i \in \N} \langle T \xi_i, \xi_i \rangle$. The following properties are satisfied :
\begin{itemize}
\item $\Phi_0(1)=\varphi(f_0(1))$.
\item For all $a,b \in M$, $\Phi_0( \lambda(a)\rho(b)) = \langle a\varphi^{1/2}f_0(b),\varphi^{1/2}\rangle$.
\item  $\Phi_0(e_N)=\Phi_0(1)$ where $e_N \in \B(\rL^2(M))$ is the Jones projection associated to $E_N$.
\item $\Phi_0(1_{[e^{-\varepsilon},e^\varepsilon]}(\Delta_{\varphi,\psi}))=\Phi_0(1)$.
\end{itemize}
By definition, $f$ is a pointwise weak*-limit of ccp maps of the form $f_0$ with $\varepsilon$ arbitrarily small. Then any accumulation point of the corresponding $\Phi_0$ will yield a state $\Phi$ satisfying the desired properties (see Remark \ref{remark eigenstate} for the last property).
\end{proof}

\section{The algebraic bicentralizer} \label{section algebraic}
 \subsection{Definition and basic properties}
 
 Recall from Section \ref{prelim modular theory} that we view a weight $\varphi \in \cP(M)$ as a positive operator affiliated with $c(M)$, hence we denote by $\{ \varphi \}''=\{ \varphi^{\ri t} \mid t \in \R \}''$ the von Neumann subalgebra of $c(M)$ that it generates.

\begin{prop} \label{independent of operator valued weight}
Let $N \subset M$ be an inclusion of von Neumann algebras. Take $\psi \in \cP(N)$ and $T \in \cP(M,N)$. The subalgebra $$(N' \cap c(M)) \vee \{ \psi \circ T \}'' \subset c(M)$$ depends only on $\psi \in \cP(N)$, not on the choice of $T \in \cP(M,N)$.
\end{prop} 
\begin{proof}
Take $S \in \cP(M,N)$. Then $(\psi \circ S)^{\ri t} (\psi \circ T)^{-\ri t} \in N' \cap c(M)$. Thus  $\{ \psi \circ S\}'' \subset (N' \cap c(M)) \vee \{ \psi \circ T \}'' \subset c(M)$.
\end{proof}
 
\begin{df}
Let $N \subset M$ be an inclusion of von Neumann algebras with $\cP(M,N) \neq \emptyset$. Take $\psi \in \cP(N)$. We define the \emph{algebraic bicentralizer} of $\psi$ inside $M$ by the formula
$$\rb(N \subset M, \psi)=\left( (N' \cap c(M)) \vee \{ \psi \circ T \}'' \right) \cap M$$
 where $T \in \cP(M,N)$.

When $N=M$, we use the simple notation $\rb(M,\psi)$.
\end{df}

To compute $\rb(N \subset M, \psi)$, we need the following lemma.
\begin{prop} \label{algebraic generation}
Let $N \subset M$ be an inclusion of von Neumann algebras with $T \in \cP(M,N)$. Take $\psi \in \cP(N)$.
\begin{enumerate}
\item $\rb(N \subset M,\psi)$ is globally $\sigma^{\psi \circ T}$-invariant and under the identification $c(M)=M \rtimes_{\sigma^{\psi \circ T}} \R$, we have 
$$(N' \cap c(M)) \vee \{ \psi \circ T \}'' = \rb(N \subset M,\psi) \rtimes_{\sigma^{\psi \circ T}} \R.$$
\item If $A \subset M$ is a globally $\sigma^{\psi \circ T}$-invariant subalgebra such that :
$$  (N' \cap c(M)) \vee \{ \psi \circ T \}'' \subset A \vee \{ \psi \circ T \}'' $$
then $\rb(N \subset M,\psi) \subset A$.
\end{enumerate}
\end{prop} 
\begin{proof}
(1) $(N' \cap c(M)) \vee \{ \psi \circ T \}''$ is globally invariant by the trace scaling action $\theta : \R^*_+ \curvearrowright c(M)$.  Moreover, we have $\theta_\lambda ((\psi \circ T)^{\ri t})=\lambda^{-\ri t}(\psi \circ T)^{\ri t}$ for all $(\lambda,t) \in \R^*_+ \times \R$. Therefore, by Theorem \ref{dual action crossed product}, we know that $(N' \cap c(M)) \vee \{ \psi \circ T \}''$ is a crossed product extension of $$((N' \cap c(M)) \vee \{ \psi \circ T \}'')^{\theta} = ((N' \cap c(M)) \vee \{ \psi \circ T \}'') \cap M= \rb(N \subset M,\psi).$$  by the unitary group $((\psi \circ T)^{\ri t})_{t \in \R}$.

(2) Since $A$ is globally $\sigma^{\psi \circ T}$-invariant, then under the identification $c(M)=M \rtimes_{\sigma^{\psi \circ T}} \R$ we have 
$$ A \vee \{ \psi \circ T \}'' = A \rtimes_{\sigma^{\psi \circ T}} \R.$$
Therefore we have
$$ \rb(N \subset M,\psi) \rtimes_{\sigma^{\psi \circ T}} \R \subset  A \rtimes_{\sigma^{\psi \circ T}} \R.$$
Taking the $\theta$-fixed points, we conclude that $ \rb(N \subset M,\psi) \subset  A$.
\end{proof}

\begin{rem}
Item (1) of Proposition \ref{algebraic generation} raises the following natural question. Is the restriction $\sigma^{\psi \circ T}|_{\rb(N \subset M,\psi)}$ equal to the modular automorphism group of some weight on $\rb(N \subset M,\psi)$? We will answer this question affirmatively (Theorem \ref{modular group of bicentralizer}) under the assumption that $N \subset M$ is with expectations. Unfortunately, we could not find a direct proof of this that does not rely on explicit computations of the bicentralizer algebra.
\end{rem}

\begin{prop} \label{algebraic bicentralizer semifinite}
Let $N \subset M$ be an inclusion of von Neumann algebras with $\cP(M,N) \neq \emptyset$. Take $\varphi \in \cP(N)$.
\begin{enumerate}
\item Suppose that $N$ is semifinite and choose a trace $\tau \in \cP(N)$. Then $\rb(N \subset M, \varphi)$ is the von Neumann algebra generated by $N' \cap M$ and $\frac{\rd \varphi}{\rd \tau}$.
\item Suppose that $M$ is semifinite and choose a trace $\tau \in \cP(M)$.Then $\rb(N \subset M, \varphi)$ is the von Neumann algebra generated by $N' \cap M$ and $\frac{\rd (\varphi \circ T)}{\rd \tau}$ where $T \in \cP(M,N)$.
\end{enumerate}
\end{prop}
\begin{proof}
(1) Pick $T \in \cP(M,N)$. Then $N' \cap c(M)=(N' \cap M) \vee \{ \tau \circ T \}'' $. Therefore, if $h=\frac{\rd\varphi}{\rd \tau}$, we have 
$$(N' \cap c(M)) \vee  \{ \varphi \circ T \}'' =(N' \cap M) \vee \{ \tau \circ T \}'' \vee \{ \varphi \circ T \}'' = (N' \cap M) \vee \{ h \}'' \vee  \{ \varphi \circ T \}''$$
hence $\rb(N \subset M,\varphi)=(N' \cap M) \vee \{ h \}'' $ by Proposition \ref{algebraic generation}.

(2) Write $c(M)=M \vee \{ \tau\}'' \cong M \ovt \{ \tau\}''$. Then $N' \cap c(M)=(N' \cap M) \vee \{ \tau\}''$. Therefore, if we let $h= \frac{ \rd (\varphi \circ T)}{\rd \tau}$, we have
$$ (N' \cap c(M)) \vee \{ \varphi \circ T \}'' =(N' \cap M) \vee \{ \tau \}'' \vee \{ \varphi \circ T \}''= (N' \cap M) \vee \{ h \}'' \vee \{ \varphi \circ T \ \}''$$ hence 
$\rb(N \subset M,\varphi)=(N' \cap M) \vee \{ h\}''$ by Proposition \ref{algebraic generation}. 
\end{proof}

Using the previous proposition, we can identify the subalgebra of Proposition \ref{independent of operator valued weight}.
\begin{prop} \label{alg bicentralizer of core}
Let $N \subset M$ be an inclusion of von Neumann algebras with $\cP(M,N) \neq \emptyset$. Take $\psi \in \cP(N)$ and $T \in \cP(M,N)$.
\begin{enumerate}
\item $\rb(N \subset c(M), \psi)=(N' \cap c(M)) \vee \{ \psi \circ T \}''$ where $T \in \cP(M,N)$.
\item Under the identification $c(M) =M\rtimes_{\sigma^{\psi \circ T} }\R$, we have $$\rb(N \subset c(M), \psi)=\rb(N \subset M,\psi) \rtimes_{\sigma^{\psi \circ T}} \R.$$
\end{enumerate}
\end{prop}
\begin{proof}
(1) Let $\tau$ be the canonical trace on $c(M)$ and let $I^\theta \in \cP(c(M),M)$ be the Haar operator valued weight associated to the scaling action. Then by definition, viewed as an positive operator affiliated with $c(M)$, $\psi \circ T$ is simply the Radon-Nikodym derivative of the dual weight $\widehat{\psi \circ T}=\psi \circ T \circ I^\theta$ with respect to $\tau$. We conclude that $\rb(N \subset c(M), \psi)=(N' \cap c(M)) \vee \{ \psi \circ T \}''$ by Proposition \ref{algebraic bicentralizer semifinite}. Item (2) follows from Proposition \ref{algebraic generation}.
\end{proof}

\begin{prop} \label{alg bic in commutant}
Let $N \subset M$ be an inclusion of von Neumann algebras with $\cP(M,N) \neq \emptyset$. Take $\psi \in \cP(N)$. Then $\rb(N \subset M, \psi) \subset N_\psi ' \cap M$.
\end{prop}
\begin{proof}
Take $T \in \cP(M,N)$. Since  $N' \cap c(M)$ and $ \{ \psi \circ T \}''$ both commute with $N_\psi$, it is clear that $\rb(N \subset M,\psi)$ also commutes with $N_\psi$.
\end{proof}

\begin{prop} \label{alg bic corner reduction}
Let $N \subset M$ be an inclusion of von Neumann algebras with $\cP(M,N) \neq \emptyset$. Take $\psi \in \cP(N)$. If $e \in N_\psi$ is a nonzero projection then $$e\rb(N \subset M, \psi)=\rb(eNe \subset eMe, \psi_e)$$
where $\psi_e=\psi|_{eNe}$.
\end{prop}
\begin{proof}
Take $T \in \cP(M,N)$ and let $S=T|_{eMe} \in \cP(eMe,eNe)$. 
The corner $ec(M)e$ is naturally identified with $c(eMe)$. The map $x \mapsto ex$ then defines a normal *-morphism from $\{e\}' \cap c(M)$ into $c(eMe)$. This *-morphism sends $N' \cap c(M)$ onto $e(N' \cap c(M))=(eNe)' \cap (ec(M)e)=(eNe)' \cap c(eMe)$. It also sends $\{ \psi \circ T \}''$ onto $\{ \psi_e \circ S \}''$. Therefore it sends $\rb(N \subset c(M), \psi)$ onto $\rb(eNe \subset c(eMe), \psi_e)$. In other words, $e\rb(N \subset c(M), \psi)=\rb(eNe \subset c(eMe), \psi_e)$. Taking the intersection with $M$, we conclude that $\rb(eNe \subset eMe, \psi_e)=e\rb(N \subset M, \psi)$.
\end{proof}

\begin{prop} \label{alg bic commutant corner reduction}
Let $N \subset M$ be an inclusion of von Neumann algebras with expectations. Take $\psi \in \cP(N)$. If $e \in N' \cap M$ is a nonzero projection then $$e\rb(N \subset M, \psi)e=\rb(eNe \subset eMe, \psi_e)$$
where $\psi_e \in \cP(eNe)$ is defined by $\psi_e(xe)=\psi(xz)$ for all $x \in N$ and $z$ is the smallest projection of $\cZ(N)$ that contains $e$.
\end{prop}
\begin{proof}
Take some $T \in \cP(M,N)$ such that $e$ is fixed by $\sigma^T$ (use \cite[Theorem 6.6]{Ha77b}). Define $S \in \cP(eMe,Ne)$ by $S(x)=T(x)e$ for all $x \in eMe$. We have $\psi_e \circ S=(\psi \circ T)_e$. 
The corner $ec(M)e$ is naturally identified with $c(eMe)$. Since $$e(N' \cap c(M))e=(eNe)' \cap (ec(M)e)=(eNe)' \cap c(eMe)$$ and $e$ commutes with $\{ \psi \circ T \}''$, we see that 
$$e\{ x (\psi \circ T)^{\ri t } \mid x \in N' \cap c(M), \; t \in \R \}e=\{ y (\psi_e \circ S)^{\ri t } \mid y \in (eNe)' \cap c(eMe), \; t \in \R \}.$$
By taking linear spans and weak*-closures, we get $e\rb(N \subset c(M), \psi)e=\rb(eNe \subset c(eMe), \psi_e)$. Taking the intersection with $M$, we conclude that $\rb(eNe \subset eMe, \psi_e)=e\rb(N \subset M, \psi)$.
\end{proof}

\begin{prop} \label{alg bicentralizer for intermediate subalgebra}
Let $P \subset N \subset M$ be inclusions of von Neumann algebras with $\cP(M,N) \neq \emptyset$ and $\cP(N,P) \neq \emptyset$. Take $\psi \in \cP(P)$. 
\begin{enumerate}
\item  $\rb(P \subset N, \psi)  \subset \rb(P \subset M,\psi)$. 
\item If $S \in \cP(N,P)$ then $\rb(N \subset M, \psi \circ S) \subset \rb(P \subset M, \psi)$
\item Assume that $T|_{P' \cap M}$ is semifinite for some $T \in \cP(M,N)$ (this is automatic if $N \subset M$ is with expectations or if $P \subset M$ is with expectations). Then  $$\rb(P \subset N, \psi)  = \rb(P \subset M,\psi) \cap N$$ and 
$$T|_{\rb(P \subset M,\psi)} \in \cP(\rb(P \subset M,\psi), \rb(P \subset N, \psi)).$$
\end{enumerate}
\end{prop}
\begin{proof}
(1) Take $S \in \cP(N,P)$ and $T \in \cP(M,N)$. Let $\phi=\psi \circ S \circ T$. Clearly, $P' \cap c_T(N)$ is contained in $P' \cap c(M)$. Therefore, $\rb(P \subset c_T(N),\psi)= (P' \cap c_T(N)) \vee \{ \phi \}''$ is contained in $\rb(P \subset c(M), \psi)= (P' \cap c(M)) \vee \{ \phi \}''$. Taking the intersection with $M$ and using the fact that $c_T(N) \cap M=N$, we obtain 
$\rb(P \subset N,\psi) \subset \rb(P \subset M, \psi)$.  

(2) Take $\phi$ as above. Since $N' \cap c(M) \subset P' \cap c(M)$, we have $(N' \cap c(M) \vee \{ \phi \}''\subset (P' \cap c(M) \vee \{ \phi \}''$. By taking the intersection with $M$, we get $\rb(N \subset M, \psi \circ S) \subset \rb(P \subset M, \psi)$.

(3) Take $\phi$ as above, with $T \in \cP(M,N)$ and $T|_{P' \cap M}$ semifinite. Let $c(T) \in \cP(c(M), c_T(N))$ be the operator valued weight induced by $T$. Observe that 
$$c(T)|_{P' \cap c(M)} \in \cP( P' \cap c(M), P' \cap c_T(N)) $$ is semifinite because $c(T)|_{P' \cap M}=T|_{P' \cap M}$ is semifinite and $P'\cap M \subset P' \cap c(M)$. Since $(P' \cap c(M)) \vee \{ \phi \}''$ is the closed linear span of $\{ x \phi^{\ri t} \mid x \in \mathbf{m}_{c(T)} \cap P' \cap c(M), \; t \in \R \}$ (see \cite[Definition 2.1]{Ha77a} for the definition of $\mathbf{m}_{c(T)}$) and $c(T)(x\phi^{\ri t})=c(T)(x)\phi^{\ri t}$ for all $x \in \mathbf{m}_{c(T)} \cap P' \cap c(M), \; t \in \R$, we see that $c(T)$ restricts to a semifinite operator valued weight from $(P' \cap c(M)) \vee \{ \phi \}''$ onto $(P' \cap c_E(N)) \vee \{ \phi \}''$. We conclude that $T=c(T)|_M$ restricts to a semifinite operator valued weight from $\rb(P \subset M, \psi)$ onto $\rb(P \subset N, \psi)$. Now, take $x \in \rb(P \subset M, \psi) \cap N$ and $a \in \rb(P \subset M,\psi)^+$ such that $T(a)=1$. Then $T(xa)=xT(a)=x$, hence $x \in \rb(P \subset N,\psi)$. We conclude that $\rb(P \subset N,\psi)=\rb(P \subset M, \psi) \cap N$.
\end{proof}

The following proposition relies on the deep Relative Commutant Theorem of Connes and Takesaki.

\begin{prop} \label{bicentralizer in center}
Let $N \subset M$ be an inclusion of von Neumann algebras with $\cP(M,N) \neq \emptyset$. Take $\psi \in \cP(N)$. Then $\rb(N, \psi) \subset \cZ(N_\psi)$ and $\rb(N, \psi) \subset \cZ(\rb(N \subset M, \psi))$.
\end{prop}
\begin{proof}
By the Relative Commutant Theorem, we have $N' \cap c(N)=\cZ(c(N))$. Therefore, $N' \cap c(N) \vee \{ \psi \}''$ is abelian. Thus $\rb(N,\psi)$ commutes with $\{ \psi \}''$ and this means that $\rb(N,\psi) \subset N_\psi$. We already saw that $\rb(N,\psi) \subset N_\psi' \cap N$. We conclude that $\rb(N,\psi) \subset \cZ(N_\psi)$. We also know that $\rb(N \subset M,\psi)$ commutes with $N_\psi$, hence also with $\rb(N,\psi)$.
\end{proof}

\begin{prop} \label{tensor product alg bic}
Let $N_i \subset M_i$ be two inclusions of von Neumann algebras with $\cP(M_i,N_i) \neq \emptyset$ for $i=1,2$. Take $\psi_i \in \cP(N_i)$. Then we have
$$\rb(N_1 \ovt N_2 \subset M_1 \ovt M_2, \psi_1 \otimes \psi_2) \subset\rb(N_1 \subset M_1, \psi_1) \ovt\rb(N_2 \subset M_2, \psi_2).$$
\end{prop}
\begin{proof}
View $c(M_1 \ovt M_2)$ as a subalgebra of $c(M_1) \ovt c(M_2)$. Then 
$$(N_1 \ovt N_2)' \cap c(M_1 \ovt M_2) \subset (N_1' \cap c(M_1) ) \ovt (N_2' \cap c(M_2)).$$
Moreover, if $T_i \in \cP(M_i,N_i), \: i=1,2$, then there exists $T_1 \otimes T_2 \in \cP(M_1 \ovt M_2,N_1 \ovt N_2)$ such that $(\psi_1 \ovt \psi_2) \circ (T_1 \otimes T_2)=(\psi_1 \circ T_1) \otimes (\psi_2 \circ T_2)$, hence 
$$\{ (\psi_1 \otimes \psi_2) \circ (T_1 \otimes T_2) \}'' \subset \{\psi_1 \circ T_1 \}'' \ovt \{ \psi_2 \circ T_2 \}''.$$
Combining both inclusions, we obtain
$$\rb(N_1 \ovt N_2 \subset c(M_1 \ovt M_2), \psi_1 \otimes \psi_2) \subset\rb(N_1 \subset c(M_1), \psi_1) \ovt\rb(N_2 \subset c(M_2), \psi_2)$$
and we conclude by taking the intersection with $M_1 \ovt M_2$.
\end{proof}

\subsection{Dominant and integrable weights}
\begin{thm} \label{computation dominant}
Let $N \subset M$ be an inclusion of von Neumann algebras with $T \in \cP(M,N)$. Take $\psi \in \cP(N)$ a dominant weight. Then
$$(N_\psi' \cap c(M), N' \cap c(M), ( (\psi \circ T)^{\ri t} )_{t \in \R} )$$ is a crossed product extension and we have $\rb(N \subset c(M),\psi)=N_\psi' \cap c(M)$ and $\rb(N \subset M, \psi)=N_\psi ' \cap M$.
\end{thm}
\begin{proof}
Let $u=(u_\lambda)_{\lambda \in \R^*_+}$ be a $\psi$-scaling group, i.e.\ a one-parameter group of unitaries in $N$ such that $u_\lambda \psi u_\lambda^*=\lambda \psi$ for all $\lambda \in \R^*_+$. Let $\rho : \R^*_+ \curvearrowright N_\psi' \cap c(M)$ be the action given by $\rho_\lambda=\Ad(u_\lambda)|_{N_\psi' \cap c(M)}$. Since $N$ is generated by $N_\psi$ and $(u_\lambda)_{\lambda \in \R^*_+}$, we have $(N_\psi' \cap c(M))^\rho=N' \cap c(M)$. Let $\phi=\psi \circ T$. Since $\rho_\lambda(\phi^{\ri t})=\lambda^{\ri t} \phi^{\ri t}$ for all $t \in \R$ and $\lambda \in \R^*_+$, we know by Theorem \ref{dual action crossed product} that $N_\psi' \cap c(M)$ is a crossed product extension of the fixed point algebra $N' \cap c(M)=(N_\psi' \cap c(M))^\rho$ by the unitaries $(\phi^{\ri t})_{t \in \R}$. In particular, $N_\psi ' \cap c(M)$ is generated by $N' \cap c(M)$ and $\{ \phi \}''$, i.e.\ $N_\psi' \cap c(M)=\rb(N \subset c(M), \psi)$, hence $\rb(N \subset M,\psi)=(N_\psi' \cap c(M)) \cap M=N_\psi' \cap M$ as we wanted.
\end{proof}

Following \cite{CT76}, we say that a weight $\varphi \in \cP(M)$ is \emph{integrable} if the action $\sigma^\psi : \R \curvearrowright M$ is integrable, i.e.\ if the Haar operator valued weight $\int_{\R} \sigma_t^\psi \: \rd t$ is semifinite.

\begin{thm}[\cite{CT76}] \label{connes takesaki integrable}
Let $M$ be a properly infinite von Neumann algebra and $\psi \in \cP(M)$ a dominant weight. Then $\varphi \in \cP(M)$ is integrable if and only if there exists an isometry $v \in M$ such that $v^*v=1$, $vv^* \in M_\psi$ and $\varphi=v^*\psi v$. If $\varphi$ has infinite multiplicity, then we can take $vv^* \in \cZ(M_\psi)$.
\end{thm}

\begin{cor}
Let $N \subset M$ be an inclusion of von Neumann algebras with $\cP(M,N) \neq \emptyset$. Take $\varphi \in \cP(N)$ an integrable weight. Then $\rb(N \subset c(M),\varphi)=N_\varphi' \cap c(M)$ and $\rb(N \subset M,\varphi)=N_\varphi' \cap M$.
\end{cor}
\begin{proof}
By Theorem \ref{connes takesaki integrable}, every integrable weight on $N$ is equivalent to $\psi_e$ for some projection $e \in N_\psi$ and we have $\rb(eNe \subset eMe,\psi_e)=e\rb(N \subset M,\psi)e=e(N_\psi' \cap M)e=N_{\psi_e}' \cap eMe$.
\end{proof}

\subsection{Lacunary weights}
Let $M$ be a von Neumann algebra. Following \cite{CT76}, we say that weight $\varphi \in \cP(M)$ is \emph{lacunary} if $1$ is isolated in the spectrum of $\Delta_\varphi$. If $M$ is of type $\III_0$ then it admits a lacunary weight. If $M$ is of type $\III_{\Lambda}$ for some $\Lambda \in \cZ(M)$ with $0 < \Lambda < 1$ then $M$ admits a lacunary weight if and only if $\| \Lambda \| < 1$. If $M$ is of type $\III_1$ then it admits no lacunary weight.

The following theorem was proved by Connes \cite{Co72} and Connes-Takesaki \cite{CT76} when $N$ is a factor. The proofs can be adapted easily to the non-factorial case.
\begin{thm}[Connes-Takesaki \cite{CT76}] \label{connes takesaki lacunary}
Let $N$ be a type $\III$ von Neumann algebra. Let $\varphi \in \cP(N)$ be a lacunary weight with infinite multiplicity.
\begin{enumerate}
\item There exists $\lambda \in ]0,1[$ and $U \in \cU(N)$ such that $UN_\varphi U^*=N_\varphi$ and $\varphi \circ \Ad(U) \leq \lambda \varphi$
\item $(N, N_\varphi, (U^n)_{n \in \Z})$ is a crossed product extension.
\item Write $\varphi \circ \Ad(U) =\varphi_\rho$ with $\rho \in \cZ(N_\varphi)$. For every $h \in N_\varphi$ with $\rho \leq h < 1$, we have $N_{\varphi_h} \subset N_\varphi$.
\item For every $\psi \in \cP(N)$, there exists $h \in N_\varphi$ with $\rho \leq h < 1$ and an isometry $v \in N$ such that $e=vv^* \in N_{\varphi_h}$ and $\psi=v^*\varphi_h v$.
\end{enumerate}
\end{thm}

\begin{thm} \label{computation lacunary}
Let $N \subset M$ be an inclusion of von Neumann algebras with $\cP(M,N) \neq \emptyset$. Suppose that $N$ is a type $\III$ von Neumann algebra with a lacunary weight $\varphi \in \cP(N)$. Take $\lambda$, $U$ and $\rho$ as above.
\begin{enumerate}
\item $\rb(N \subset M, \varphi)=N_\varphi' \cap M$
\item Take $h \in N_\varphi$ with $\rho \leq h < 1$. Then 
$$\rb(N \subset M, \varphi_h)=\rb(N_\varphi \subset M, \varphi_h|_{N_\varphi})=(N_\varphi' \cap M) \vee \{h\}''.$$
\end{enumerate}
\end{thm}
\begin{proof}
(1) We already know that $\rb(N \subset M,\varphi) \subset N_\varphi ' \cap M$. Let $H$ be a Hilbert space and let $\tilde{N}=N \ovt \B(H)$, $\tilde{M}=M \ovt \B(H)$ and $\tilde{\varphi}=\varphi \otimes \Tr$.

Take $k \in \B(H)_+$ such that $\lambda \leq k < 1$ and $k$ has absolutely continuous spectrum. Then $\Tr(k \cdot)$ is integrable. Therefore $\tilde{\varphi}_{1 \otimes k}=\varphi \otimes \Tr(k \cdot) \in \cP(N \ovt \B(H))$ is integrable, hence 
$$\tilde{N}_{\tilde{\varphi}_{1 \otimes k}}' \cap \tilde{M}= \rb(\tilde{N} \subset \tilde{M}, \tilde{\varphi}_{1 \otimes k}) \subset\rb(N \subset M, \varphi) \ovt \{k\}''.$$
But since $\rho \otimes 1 \leq \lambda \leq 1 \otimes k < 1$, we have $\tilde{N}_{\tilde{\varphi}_{1 \otimes k}} =\{ 1 \otimes k\}' \cap  \tilde{N}_{\tilde{\varphi}}$. Therefore, 
$$(N_\varphi' \cap M) \otimes 1 \subset \tilde{N}_{\tilde{\varphi}_{1 \otimes k}}' \cap \tilde{M} \subset\rb(N \subset M, \varphi) \ovt \{k\}''.$$
hence $N_\varphi' \cap M \subset\rb(N \subset M,\varphi)$ as we wanted.

(2) By \ref{alg bicentralizer for intermediate subalgebra}, we already know that $\rb(N \subset M,\varphi_h) \subset\rb(N_\varphi \subset M, \varphi_h)= \{h\}'' \vee (N_\varphi ' \cap M)$. We want to prove the reverse inclusion. Suppose first that $\|h \| < 1$.  We keep the same notations as above, but this time, we take $k \in \B(H)_+$ such that $1 \leq k < \|h\|^{-1}$ and $k$ has absolutely continuous spectrum. Then we again have 
$$\tilde{N}_{\tilde{\varphi}_{h \otimes k}}' \cap \tilde{M}= \rb(\tilde{N} \subset \tilde{M}, \tilde{\varphi}_{h \otimes k}) \subset\rb(N \subset M, \varphi_h) \ovt \{k\}''.$$
We also have $\rho \otimes 1 \leq h \otimes k < 1$. Thus  $\tilde{N}_{\tilde{\varphi}_{h \otimes k}} =\{ h \otimes k\}' \cap \tilde{N}_{\tilde{\varphi}}$.  Therefore $$(\{h\}'' \vee (N_\varphi ' \cap M)) \otimes 1 \subset \tilde{N}_{\tilde{\varphi}_{h \otimes k}}' \cap \tilde{M} \subset\rb(N \subset M, \varphi_h) \ovt \{k\}'',$$
hence $\{h\}'' \vee (N_\varphi ' \cap M) \subset\rb(N \subset M,\varphi_h)$ as we wanted.

Now, if we take any $h \in N_\varphi$, with $\rho \leq h < 1$, we can find an arbitrarily large projection $e \in \{h\}''$ such that $\| he\| < 1$. Then we know by the first case that 
$$\rb(eNe \subset eMe, \varphi_{he})=\{he\}'' \vee ((eNe)_{\varphi_{he}}' \cap eMe) = e( \{h\}'' \vee (N_{\varphi_h}' \cap M) )e.$$
On the other hand, we have $\rb(eNe \subset eMe, \varphi_{he})=e\rb(N \subset M, \varphi_h)e$. By letting $e$ increase to $1$, we conclude that $\rb(N \subset M, \varphi_h)=\{h\}'' \vee (N_{\varphi_h}' \cap M)$.
\end{proof}

\begin{cor} \label{no III1 semifinite reduction}
Let $N$ be a type $\III$ von Neumann algebra with no type $\III_1$ summand. Take $\psi \in \cP(N)$. There exists a subalgebra $Q \subset N$ with  a $\psi$-preserving faithful normal conditional expectation $E_Q \in \cE(N,Q)$ such that :
\begin{itemize}
\item $Q$ is semifinite.
\item $Q' \cap N \subset Q$.
\item $N_\psi \subset Q$
\item $\rb(N,\psi)=\rb(Q,\psi|_Q)$ and for every inclusion $N \subset M$ with $\cP(M,N) \neq \emptyset$, we have 
$$\rb(N \subset M,\psi)=\rb(Q \subset M, \psi|_Q).$$
\end{itemize}
If in addition $\psi$ has infinite multiplicity, we can choose $Q$ such that $N=Q \rtimes_\alpha \Z$ for some $\alpha \in \Aut(Q)$.
\end{cor}
\begin{proof}
It is clear that we may decompose $N$ into direct summands and prove the corollary for each summand separately. Thus we may assume that $N$ is of type $\III_0$ or of type $\III_{\Lambda}$ for some $\Lambda \in \cZ(M)$ with $\| \Lambda \| < 1$. In particular, we may assume that $N$ admits a lacunary weight $\varphi$ with infinite multiplicity. By Theorem \ref{connes takesaki lacunary}, there exists $h \in N_\varphi$ with $\rho \leq h < 1$ and an isometry $v \in N$ such that $e=vv^* \in N_{\varphi_h} \subset N_\varphi $ and $\psi=v^*\varphi_h v$. We then take $Q=v^* N_\varphi v$ and we apply Theorem \ref{computation lacunary}. If $\psi$ has infinite multiplicity, then we may assume that $e \in  \cZ(N_{\varphi_h})= \cZ(N_\varphi)$ and $\varphi_e$ will be a lacunary weight with infinite multiplicity. We can then apply Theorem \ref{connes takesaki lacunary} to deduce that $N$ will decompose as a crossed product of $Q$ by an action of $\Z$.
\end{proof}

\subsection{Invariance under ucp maps and the type $\III_1$ case}
\begin{lem} \label{ucp map transition}
Let $N$ be a von Neumann algebra. Take $\psi,\varphi \in \cP(N)$ and suppose that there exists $f \in \cV^{\psi,\varphi}(N)$ such that $f|_{\rb(N,\varphi)}$ is normal. Then for any inclusion with expectations $N \subset M$, the following holds.
\begin{enumerate}
\item There exists a unique surjective normal *-morphism $\rb^{\psi,\varphi}$ from $\rb(N \subset c(M),\varphi)$ onto $\rb(N \subset c(M),\psi)$ that fixes $N' \cap c(M)$ and sends $(\varphi \circ T)^{\ri t}$ to $(\psi \circ T)^{\ri t}$ for every $T \in \cP(M,N)$ and every $t \in \R$.
\item We have $\rb^{\psi,\varphi}(x)=c(f)(x)$ for every $x \in \rb(N \subset c(M),\varphi)$ and every $f \in \cV^{\psi,\varphi}(N \subset M)$.
\item $\rb^{\psi,\varphi}$ is equivariant with respect to the trace scaling action $\theta : \R^*_+ \curvearrowright c(M)$ and restricts to a surjective normal *-morphism from $\rb(N \subset M,\varphi)$ onto $\rb(N \subset M,\psi)$
\end{enumerate}
\end{lem}
\begin{proof}
(1) and (2). We may assume that $N' \cap M$ is countably decomposable and take $E \in \cE(M,N)$. Extend $f$ to some element of $\cV^{\psi,\varphi}(N \subset M)$ that we still denote $f$. Then $E \circ f=f \circ E$ because $f \in \cV(N \subset M)$. Since $E(\rb(N \subset M, \varphi))=\rb(N,\varphi)$, we know that $(f \circ E)|_{\rb(N \subset M,\varphi)}$ is normal. Thus $(E \circ f)|_{\rb(N \subset M,\varphi)}$ is normal, hence $f|_{\rb(N \subset M,\varphi)}$ is normal. By Theorem \ref{extension ucp to core} and Remark \ref{normal crossed product ucp}, we deduce that $c(f)$ is normal on $\rb(N \subset M, \varphi) \rtimes_{\sigma^{\varphi \circ E}} \R=\rb(N \subset c(M),\varphi)$. Moreover, $c(f)$ fixes $N' \cap c(M)$ and sends $(\varphi \circ E)^{\ri t}$ to $(\psi \circ E)^{\ri t}$ for every $t \in \R$, hence $c(f)$ is a *-morphism when restricted to the *-algebra generated by $N' \cap c(M)$ and $(\varphi \circ E)^{\ri t}$ for $t \in \R$. Since $c(f)|_{\rb(N \subset c(M),\varphi)}$ is normal and $\rb(N \subset c(M),\varphi)=(N' \cap c(M)) \vee \{ \varphi \circ E\}''$, we conclude that $c(f)|_{\rb(N \subset c(M),\varphi)}$ is a surjective *-morphism from $\rb(N \subset c(M),\varphi)$ onto $\rb(N \subset c(M),\psi)$. This proves (1) and (2).

(3) Observe that the trace scaling action $\theta : \R^*_+ \curvearrowright c(M)$ commutes with $c(f)$, hence also with $\rb^{\psi,\varphi} : \rb(N \subset c(M),\varphi) \rightarrow \rb(N \subset c(M),\psi)$. It follows that $\rb^{\psi,\varphi}$ sends $\rb(N \subset M, \varphi) =\rb(N \subset c(M),\varphi)^\theta$ into $\rb(N \subset M, \psi) =\rb(N \subset c(M),\psi)^\theta$. Let us prove the surjectivity. Let $A=\rb^{\psi,\varphi}(\rb(N \subset M, \varphi))$. Since $\rb(N \subset c(M),\varphi)$ is generated by $\rb(N \subset M, \varphi)$ and $(\varphi \circ E)^{\ri t}, \: t \in \R$ and $\rb^{\psi,\varphi}$ is sujrective, then  $\rb(N \subset c(M),\psi)$ is generated by $A$ and $(\psi \circ E)^{\ri t}, \: t \in \R$. We conclude by Theorem \ref{dual action crossed product} that $A=\rb(N \subset c(M),\psi)^\theta=\rb(N \subset M,\psi)$ as we wanted.
\end{proof}

\begin{thm} \label{ucp map transition III1}
Let $N \subset M$ be an inclusion of von Neumann algebras with expectations. Suppose that $N$ is of type $\III_1$. Take $\varphi, \psi \in \cP(N)$.
\begin{enumerate}
\item  There exists a unique isomorphism $$\rb^{\psi,\varphi} :\rb(N \subset c(M),\varphi) \rightarrow\rb(N \subset c(M),\psi)$$ such that $\rb^{\psi,\varphi}(x)=x$ for every $x \in N' \cap c(M)$ and $\rb^{\psi,\varphi}( (\varphi \circ T)^{\ri t})=(\psi \circ T)^{\ri t}$ for every $T \in \cP(M,N)$ and $t \in \R$.
\item $\rb^{\psi,\varphi}$ is equivariant with respect to the trace scaling action $\theta : \R^*_+ \curvearrowright c(M)$ and restricts to an isomorphism from $\rb(N \subset M,\varphi)$ onto $\rb(N \subset M,\psi)$
\item If $\phi \in \cP(N)$ is a third weight, we have $\rb^{\psi,\varphi} \circ \rb^{\varphi,\phi}=\rb^{\psi,\phi}$.
\item The group morphism $$ \R^*_+  \ni \lambda \mapsto \rb^{\lambda \varphi,\varphi} \in \Aut(\rb(N \subset c(M),\varphi))$$
defines a continuous action $\rb^\varphi : \R^*_+ \curvearrowright \rb(N \subset c(M),\varphi)$.
\item $\rb(N \subset c(M),\varphi)^{\rb^\varphi}=N' \cap c(M)$.
\end{enumerate}
\end{thm}
\begin{proof}
(1), (2) and (3) follow directly from Proposition \ref{transition UCP weights} and Lemma \ref{ucp map transition}. 

(4) We see from item (3) that $\rb^\varphi$ is indeed an action and that $\rb^{\psi,\varphi}$ conjugates the action $\rb^\varphi$ with the action $\rb^\psi$. Therefore, it is enough to prove the continuity of $\rb^\psi$ when $\psi$ is a dominant weight. Let $u=(u_\lambda)_{ \lambda \in \R^*_+}$ be a $\psi$-scaling group. We have $\rb^\psi_\lambda(x)=u_\lambda xu_\lambda^*$ for every $x \in \rb(N \subset c(M),\psi)=N_\psi' \cap c(M)$ and every $\lambda \in \R^*_+$. Thus the continuity of $\rb^\psi$ follows from the continuity of $u$.

(5) Take $T \in \cP(M,N)$. Then $$(\rb(N \subset c(M),\varphi), N' \cap c(M), ((\varphi \circ T)^{\ri t})_{t \in \R} )$$ is a split $\R$-extension and we see that $\rb^\psi$ is a dual action for this extension, hence we can apply Theorem \ref{dual action crossed product}.
\end{proof}

\begin{rem}
The combination of Theorem \ref{computation dominant} and Theorem \ref{ucp map transition III1} answers, at the algebraic bicentralizer level, the open question at the end of \cite[Introduction]{AHHM18} : all bicentralizer algebras $\rb(N \subset M,\varphi)$ are canonically isomorphic to $\rb(N \subset M,\psi)=N_\psi' \cap M$ where $\psi$ is a dominant weight. 
\end{rem}

\subsection{The modular automorphism group}
Let $N \subset M$ be an inclusion of von Neumann algebras with expectations and $\psi \in \cP(N)$. Recall from Proposition \ref{alg bicentralizer for intermediate subalgebra} that $T|_{\rb(N \subset M,\psi)} \in  \cP(\rb(N \subset M,\psi),\rb(N,\psi))$. Moreover, by \ref{bicentralizer in center}, $\rb(N,\psi)$ is in the center of $\rb(N \subset M,\psi)$, hence the modular automorphism group of $T|_{\rb(N \subset M,\psi)}$ (in the sense of \cite{Ha77a}) is defined on $\rb(N,\psi)' \cap \rb(N \subset M,\psi)=\rb(N \subset M,\psi)$. The next theorem tells us what this modular automorphism group is. This is immediated when $\psi$ is finite, but the author could not find a simple and direct proof of this result when $\psi$ is infinite, even when $\psi$ is a dominant weight, which was our main motivation.

\begin{thm} \label{modular group of bicentralizer}
Let $N \subset M$ be an inclusion of von Neumann algebras with expectations. Take $T \in \cP(M,N)$ and $\psi \in \cP(N)$. Then the modular automorphism group of $T|_{\rb(N \subset M,\psi)} \in  \cP(\rb(N \subset M,\psi),\rb(N,\psi))$ is given by $\sigma^{\psi \circ T} |_{\rb(N \subset M,\psi)}$.
\end{thm}
\begin{proof}
We may assume that $N$ is countably decomposable.

\textbf{Case 0 : $\psi$ is finite.} In this case, $\eta=(\psi \circ T)|_{\rb(N \subset M,\psi)}$ is semifinite. Since $\rb(N \subset M,\psi)$ is globally invariant under $\sigma^{\psi \circ T}$, we know that $\sigma^{\eta}=\sigma^{(\psi \circ T)}|_{\rb(N \subset M,\psi)}$. By definition, $\sigma^\eta$ is the modular automorphism group of $T|_{\rb(N \subset M,\psi)}$.

\textbf{Case 1 : $N$ is semifinite.} Let $\tau \in \cP(N)$ be a trace and write $\psi=\tau(h \cdot)$ for some positive $h$ affiliated with $N$. Then $\rb(N,\psi)=\{h\}'' \vee \cZ(N)$ and $\rb(N \subset M, \psi)=\{h\}'' \vee (N' \cap M)$. Take $\varphi$ a faithful normal state on $N$ and let $\eta=(\varphi \circ T)|_{\rb(N \subset M, \psi)}$. The modular automorphism group of $T|_{\rb(N \subset M, \psi)}$ is given by $\sigma^\eta$. So, we have to show that $\sigma^\eta=\sigma^{\psi \circ T}|_{\rb(N \subset M, \psi)}$. For this it is enough to show that they coincide on both $\rb(N,\psi)$ and $N' \cap M$ because these two subalgebras generate $\rb(N \subset M,\psi)$. 

Firstly, since $\rb(N,\psi) \subset \cZ(\rb(N \subset M,\psi))$, we have $\sigma_t^\eta(x)=x$ for all $x \in\rb(N,\psi)$. Since $\sigma^{\psi \circ T}_t=\Ad(h^{\ri t})$, we also have $\sigma^{\psi \circ T}_t(x)=x$ for all $x \in\rb(N,\psi)=\{h\}'' \vee \cZ(N)$. Therefore, we get $\sigma_t^\eta(x)=\sigma^{\psi \circ T}_t(x)$ for all $x \in\rb(N,\psi)$. Secondly, we have $\sigma^{\varphi \circ T}|_{N' \cap M}=\sigma^{\psi \circ T}|_{N' \cap M}=\sigma^T$. Therefore, $\sigma^\eta_t(x)=\sigma^{\varphi \circ T}_t(x)=\sigma^{\psi \circ T}_t(x)$ for all $x \in N' \cap M$. We conclude that $\sigma^\eta=\sigma^{\psi \circ T}|_{\rb(N \subset M, \psi)}$ as we wanted.

\textbf{Case 2 : $N$ has no type $\III_1$ summand.} Take $Q \subset N$ with a $\psi$-preserving conditional expectation $E_Q \in \cE(N,Q)$ as in Corollary \ref{no III1 semifinite reduction}. Then we have $\rb(N \subset M,\psi)=\rb(Q \subset M,\psi|_Q)$ and $\rb(N,\psi)=\rb(Q,\psi|_Q)$. Let $S=E_Q \circ T$. We have $S|_{\rb(Q \subset M,\psi)}=T|_{\rb(N \subset M,\psi)}$. By the semifinite case, we know that the modular automorphism group of $S|_{\rb(Q \subset M,\psi)}$ is $\sigma^{(\psi|_Q) \circ S}|_{\rb(Q \subset M,\psi)}$. Since $\psi \circ T=(\psi|_Q) \circ S$, we conclude that the modular automorphism group of $T|_{\rb(N \subset M,\psi)}=S|_{\rb(Q \subset M,\psi)}$ is $\sigma^{\psi \circ T}|_{\rb(N \subset M, \psi)}$.

\textbf{Case 3 : $N$ is of type $\III_1$.} Take $\varphi$ a faithful normal state on $N$. Take $f \in \cV^{\psi,\varphi}(N \subset M)$. Then $f$ commutes with $T$, intertwines $\sigma^{\varphi \circ T}$ with $\sigma^{\psi \circ T}$ and $f|_{\rb(N \subset M, \varphi)}=\rb^{\psi,\varphi}$ is an isomorphism from $\rb(N \subset M, \varphi)$ onto $\rb(N \subset M, \psi)$ that sends $\rb(N,\varphi)$ onto $\rb(N,\psi)$. Thus the result follows from Case 0 applied to $\varphi$.
\end{proof}

\begin{cor}
Let $N \subset M$ be an inclusion of von Neumann algebras with expectations. Take $\psi \in \cP(N)$ and $\omega \in \cP(\rb(N,\psi))$. There exists a unique $S \in \cP(M,\rb(N \subset M,\psi))$ such that $\psi \circ T = \omega \circ T \circ S$ for every $T \in \cP(M,N)$.
Moreover, we have $$\rb(N \subset c(M),\psi)=c_{S}(\rb(N \subset M,\psi)).$$

If $\psi$ is finite and $\omega=\psi|_{\rb(N,\psi)}$, then $S$ is a conditional expectation, denoted by $E_{\rb(N \subset M,\psi)}$.
\end{cor}
\begin{proof}
Take $T \in \cP(M,N)$. Let $\phi=\omega \circ T|_{\rb(N \subset M, \psi)} \in \cP(\rb(N \subset M,\psi))$. By Theorem \ref{modular group of bicentralizer}, we have $\sigma^\phi=\sigma^{\psi \circ T}|_{\rb(N \subset M,\psi)}$. Thus, there exists a unique operator valued weight $S \in \cP(M, \rb(N \subset M,\psi))$ such that $\psi \circ T=\phi \circ S=\omega \circ T \circ S$. By Proposition \ref{alg bicentralizer of core}, we deduce that $\rb(N \subset c(M),\psi)=c_S(\rb(N \subset M,\psi))$.

Take another $T' \in \cP(M,N)$. In order to show $\psi \circ T'=\omega \circ T' \circ S$, we use \cite[Theorem 6.5 and 6.6]{Ha77b}. Let $T_0=T|_{\rb(N \subset M,\psi)}$ and $T_0'=T'|_{\rb(N \subset M,\psi)}$. It follows from Theorem \ref{modular group of bicentralizer} and the usual 2 by 2 matrix trick that 
$$(DT'_0 : DT_0)_t =(D(\psi \circ T') : D(\psi \circ T))_t=(DT': DT)_t$$ for all $t \in \R$. Therefore, we get
$$ (D(\omega \circ T' \circ S) : D(\omega \circ T \circ S))_t=(D( \omega \circ T'_0) : D(\omega \circ T_0))_t=(DT'_0 : DT_0)_t=(D(\psi \circ T') : D(\psi \circ T))_t$$ for all $t \in \R$. Since $\psi \circ T=\omega \circ T \circ S$, we conclude that $\psi \circ T'=\omega \circ T' \circ S$.
\end{proof}

%

\subsection{The canonical bicentralizer extension}

\begin{prop}
Let $N \subset M$ an inclusion of von Neumann algebras with $\cP(M,N) \neq \emptyset$. The von Neumann subalgebra $\rb^\sharp(N \subset M) \subset M$ generated by $N$ and $\rb(N\subset M, \varphi)$ does not depend on the choice of $\varphi \in \cP(N)$.
\end{prop}
\begin{proof}
Take $T \in \cP(M,N)$. On one hand, we have
$$\rb(N \subset c(M),\varphi) \vee N= (N' \cap c(M)) \vee \{ \varphi \circ T \}'' \vee N= (N' \cap c(M)) \vee c_T(N).$$
On the other hand, under the identification $c(M)=M \rtimes_{\sigma^{\varphi \circ T}} \R$, we have
$$\rb(N \subset c(M),\varphi) \vee N=\rb(N \subset M,\varphi)  \vee \{ \varphi \circ T \}'' \vee N=(\rb(N \subset M,\varphi)  \vee N) \rtimes_{\sigma^{\varphi \circ T}} \R.$$
We conclude that 
$$\rb(N \subset M,\varphi)  \vee N = \left( (N' \cap c(M)) \vee c_T(N) \right) \cap M$$
which clearly does not depend on $\varphi \in \cP(N)$.
\end{proof}

\begin{prop}
Let $N \subset M$ be an inclusion of von Neumann algebras with $\cP(M,N) \neq \emptyset$. Let $e$ be a projection in $N$ or in $N' \cap M$. Then
 $$e\rb^\sharp(N \subset M)e=\rb^\sharp(eNe \subset eMe).$$
\end{prop}
\begin{proof}
First, it is easy to see that the proposition holds when $e \in \cZ(N)$. Thus, up to decomposing into direct summands, we can reduce to the case where $N$ is either finite or properly infinite. If $N$ is finite then $\rb^\sharp(N \subset M)=N \vee (N' \cap M)$ and it is known that $e(N \vee (N' \cap M))e=(eNe) \vee ((eNe)' \cap eMe)$.

Suppose that $N$ is properly infinite. Take a dominant weight $\psi \in \cP(N)$. If $e \in N$, we choose $\psi$ such that $e \in N_\psi$. Then we have
$$ e ( N \cdot (N_\psi' \cap M)) e= (eNe) \cdot ((eNe)_{\psi_e}' \cap eMe).$$
Since $ N \cdot (N_\psi' \cap M)$ is dense in $\rb^\sharp(N \subset M)$, we conclude that 
 $$e\rb^\sharp(N \subset M)e=\rb^\sharp(eNe \subset eMe).$$
\end{proof}

\begin{prop}
Let $N \subset M$ an inclusion of von Neumann algebras with expectations. There exists a unique normal conditional expectation from $M$ onto $\rb^\sharp(N \subset M)$ and we have $\rb^\sharp(N \subset c(M))=c(\rb^\sharp(N \subset M))$.
\end{prop}
\begin{proof}
Take $\varphi \in \cP(N)$ and $T \in \cP(M,N)$. Since $N \subset M$ is with expectations, $T|_{N' \cap M}$ is semifinite. But $N \vee (N' \cap M) \subset \rb^\sharp(N \subset M)$, hence $\varphi \circ T$ is semifinite on $\rb^\sharp(N \subset M)$. Since $\rb^\sharp(N \subset M)$ is globally $\sigma^{\varphi \circ T}$-invariant, we conclude that there exists a faithful normal $\varphi \circ T$-preserving conditional expectation $E \in \cE(M, \rb^\sharp(N \subset M))$. Since $\rb^\sharp(N \subset M)' \cap M \subset N' \cap M \subset \rb^\sharp(N \subset M)$, then $\cE(M, \rb^\sharp(N \subset M))$ is a singleton. Finally, $\rb^\sharp(N \subset c(M))$ is generated by $\rb^\sharp(N \subset M)$ and $\{ \varphi \circ T \}''$, hence $\rb^\sharp(N \subset c(M))=c(\rb^\sharp(N \subset M))$.
\end{proof}

%
%

\begin{prop} \label{structure of bicentralizer extension}
Let $N \subset M$ be an inclusion of von Neumann algebras with expectations. Suppose that $N$ is properly infinite and let $\psi \in \cP(N)$ be a dominant weight and $u=(u_\lambda)_{\lambda \in \R^*_+}$ be a $\psi$-scaling group.
\begin{enumerate}[(1)]
\item $N_\psi$ is with expectations inside $N_\psi \vee (N_\psi' \cap c(M))$, hence $$N_\psi \vee (N_\psi' \cap c(M))  \cong N_\psi \ovt_{\cZ(N_\psi)} (N_\psi' \cap c(M)).$$
\item $\left( \rb^\sharp(N \subset c(M)) , N_\psi \vee (N_\psi' \cap c(M)), u \right)$ is a crossed product extension, hence
$$\rb^\sharp(N \subset c(M)) \cong \left( N_\psi \ovt_{\cZ(N_\psi)} (N_\psi' \cap c(M)) \right) \rtimes_{\Ad(u)} \R^*_+$$
and 
$$\rb^\sharp(N \subset M) \cong \left( N_\psi \ovt_{\cZ(N_\psi)} (N_\psi' \cap M) \right) \rtimes_{\Ad(u)} \R^*_+.$$
\end{enumerate} 
\end{prop}
\begin{proof}
We will reduce the problem to the case $N=M$. Take $S \in \cP(N' \cap M, \cZ(N))$ such that $S|_{(N' \cap M)^{\sigma^S}}$ is semifinite. Take $T \in \cP(M,N)$ such that $T|_{N' \cap M}=S$. Then $c_T(N)' \cap M=(N' \cap M)^{\sigma^S}$, hence $c(T)|_{c_T(N)' \cap c(M)}$ is semifinite. This shows that $c_T(N) \subset c(M)$ is with expectations, i.e.\  it admits faithful family of normal conditional expectations. Take $E$ a normal conditional expectation from $c(M)$ onto $c_T(N)$. Then $E$ restricts to a normal conditional expectation from $N_\psi \vee (N_\psi' \cap c(M))$ onto $N_\psi \vee (N_\psi' \cap c_T(N))$. This shows that $N_\psi \vee (N_\psi' \cap c_T(N)) \subset N_\psi \vee (N_\psi' \cap c(M))$ is with expectations.

(1) We have $N_\psi' \cap c(N)=(N_\psi' \cap N) \vee \{\psi\}''=\cZ(N_\psi) \vee \{\psi\}''$, hence $N_\psi \vee (N_\psi' \cap c(N))=N_\psi \vee \{ \psi \}'' \cong N_\psi \ovt \{ \psi \}''$ and this shows that $N_\psi$ is with expectation in $N_\psi \vee (N_\psi ' \cap c(N))$. From the reduction above, we conclude that $N_\psi \subset (N_\psi \vee N_\psi' \cap c(M))$ is with expectations and we can apply Corollary \ref{corollary expectation tensor product}.

(2) Observe that $(c(N), N_\psi \vee \{ \psi \}'', u)$ is a crossed product extension because it admits a dual action given by $\beta_t = \Ad( \psi^{\ri t})$ for all $t \in \R$ that satisfies $c(N)^\beta=N_\psi \vee \{ \psi \}''$ (we use Theorem \ref{dual action crossed product}). Since every normal conditional expectation from $c(M)$ onto $c_T(N)$ restricts to a conditional expectation from $N_\psi \vee (N_\psi' \cap c(M))$ onto $N_\psi \vee (N_\psi' \cap c_T(N))= N_\psi \vee \{ \psi \circ T\}''$, we can apply Proposition \ref{criterion expectation crossed product} to conclude that $(c(M), N_\psi \vee (N_\psi' \cap c(M)), u)$ is a crossed product extension. 
\end{proof}

\begin{prop} \label{no type 0 expectation}
Let $N \subset M$ an inclusion of von Neumann algebras with expectations. If $N$ has no type $\III_0$ summand, then $N$ is with expectations in $N \vee (N' \cap c(M))$, hence 
$$N \vee (N' \cap c(M)) \cong N \ovt_{\cZ(N)} (N' \cap c(M)).$$
\end{prop}
\begin{proof}
We first prove the proposition when $N=M$. We want to show that $N$ is with expectations in $N \vee \cZ(c(N))$.

If $N$ is semifinite, $N$ is with expectations in $c(N)$ and we have nothing to do. If $N$ is of type $\III_1$, then $\cZ(c(N))=\cZ(N)$, hence $N \vee \cZ(c(N))=N$ and again we have nothing to do.

Now suppose that $N$ is of type $\III_\Lambda$ with $\Lambda \in \cZ(N)$, $0 < \Lambda < 1$ (see Section \ref{prelim modular theory}). The scaling flow $\theta : \R^*_+ \curvearrowright c(N)$ extends to a continuous action $\widetilde{\theta} : \rL^0(\cZ(N),\R^*_+) \curvearrowright c(N)$ where $\rL^0(\cZ(N),\R^*_+)$ is the abelian topological group of all unbounded positive functions affiliated with $\cZ(N)$. By definition of the type function, $\widetilde{\theta}_{\Lambda}$ acts trivially on $\cZ(c(N))$. Let $h=\log(\Lambda)$. Define a new flow $\alpha : \R \curvearrowright c(N)$ by the formula $\alpha_t=\widetilde{\theta}_{e^{th}}$ for all $t \in \R$. We have $c(N)^\alpha=c(N)^\theta=N$. Moreover, $\alpha_1$ acts trivially on $\cZ(c(N))$, hence it also acts trivially on $N \vee \cZ(c(N))$. Therefore $\alpha$ restricts to a periodic flow on $N \vee \cZ(c(N))$ whose fixed point algebra is $N$. The Haar conditional expectation associated to $\alpha$ then provides us with a faithful normal conditional expectation from $N \vee \cZ(c(N))$ onto $N$.

Now, suppose $N \subset M$ is an inclusion with expectations. Take $T \in \cP(M,N)$. Then $c_T(N)$ is with expectations in $c(M)$ and every normal conditional expectation from $c(M)$ to $c(N)$ restricts to a normal conditional expectation from $N \vee (N' \cap c(M))$ onto $N \vee \cZ(c_T(N))$. This shows that $N \vee \cZ(c_T(N))$ is with expectations in $N \vee (N' \cap c(M))$. By the first part of the proof, we conclude that $N$ is with expectations in $N \vee (N' \cap c(M))$. By Proposition \ref{corollary expectation tensor product}, we get $N \vee (N' \cap c(M)) \cong N \ovt_{\cZ(N)} (N' \cap c(M))$.
\end{proof}

\begin{rem} \label{type 0 no expectation}
Suppose that $N$ is a von Neumann algebra of type $\III_0$. Then there is no normal conditional expectation from $N \vee \cZ(c(N))$ onto $N$, because in this case we have $N \vee \cZ(c(N))=c(N)$. Indeed, since $N$ is of type $\III_0$, the flow of weights $\theta : \R^*_+ \curvearrowright \cZ(c(N))$ is free. An action $\alpha : G \curvearrowright A$ of a locally compact group $G$ on an abelian von Neumann algebra $A$ is free if and only if $\alpha(A) \vee (A \otimes 1) = A \ovt \rL^\infty(G)$, where $\alpha : A \rightarrow A \ovt \rL^\infty(G)$ is the natural coaction map. Using this fact, one verifies that if the flow of weights is free and $\psi$ is a dominant weight on $N$, then $\cZ(N_\psi) \vee \cZ(c(N)) = N_\psi ' \cap c(N) \cong \cZ(N_\psi) \ovt \{ \psi\}''$, hence $N \vee \cZ(c(N)) = c(N)$.
\end{rem}

\begin{prop}
Let $N \subset M$ be an inclusion of von Neumann algebras with expectations. Suppose that $N$ is of type $\III_1$. There exists a unique flow $\rb^\sharp : \R^*_+ \curvearrowright \rb^\sharp(N \subset c(M))$ such that :
\begin{enumerate}
\item $N$ is fixed by $\rb^\sharp$.
\item $\rb^\sharp_\lambda(x)=\rb^{\varphi}_\lambda(x)$ for all $\lambda \in \R^*_+$, all $\varphi \in \cP(N)$ and all $x \in\rb(N \subset c(M),\varphi)$.
\end{enumerate}
Moreover, $\rb^\sharp(N \subset c(M))^{\rb^\sharp}=N \vee (N' \cap c(M))$.
\end{prop}
\begin{proof}
The uniqueness is clear. Let us prove the existence. By Proposition \ref{structure of bicentralizer extension}, we have $$\rb^\sharp(N \subset c(M)) \cong \left( N_\psi \ovt_{\cZ(N_\psi)} (N_\psi' \cap c(M)) \right) \rtimes_{\Ad(u)} \R^*_+$$
where $\psi \in \cP(N)$ is a dominant weight and $u$ is a $\psi$-scaling group. Let $\alpha=\Ad(u)|_{N_\psi' \cap c(M)}$. Since $N$ is of type $\III_1$, then $\cZ(N_\psi)=\cZ(N)$ is fixed by $\alpha$. Thus we have an action
$$\id \otimes \alpha : \R^*_+ \curvearrowright N_\psi \ovt_{\cZ(N_\psi)} (N_\psi' \cap c(M)).$$
This action $\id \otimes \alpha$ commutes with $\Ad(u)$. Therefore, it extends to an action $$\rb^\sharp : \R^*_+ \curvearrowright \left( N_\psi \ovt_{\cZ(N_\psi)} (N_\psi' \cap c(M)) \right) \rtimes_{\Ad(u)} \R^*_+$$
that fixes $N_\psi$ and $1 \rtimes_{\Ad(u)} \R^*_+$. Now, if we view $\rb^\sharp$ as an action on $\rb^\sharp(N \subset c(M))$, we see that it fixes $N$ and $N' \cap c(M)$. Moreover, 
$$\rb^\sharp_\lambda((\psi \circ T)^{\ri t})=\lambda^{\ri t} (\psi \circ T)^{\ri t}$$ 
for all $(\lambda,t) \in \R^*_+ \times \R$ and $T \in \cP(M,N)$.
This shows that $\rb^\sharp|_{c_T(N)}$ is the trace scaling action, hence $$\rb^\sharp_\lambda((\varphi \circ T)^{\ri t})=\lambda^{\ri t} (\varphi \circ T)^{\ri t}$$
for all $(\lambda,t) \in \R^*_+ \times \R$ and $\varphi \in \cP(N)$. We conclude that $\rb^{\sharp}|_{\rb(N \subset c(M),\varphi)}=\rb^\varphi$ as we wanted.
Finally, by Theorem \ref{dual action crossed product}, we have $\rb^\sharp(N \subset c(M))^{\rb^\sharp}=N \vee (N' \cap c(M))$.
\end{proof}

\section{The analytic bicentralizer} \label{section analytic}
In this section, we will study the \emph{analytic bicentralizer} using ultrapowers, following \cite{HI15} and \cite{AHHM18}. In \cite{Ok21}, it was observed that the ultrapower approach can be used to define the analytic bicentralizer not only for states but also for \emph{strictly semifinite weights}. This will be useful for us. Recall that a weight $\varphi \in \cP(M)$ is \emph{strictly semifinite} if $\varphi|_{M_\varphi}$ is semifinite. In that case, there exists a unique $\varphi$-preserving conditional expectation $E \in \cE(M,M_\varphi)$. We denote by $\cP_s(M) \subset \cP(M)$ the set of all normal faithful strictly semifinite weights. 

Recall that if $N \subset M$ is with expectations then we have natural inclusions $N^\omega \subset M^\omega$ and $M^\omega \subset c(M)^\omega$ (see \cite{AH12}).
\begin{df}
Let $N \subset M$ be a compatible inclusion of von Neumann algebras. For every strictly semifinite $\varphi \in \cP_s(N)$, we define
$$ \rB(N\subset M, \varphi)=(N^\omega_{\varphi^\omega})' \cap M$$
$$ \rB(N\subset c(M), \varphi)=(N^\omega_{\varphi^\omega})' \cap c(M)$$
where $\omega$ is a cofinal ultrafilter on $\N$.
\end{df}
A priori, this definition depends on the choice of the ultrafilter but actually it does not as observed in \cite[Proposition 3.2]{HI15} (for states) and \cite[Proposition 3.13]{Ok21} for strictly semifinite weights. This also true for $c(M)$ thanks to the following proposition.

\begin{prop} \label{anal bicentralizer core}
Let $N \subset M$ be an inclusion of von Neumann algebras with expectations. Take a strictly semifinite weight $\varphi \in \cP_s(N)$ and $T \in \cP(M,N)$. Then under the identification $c(M)=M \rtimes_{\sigma^{\varphi \circ T}} \R$, we have  
$$\rB(N \subset c(M), \varphi)=\rB(N \subset M, \varphi) \rtimes_{\sigma^{\varphi \circ T}} \R.$$
\end{prop}
\begin{proof}
By applying the first part of Proposition \ref{dixmier in crossed product} to $P=N^\omega_{\varphi^\omega}$ and $\alpha=\sigma^{(\varphi \circ T)^\omega} : \R \curvearrowright M^\omega$, we get
$$  (N^\omega_{\varphi^\omega})' \cap (M^\omega \rtimes_{\sigma^{(\varphi \circ T)^\omega}} \R) = ((N^\omega_{\varphi^\omega})' \cap M^\omega ) \rtimes_{\sigma^{(\varphi \circ T)^\omega}} \R.$$
Taking the intersection with $M \rtimes_{\sigma^{\varphi \circ T}} \R$, we conclude that 
$$  (N^\omega_{\varphi^\omega})' \cap (M \rtimes_{\sigma^{\varphi \circ T}} \R) = ((N^\omega_{\varphi^\omega})' \cap M ) \rtimes_{\sigma^{\varphi \circ T}} \R.$$
\end{proof}

\begin{prop} \label{amplification anal bic}
Let $N \subset M$ be an inclusion of von Neumann algebras with expectations. Let $F$ be a type $\rm I$ factor with its canonical trace $\rm Tr$. For every $\varphi \in \cP_s(N)$, we have
$$\rB(N \ovt F \subset M \ovt F, \varphi \otimes {\rm Tr}) =\rB(N \subset M, \varphi) \otimes 1.$$
The same is true if we replace $M$ by $c(M)$.
\end{prop}
\begin{proof}
This follows from the definition and the fact that 
$$(N \ovt F)^\omega_{(\varphi \otimes {\rm Tr})^\omega}=N_{\varphi^\omega}^\omega \ovt F$$ for every cofinal ultrafilter $\omega$.
\end{proof}

\begin{prop} \label{anal image expectation}
Let $N \subset M$ be an inclusion of von Neumann algebras with expectations. Take $\varphi \in \cP_s(N)$ and a cofinal ultrafilter $\omega$ on $\N$. Then $\rB(N \subset M,\varphi)$ is the image of $(N^\omega_{\varphi^\omega})' \cap M^\omega$ by the canonical conditional expectation $E : M^\omega \rightarrow M$.
\end{prop}
\begin{proof}
If $\varphi(1) < +\infty$, this follows from \cite[Proposition 3.3]{AHHM18}. If $\varphi$ has infinite multiplicity, then it is of the form $\varphi =\psi \otimes \mathrm{Tr}$ for some finite weight $\psi$ and we can apply Proposition \ref{amplification anal bic} and the previous case. Finally, if $\varphi$ is arbitrary, then $\varphi \otimes \mathrm{Tr}$ has infinite multiplicity and we can apply the previous case together with Proposition \ref{amplification anal bic} again.
\end{proof}

\begin{prop} \label{anal expectation}
Let $N \subset M$ be an inclusion of von Neumann algebras with expectations. Let $\varphi$ be a faithful normal state on $N$. There exists a unique faithful normal conditional expectation $E_{\rB(N\subset M, \varphi)}$ from $M$ onto $\rB(N \subset M,\varphi)$ that preserves $\varphi \circ T$ for every $T \in \cP(M,N)$.
\end{prop}
\begin{proof}
Let $\omega$ be a cofinal ultrafilter on $\N$. Let $F$ be the faithful normal conditional expectation from $M^\omega$ onto $(N_{\varphi^\omega}^\omega)' \cap M^\omega$ that preserves $(\varphi \circ T)^\omega$. Then $F$ does not depend on the choice of $T \in \cP(M,N)$. Let $E : M^\omega \rightarrow M$ be the canonical conditional expectation. Then $E_{\rB(N\subset M, \varphi)} := E \circ F|_{M}$ is a faithful normal conditional expectation from $M$ onto $\rB(N \subset M,\varphi)$ (see \cite[Proposition 3.3]{AHHM18}).
\end{proof}

\begin{prop} \label{anal bic corner reduction}
Let $N \subset M$ be an inclusion of von Neumann algebras with expectations. For every $\varphi \in \cP_s(N)$ and every projection $e \in N_\varphi$ or $e \in N' \cap M$, we have 
$$e\rB(N \subset M, \varphi)e=\rB(eNe \subset eMe, \varphi_e)$$
and
$$e\rB(N \subset c(M), \varphi)e=\rB(eNe \subset c(eMe), \varphi_e)$$
\end{prop}
\begin{proof}
Let $\omega$ be a cofinal ultrafilter on $\N$. Then 
$$e((N^\omega_{\varphi^\omega})' \cap M^\omega)e = (eN_{\varphi^\omega}^\omega e)' \cap (eM^\omega e)=((eNe)^\omega_{\varphi_e^\omega})' \cap (eMe)^\omega.$$
Appliying the canonical conditional expectation $E : M^\omega \rightarrow M$ and using Proposition \ref{anal image expectation} we obtain
$$e\rB(N \subset M, \varphi)e=\rB(eNe \subset eMe, \varphi_e).$$ 
For the second part, we choose $T \in \cP(M,N)$ such that $\sigma^{\varphi \circ T}$ fixes $e$ and we apply Proposition \ref{anal bicentralizer core}.
\end{proof}

\begin{prop} \label{anal tensor product}
Let $N_i \subset M_i, \: i=1,2$ be an inclusion of von Neumann algebras with expectations. For every  $\varphi_i \in \cP_s(N_i), \: i=1,2$, we have
$$ \rB(N_1 \ovt N_2 \subset M_1 \ovt M_2, \varphi_1 \otimes \varphi_2) \subset \rB(N_1 \subset M_1, \varphi_1) \ovt \rB(N_2 \subset M_2, \varphi_2).$$
\end{prop}
\begin{proof}
This follows from the fact that $$(N_1)_{\varphi_1^\omega}^\omega \ovt (N_2)_{\varphi_2^\omega}^\omega \subset (N_1 \ovt N_2)^\omega_{(\varphi_1 \otimes \varphi_2)^\omega}.$$
\end{proof}

\begin{prop} \label{bicentralizer of tensor product with finite}
Let $N \subset M$ be an inclusion of von Neumann algebras with expectation. Let $\varphi$ be a faithful normal state on $N$. Let $(A,\tau)$ be a tracial abelian von Neumann algebra. Then $$\rB( A \ovt N \subset A \ovt M, \tau \otimes \varphi)= A \ovt\rB(N \subset M, \varphi).$$
\end{prop}
\begin{proof}
The inclusion
$$\rB( A \ovt N \subset A \ovt M, \tau \otimes \varphi) \subset A \ovt\rB(N \subset M,\varphi)$$
is a special case of Proposition \ref{anal tensor product}. Let us prove the other inclusion.

Extend $\varphi$ to $M$ by using a faithful normal conditional expectation from $M$ to $N$. Write $A=\rL^{\infty}(T,\mu)$ where $(T,\mu)$ is a probability space and the trace $\tau$ is given by integration with respect to $\mu$.  Let $a \in \rB(N \subset M,\varphi)$. Let $x_n=(t \mapsto x_n(t)) \in A \ovt N =\rL^{\infty}(T,\mu,N)$ be a sequence in the unit ball of $A \ovt N$ that asymptotically centralizes $\tau \otimes \varphi $.  Since
$$ \| x_n (\tau \otimes \varphi)^{1/2}-(\tau \otimes \varphi)^{1/2} x_n\|^{2}=\int_T \|x_n(t) \varphi^{1/2}-\varphi^{1/2} x_n(t)\|^{2} \: \mathrm{d} \mu(t),$$
the function $t \mapsto \| x_n(t) \varphi - \varphi x_n(t)\|$ converges almost surely to $0$. Therefore, since $a \in \rB(N \subset M,\varphi)$, the function $t \mapsto \| x_n(t) a -a x_n(t)\|_\varphi$ converges almost surely to $0$. 
$$ \lim_{n \to +\infty} \mu( \{ t \in T \mid \|x_n(t) \varphi^{1/2}-\varphi^{1/2} x_n(t)\| \leq \delta \}) =1.$$
By the dominated convergence theorem, we conclude that 
$$\| x_n ( 1 \otimes a) - (1 \otimes a)x_n \|_{\tau \otimes \varphi}^2 =\int_T \|x_n(t) a-a x_n(t)\|_\varphi^{2} \: \mathrm{d} \mu(t)$$
converges to $0$, i.e.\ $1 \otimes a \in \rB(A \ovt M, \tau \otimes \varphi)$. We proved that $1 \otimes\rB(N \subset M,\varphi) \subset \rB( A \ovt N \subset A \ovt M, \tau \otimes \varphi)$. Clearly, we also have $A \otimes 1 \subset \rB( A \ovt N \subset A \ovt M, \tau \otimes \varphi)$. Therefore, we have $$A \ovt\rB(N \subset M,\varphi) \subset \rB( A \ovt N \subset A \ovt M, \tau \otimes \varphi).$$ 
\end{proof}

By a similar argument, we can show the following.

\begin{prop} \label{desintegration bicentralizer}
Let $(X,\mu)$ be a standard borel probability space. Let $x \mapsto M_x$ a measurable field of von Neumann algebras with separable predual and $x \mapsto N_x \subset M_x$ a measurable field of von Neumann subalgebras with expectation. Let $x \mapsto \varphi_x$ be a measurable field of faithful normal states on $N_x$. Let 
$$N=\int^\oplus_X N_x \: \rd \mu(x), \quad M= \int^\oplus_X M_x \:  \rd \mu(x) \quad \text{ and } \quad  \varphi = \int^\oplus_X \varphi_x  \: \rd \mu(x).$$
Then we have
$$ \rB(N \subset M, \varphi) = \int^\oplus_X \rB(N_x \subset M_x, \varphi_x) \: \rd \mu(x).$$
\end{prop}
\begin{proof}
For $a \in M$ and $\delta > 0$, we let $\varepsilon(a,\delta)=\sup \{ \| ua-au\|_\varphi \mid u \in \cU(N), \| u\varphi - \varphi u\| \leq \delta \}$. We have $a \in \rB(N \subset M,\varphi)$ if and only if $\varepsilon(a,\delta) \to 0$ when $\delta \to 0$.

 If we write $a$ as a measurable function $a : x \mapsto a_x \in M_x$, then for every $\delta > 0$, by using appropriate measurable selections of unitaries $u : x \mapsto u_x$, we see that $$\varepsilon(a,\delta)^2 \geq \int_X \varepsilon(a_x, \delta)^2 \: \rd \mu(x).$$
 By letting $\delta \to 0$, this shows that if $a \in \rB(N \subset M,\varphi)$, then $a_x \in \rB(N_x \subset M_x, \varphi_x)$ for almost every $x \in X$.
 
 For the other direction we proceed exactly as in the previous proposition.
\end{proof}

\subsection{Transition maps in the type $\III_1$ case}

The following theorem is due to Connes and St\o rmer in the factorial separable case \cite{CS78}. For the non-factorial non-separable case, we refer to \cite{HS90}. The strictly semifinite case follows easily from the state case.
\begin{thm}[Connes-St\o rmer] \label{connes stormer}
Let $N$ be a von Neumann algebra of type $\III_1$. Let be $\varphi, \psi \in \cP_s(N)$ be two strictly semifinite weights such that $\varphi|_{\cZ(N)}=\psi|_{\cZ(N)}$. For every cofinal ultrafilter $\omega$ on $\N$, there exists a unitary $u \in N^\omega$ such that $u\varphi^\omega u^*=\psi^\omega$.
\end{thm}

The idea of the following theorem appears in \cite{AHHM18}.

\begin{thm} \label{bicentralizer equivariance and flow}
Let $N \subset M$ be an inclusion of von Neumann algebras with expectations. Suppose that $N$ is of type $\III_1$. 
\begin{enumerate}
\item Take $\varphi,\psi \in \cP_s(N)$. For every $x \in\rB(N \subset M,\varphi)$, there exists a unique element $\beta^{\psi,\varphi}(x) \in\rB(N \subset M,\psi)$ such that $ax=\beta^{\psi,\varphi}(x)a$ for every $a \in N^\omega_{\psi^\omega,\varphi^\omega}$. Moreover, $$\beta^{\psi,\varphi} :\rB(N \subset M, \varphi) \rightarrow\rB(N \subset M,\psi)$$ is an isomorphism of von Neumann algebras.
\item Take $\varphi,\phi,\psi \in \cP_s(N)$. Then $\beta^{\psi,\phi} \circ \beta^{\phi,\varphi}=\beta^{\psi,\varphi}$.
\item Take $\varphi,\psi \in \cP_s(N)$ and a projection $e \in N_\varphi \cap N_\psi$ or $e \in N' \cap M$. Then $\beta^{\psi_e,\varphi_e}(xe)= \beta^{\psi,\varphi}(x)e$ for every $x \in\rB(N \subset M,\varphi)$.
\item Take $\varphi \in \cP_s(N)$ and $(\psi_n)_{n \in \N}$ a sequence of faithful finite weights in $N_*^+$ that converges in norm to some faithful $\psi \in N_*^+$. Then $\beta^{\psi_n,\varphi}$ converges pointwise *-strongly to $\beta^{\psi,\varphi}$.
\item Take $\varphi \in \cP_s(N)$. For every $\lambda \in \R^*_+$, we have $\rB(N \subset M,\lambda \varphi)=\rB(N \subset M,\varphi)$ and $\beta_\lambda^\varphi := \beta^{\lambda \varphi,\varphi} \in \Aut(\rB(N \subset M, \varphi))$. Moreover, $\beta^\varphi :  \lambda \mapsto \beta_\lambda^\varphi$ is a continuous group action of $\R^*_+$ on $\rB(N \subset M, \varphi)$.
\end{enumerate}
The same properties hold if we replace $M$ by $c(M)$.
\end{thm}
\begin{proof}
(1) Using Proposition \ref{amplification anal bic}, we may assume that $\varphi$ and $\psi$ have infinite multiplicity, i.e.\ $\varphi|_{\cZ(N)}$ and $\psi|_{\cZ(N)}$ are both purely infinite. By Theorem \ref{connes stormer}, there exists a unitary $u \in N^\omega_{\varphi^\omega}$ such that $u\varphi^{\omega}u^*=\psi^{\omega}$. We then proceed exactly as in the proof of \cite[Theorem A.(\rm i)]{AHHM18}.

(2) Obvious from (1).

(3) Fix $x \in\rB(N \subset M,\varphi)$. Suppose first that $e \in N_\varphi \cap N_\psi$. Take $a \in (eNe)^\omega_{\psi^\omega_e,\varphi^\omega_e}$. Then, viewing $a$ as an element of $eN^\omega e$, we have that $a \in N^\omega_{\psi^\omega,\varphi^\omega}$. Thus $axe=ax=\beta^{\psi,\varphi}(x)a=\beta^{\psi,\varphi}(x)ea$. This shows that $\beta^{\psi_e,\varphi_e}(xe)= \beta^{\psi,\varphi}(x)e$. Suppose now that $e \in N' \cap M$. Take $z \in \cZ(N)$ the unique central projection such that $x \mapsto xe$ is an isomorphism from $Nz$ to $Ne$. Take $a \in (eNe)^\omega_{\psi^\omega_e,\varphi^\omega_e}$ and let $b \in N^\omega z$ be the unique element such that $be=a$. Then $b \in N^\omega_{\psi^\omega,\varphi^\omega}$. Therefore, we have
$ axe=bxe=\beta^{\psi,\varphi}(x)be=\beta^{\psi,\varphi}(x)e a$ which shows that $\beta^{\psi_e,\varphi_e}(xe)=\beta^{\psi,\varphi}(x)e$.

(4) Since $\beta^{\psi_n,\varphi}=\beta^{\psi_n,\psi} \circ \beta^{\psi,\varphi}$, we may assume that $\varphi=\psi$. Fix $x \in\rB(N \subset M,\psi)$, we have to show that $\beta^{\psi_n,\psi}(x)$ converges strongly to $x$. Since $\psi_n$ converges in norm to $\psi$, we can find a sequence of partial isometries $v_n \in N^\omega_{\psi_n^\omega,\psi^\omega}$ such that $v_n^*v_n$ and $v_nv_n^*$ converge strongly to $1$. Then $\|v_n \psi^\omega-\psi^\omega v_n\|$ converges strongly to $0$. Take $\eta$ a second cofinal ultrafilter on $\N$ and let $u=(v_n)^{\eta} \in (N^\omega)^{\eta}=N^{\eta \otimes \omega}$. Take $x \in\rB(N \subset M,\psi)$. Then $uxu^*=x$. This means that $\lim_{n \to \eta} v_nxv_n^*=x$ strongly. Since this holds for every cofinal ultrafilter $\eta$ on $\N$, we conclude that $v_nxv_n^*$ converges strongly to $x$. Now, since $v_nxv_n^*=\beta^{\psi_n,\psi}(x)v_nv_n^*$ and $v_nv_n^*$ converges strongly to $1$, we conclude that $\beta^{\psi_n,\psi}(x)$ converges strongly to $x$. 

(5) The equality $\rB(N \subset M, \lambda \varphi)=\rB(N \subset M, \varphi)$ is obvious. The fact that $\beta^\varphi$ is a group action follows from (2). To prove that it is continuous, it is enough, by \cite[Proposition X.1.2]{Ta03}, to show that for every $x \in \rB(N \subset M,\varphi)$ the map $\lambda \mapsto \beta_\lambda^\varphi(x)$ is continuous in the strong topology (or weak* topology). For the continuity of $\beta^\varphi$ we use (4) when $\varphi$ is finite. In the case where $\varphi$ is not finite, we use (3). Indeed, take $e \in N_\varphi$ such that $\varphi(e) < +\infty$. Then (3) tells us that $\beta_\lambda^\varphi(x)e=\beta_\lambda^{\varphi_e}(xe)$ for every $x \in\rB(N \subset M, \varphi)$ and every $\lambda \in \R^*_+$. This shows that $\lambda \mapsto \beta^\varphi_\lambda(x)e$ is strongly continuous for every $e \in N_\varphi$ with $\varphi(e) < +\infty$, hence $\lambda \mapsto \beta^\varphi_\lambda(x)$ is strongly continuous. 
\end{proof}

The following lemma is new and will play an important role in the study of the bicentralizer conjecture.
\begin{lem} \label{weight commute fixed point}
Let $N \subset M$ be an inclusion of von Neumann algebras with expectations. Suppose that $N$ is of type $\III_1$. Take  $\varphi,\psi \in \cP_s(N)$ such that $\varphi$ and $\psi$ commute. Then $\beta^{\psi,\varphi}(x)=x$ for every $x \in\rB(N \subset M, \varphi)^{\beta^\varphi}$. In particular, $\rB(N \subset M, \varphi)^{\beta^\varphi}=\rB(N \subset M, \psi)^{\beta^\psi}$.

The same is true if we replace $M$ by $c(M)$.
\end{lem}
\begin{proof}
 Since $\varphi$ and $\psi$ commute, we can write $\psi=\varphi_h$ for some positive operator $h$ affiliated with $N_\varphi \cap N_\psi$. Assume first that $h$ has finite spectrum. Then we can write $h=\sum_{i=1}^n \lambda_i e_i$ for some $\lambda_i \in \R^*_+$ and some partition of unity $(e_i)_{i =1}^n$ in $N_\varphi \cap N_\psi$. Take $x \in\rB(N \subset M, \varphi)$. Then by Theorem \ref{bicentralizer equivariance and flow}.(3), we have
$$ \beta^{\psi,\varphi}(x) =\sum_{i=1}^n \beta^{\psi,\varphi}(x)e_i = \sum_{i=1}^n \beta^{\psi_{e_i},\varphi_{e_i}}(xe_i) =\sum_{i=1}^n \beta^{\lambda_i\varphi_{e_i},\varphi_{e_i}}(xe_i) = \sum_{i=1}^n \beta_{\lambda_i}^{\varphi_{e_i}}(xe_i) =\sum_{i=1}^n \beta_{\lambda_i}^{\varphi}(x)e_i.$$
Therefore, if $x$ is fixed by $\beta^\varphi$, then $\beta^{\psi,\varphi}(x)=\sum_{i=1}^n xe_i=x$.

Next, we deal with the case where $\psi$ is finite, i.e.\ $\varphi(h) < +\infty$. Take a sequence of positive elements $h_n \in N_\varphi$ with finite spectrum such that $\psi_n=\varphi_{h_n}$ converges in norm to $\psi$.  Then for every $x \in\rB(N \subset M,\varphi)^{\beta^\varphi}$ and every $n \in \N$, we have $\beta^{\psi_n,\varphi}(x)=x$, hence $\beta^{\psi,\varphi}(x)=x$ thanks to Theorem \ref{bicentralizer equivariance and flow}.(4).

Finally, we deal with the general case. Since $\psi$ and $\varphi$ commutes, $\psi$ is of the form $\psi = \psi|_{N_\varphi} \circ E$ where $E : N \rightarrow N_\varphi$ is the $\varphi$-preserving normal conditional expectation. In particular, $\psi|_{N_\varphi}$ is strictly semifinite, i.e.\ the restriction of $\psi$ to $N_\varphi \cap N_\psi$ is semifinite.  Now, fix $x \in\rB(N \subset M,\varphi)^{\beta^\varphi}$. For every $e \in N_\varphi \cap N_\psi$ such that $\psi(e) < +\infty$, we have that $\psi_e$ is finite and $xe \in\rB(Ne \subset eMe,\varphi_e)^{\beta^{\varphi_e}}$ by Theorem \ref{bicentralizer equivariance and flow}.(3). By the previous case, we get $\beta^{\psi,\varphi}(x)e=\beta^{\psi_e,\varphi_e}(xe)=xe$. By making $e$ increase to $1$, we conclude that $\beta^{\psi,\varphi}(x)=x$.
\end{proof}

\begin{prop} \label{anal bic intermediate}
Let $N \subset M$ be an inclusion of von Neumann algebras with expectations. Suppose that $N$ is of type $\III_1$. Then the subalgebras
$$ \rB^\sharp(N \subset M) = N \vee \rB(N \subset M,\varphi)$$
$$ \rB^\sharp(N \subset c(M)) = N \vee \rB(N \subset c(M),\varphi)$$
do not depend on the choice of $\varphi \in \cP_s(N)$. Moreover, there exists a unique faithful normal conditional expectation from $M$ onto $\rB^\sharp(N \subset M)$ and as subalgebras of $c(M)$, we have
$$\rB^\sharp(N \subset c(M)) = c(\rB^\sharp(N \subset M)).$$
\end{prop}
\begin{proof}
Take $\psi \in \cP_s(N)$. It follows from Theorem \ref{connes stormer} and Theorem \ref{bicentralizer equivariance and flow}.(1) that $\rB(N \subset M,\psi) \subset N\vee \rB(N \subset M,\varphi)$ and the same holds if we exchange the role of $\varphi$ and $\psi$. This shows that the subalgebra $\rB^\sharp(N \subset M)$ does not depend on the choice of $\varphi$ and the same argument works for $\rB^\sharp(N \subset c(M))$.

Take $T \in \cP(M,N)$. Then $\rB^\sharp(N \subset M)$ is clearly invariant under the modular flow $\sigma^{\varphi \circ T}$ and $\varphi \circ T$ is semifinite on $\rB^\sharp(N \subset M)$. Therefore, there exists a $\varphi \circ T$-preserving conditional expectation onto from $M$ onto $\rB^\sharp(N \subset M)$. It is the unique normal conditional expectation because $\rB^\sharp(N \subset M)' \cap M \subset N' \cap M \subset \rB^\sharp(N \subset M)$. Finally, Proposition \ref{anal bicentralizer core} shows that under the identification $c(M)=M \rtimes_{\sigma^{\varphi \circ T}} \R$, we have
$$ \rB^\sharp(N \subset c(M))=\rB^\sharp(N \subset M)\rtimes_{\sigma^{\varphi \circ T}} \R.$$
\end{proof}

\subsection{Invariance under ucp maps}

\begin{prop} \label{transition map is in ucp}
Let $N \subset M$ be an inclusion of von Neumann algebra with expectations. Suppose that $N$ is of type $\III_1$, then for every $\varphi,\psi \in \cP_s(N)$, there exists a normal ucp map $f \in \cV^{\psi,\varphi}(N \subset M)$ such that $\beta^{\psi,\varphi}(x)=f(x)$ for every $x \in \rB(N \subset M,\varphi)$. If $\varphi$ and $\psi$ have infinite multiplicity we can choose $f$ in $\cD(N \subset M) \cap \cV^{\psi,\varphi}(N \subset M)$.

The same is true if we replace $M$ by $c(M)$.
\end{prop}
\begin{proof}
If $\varphi$ and $\psi$ are both have infinite multiplicity then by Theorem \ref{connes stormer}, there exists a unitary $u \in N^\omega$ such that $u\varphi^\omega u^*=\psi^\omega$. By Theorem \ref{bicentralizer equivariance and flow}.(1), we have $\beta^{\psi,\varphi}(x)=uxu^*$ for all $x \in \rB(N \subset M,\varphi)$. Therefore, we can take $f$ the map given by $f(x)=E(uxu^*)$ for all $x \in M$ where $E : M^\omega \rightarrow M$ is the canonical conditional expectation. We have $f \in \cD(N \subset M)  \cap \cV^{\psi,\varphi}(N \subset M)$ by Proposition \ref{ucp map from ultrapower}. 

If $\varphi$ and $\psi$ are arbitrary, we use Proposition \ref{amplification ucp map} and Theorem \ref{bicentralizer equivariance and flow}.(3).
\end{proof}

\begin{prop} \label{conditional expectation is in ucp}
Let $N \subset M$ be an inclusion of von Neumann algebras with expectations. 
\begin{enumerate}[(1)]
\item  If $\varphi$ is a faithful normal state on $N$, then $E_{\rB(N\subset M, \varphi)} \in \cD(N \subset M) \cap \cV^\varphi(N \subset M)$.
\item If $\varphi \in \cP_s(N)$, then $\cD(N \subset M) \cap \cV^\varphi(N \subset M)$ contains some conditional expectation (not necessarily normal) onto $\rB(N \subset M,\varphi)$.
\item If $\varphi \in \cP_s(N)$ and $N$ is of type $\III_1$, then $\cV(N \subset M)$ contains some conditional expectation onto $\rB(N \subset M,\varphi)^{\beta^\varphi}$. If moreover $\varphi \in \cP_s(N)$ has infinite multiplicity then $\cD(N \subset M)$ contains some conditional expectation onto $\rB(N \subset M,\varphi)^{\beta^\varphi}$.
\end{enumerate}
The same is true if we replace $M$ by $c(M)$.
\end{prop}
\begin{proof}
(1) Let $\omega$ be a cofinal ultrafilter on $\N$. Let $F$ be the faithful normal conditional expectation from $M^\omega$ onto $(N_{\varphi^\omega}^\omega)' \cap M^\omega$ that preserves $(\varphi \circ T)^\omega$. Then $E_{\rB(N\subset M, \varphi)} = E \circ F|_{M}$ (see Proposition \ref{anal expectation}). Since $F$ lies in the weak* closed convex hull $\{ \Ad(u) \mid u \in \cU(N_{\varphi^\omega}^\omega) \}$, we have $E_{\rB(N\subset M, \varphi)} \in \cD(N \subset M)$ and Proposition \ref{ucp map from ultrapower} tells us that $E_{\rB(N\subset M, \varphi)} \in \cV^\varphi(N \subset M)$. 

(2) By using Proposition \ref{amplification anal bic}, and since type $\rm I$ factors have the weak Dixmier property, we reduce to the case where $\varphi$ is finite.

(3) Take $E \in \cD(N \subset M)$ a conditional expectation onto $\rB(N \subset M,\varphi)$, then $\beta_\lambda^\varphi \circ E \in \cV(N \subset M)$ for all $\lambda \in \R^*_+$ thanks to Proposition \ref{transition map is in ucp}. Since $\R^*_+$ is amenable, we conclude that $\cV(N \subset M)$ contains some conditional expectation onto $\rB(N \subset M,\varphi)^{\beta^\varphi}$. If $\varphi$ has infinite multiplicity, then $\beta_\lambda^\varphi \circ E \in \cD(N \subset M)$ for all $\lambda \in \R^*_+$ and a similar conclusion holds.
\end{proof}

\begin{lem} \label{commutation relation vector}
Let $N \subset M$ be an inclusion of von Neumann algebras with a faithful normal conditional expectation $E_N$ and use it to view $\rL^2(N)$ as a subspace of $\rL^2(M)$. Let $\varphi$ be a faithful normal state on $M$ such that $\varphi =\varphi \circ E_N$. Take $x,y \in M$ such that $x \varphi^{1/2}=\varphi^{1/2}y$. If $x \in N' \cap M$, then $x \xi=\xi y$ for all $\xi \in \rL^2(N)$.
\end{lem}
\begin{proof}
Take $a \in N$. Then $xa\varphi^{1/2}=ax\varphi^{1/2}=a\varphi^{1/2}y$. Since $\{ a\varphi^{1/2} \mid a \in N \}$ is dense in $\rL^2(N)$, we conclude that $x\xi=\xi y$ for all $\xi \in \rL^2(N)$.
\end{proof}

The following is the key result of this section. It will be used in the proof of our main theorems. It is in this Theorem that we use the ultrapower implementation of binormal states of Section \ref{section ultrapower binormal}.

\begin{thm} \label{invariance bicentralizer ucp}
Let $N \subset M$ be an inclusion of von Neumann algebras with expectation. Let $\varphi$ be a faithful normal state on $N$. Take $f \in \cV^\varphi(N \subset M)$. If $f|_{\rB(N,\varphi)}$ is normal, then $f(x)=x$ for all $x \in\rB(N \subset M, \varphi)$.
\end{thm}
\begin{proof}
First, we deal with the case where $f$ is normal on $M$. Extend $\varphi$ to $M$ by letting $\varphi =\varphi \circ E_N$. Since $f \in \cV^\varphi(N \subset M)$, we can use Lemma \ref{binormal state vs ucp} to find a state $\Phi \in \B(\rL^2(M))^*$ such that :
\begin{itemize}
\item $\Phi( \lambda(a)\rho(b)) = \langle a\varphi^{1/2}f(b),\varphi^{1/2}\rangle$ for all $a,b \in M$.
\item $\Phi(e_N)=1$,  where $e_N : \rL^2(M) \rightarrow \rL^2(N)$ is the Jones projection associated to $E_N$.
\item $\Phi$ is a $1$-eigenstate for $\Delta_\varphi$.
\end{itemize}
Clearly, $\Phi \circ \lambda = \varphi$ is normal.  Since $f$ is normal, the state $\Phi \circ \rho = \varphi \circ f$ is also normal. Therefore, by Theorem \ref{ultrapower all type}, we can find some abelian algebra $A$, a cofinal ultrafilter $\omega$ on some directed set $I$ and a vector $\xi \in \rL^2((A \ovt M)^\omega)$ such that
$$ \langle (1 \otimes T)^\omega \xi,\xi \rangle = \Phi(T)$$
for all $T \in \B(\rL^2(M))$. In fact, if we choose some faithful normal state $\mu$ on $A$ and let $\phi=\mu \otimes \varphi$, we also have $\xi \in \rL^2((A \ovt M)^\omega_{\phi^{\omega}})$. Note also that $(1 \otimes e_N)^\omega$ projects $\rL^2((A \ovt M)^\omega)$ onto $\rL^2((A \ovt N)^\omega)$. Thus, the condition $\Phi(e_N)=1$ means that $\xi=(1 \otimes e_N)^\omega \xi \in \rL^2((A \ovt N)^\omega)$. Alltogether, we get $\xi \in \rL^2((A \ovt N)^\omega_{\phi^{\omega}})$. 

Take $x \in\rB(N \subset M, \varphi)$ such that $x$ is $\sigma^{\varphi}$-analytic. Then $x \varphi^{1/2}=\varphi^{1/2}y$ where $y=\sigma_{\ri/2}(x) \in\rB(N \subset M, \varphi)$. Since $f$ commutes with $\sigma^\varphi$, it also commutes with $\sigma^\varphi_{\ri/2}$. Thus we have $f(x) \varphi^{1/2}=\varphi^{1/2}f(y)$. 
Then, for every $a \in M$, we get
$$ \Phi(\lambda(a)\rho(y))=\langle a\varphi^{1/2}f(y), \varphi^{1/2}\rangle = \langle a f(x) \varphi^{1/2}, \varphi^{1/2} \rangle = \varphi(af(x)).$$
On the other hand, since $x \in\rB(N \subset M, \varphi)$, we also have $1 \otimes x \in \rB(A \ovt N \subset A \ovt M, \phi)$ by Proposition \ref{bicentralizer of tensor product with finite}. Moreover, we have $(1 \otimes x) \phi^{1/2} =\phi^{1/2} (1 \otimes y)$. By Lemma \ref{commutation relation vector} (applied to the inclusion $(A \ovt M)^\omega_{\phi^{\omega}} \subset (A \ovt M)^\omega$), we obtain $(1 \otimes x)\xi=\xi (1 \otimes y)$.

Then for every $a \in M$, we get
$$ \Phi(\lambda(a)\rho(y))=\langle (1 \otimes a) \xi (1 \otimes y) , \xi \rangle = \langle (1 \otimes ax) \xi, \xi \rangle = \Phi(\lambda(ax))=\varphi(ax).$$

From the two equations above, we conclude that $\varphi(af(x))=\varphi(ax)$ for all $a \in M$. This shows that $f(x)=x$ for every $\sigma^\varphi$-analytic $x \in\rB(N \subset M,\varphi)$. Since $f$ is normal and the $\sigma^\varphi$-analytic elements are dense in $\rB(N \subset M,\varphi)$, we conclude that $f(x)=x$ for all $x \in\rB(N \subset M,\varphi)$ as we wanted.

Now, we deal with the case where $f$ is not necessarily normal on all of $M$ but only on $\rB(N,\varphi)$. Let $f' =f \circ E_{\rB(N \subset M, \varphi)}$, which is still in $\cV^\varphi(N \subset M)$ by Proposition \ref{conditional expectation is in ucp}. It is sufficient to prove that $f'(x)=x$ for all $x \in\rB(N\subset M,\varphi)$. Choose a faithful normal conditional expectation $E \in \cE(M,N)$. We have
$$ E \circ f' = E \circ  f \circ E_{\rB(N \subset M, \varphi)} =f \circ E \circ E_{\rB(N \subset M, \varphi)}= f \circ E_{\rB(N,\varphi)}.$$
This shows in particular that $E \circ f'$ is normal, hence that $f'$ is normal. We conclude that $f'(x)=x$ for all $x \in\rB(N\subset M,\varphi)$ by the first case.
\end{proof}

\section{The bicentralizer conjecture} \label{section bicentralizer conjecture}

\begin{conj}[Relative Bicentralizer Conjecture]
Let $N \subset M$ be an inclusion of von Neumann algebras with expectations. Then $\rB(N \subset M, \varphi) =\rb(N \subset M, \varphi)$ for every strictly semifinite weight $\varphi \in \cP(N)$.
\end{conj}

\begin{prop}
Let $N \subset M$ be an inclusion of von Neumann algebras with expectations. Let $\varphi \in \cP(N)$ be a strictly semifinite weight. Then
$$\rb(N \subset M, \varphi) \subset \rB(N \subset M, \varphi).$$
\end{prop}
\begin{proof}
It is clear that $\rB(N \subset c(M), \varphi)$ contains $N' \cap c(M)$ and $\{ \varphi \circ T \}''$ for any $T \in \cP(M,N)$. This means by definition that 
$$\rb(N \subset c(M), \varphi) \subset \rB(N \subset c(M), \varphi).$$
By taking the intersection with $M$, we obtain the desired conclusion.
\end{proof}

\begin{prop} \label{projection reduction conjecture}
Let $N \subset M$ be an inclusion of von Neumann algebras with expectations. Take a strictly semifinite weight $\varphi \in \cP(N)$. If $\rB(N \subset M, \varphi)=\rb(N \subset M, \varphi)$, then for every projection $e \in N_\varphi$ or $e \in N' \cap M$, we have $\rB(eNe \subset eMe, \varphi_e)=\rb(eNe \subset eMe, \varphi_e)$. 

Conversely, if  $\cF \subset N_\varphi$ or $\cF \subset N' \cap M$ is a directed set of projections such that $\bigvee_{e \in \cF} e=1$ and $\rB(eNe \subset eMe, \varphi_e)=\rb(eNe \subset eMe, \varphi_e)$ for all $e \in \cF$, then $\rB(N \subset M, \varphi)=\rb(N \subset M, \varphi)$.
\end{prop}
\begin{proof}
Use Proposition \ref{alg bic corner reduction}, Proposition \ref{alg bic commutant corner reduction} and Proposition \ref{anal bic corner reduction}.
\end{proof}

\begin{thm} \label{conjecture for no type III1}
Let $N \subset M$ be an inclusion of von Neumann algebras with expectations. Suppose that $N$ has no type $\III_1$ summand. Then $\rB(N \subset M,\varphi)=\rb(N \subset M,\varphi)$ for every strictly semifinite $\varphi \in \cP(N)$.
\end{thm}
\begin{proof}
We deal first with the case where $N$ is semifinite. Thanks to Proposition \ref{projection reduction conjecture}, we may assume that $\varphi$ is finite. Take a trace $\tau \in \cP(N)$ and let $h = \frac{\rd \varphi}{\rd \tau}$. Then we know that $\rb(N \subset M, \varphi)=\{h\}'' \vee (N' \cap M)$. By taking finer and finer partitions of unity in $\{h \}''$ and approximating $h$ with step functions, we can find a sequence of states $\varphi_n$ on $N$ that converges in norm to $\varphi$, such that $h_n=\frac{\rd \varphi_n}{\rd \tau}$ has atomic spectrum for each $n \in \N$ and the sequence of subalgebras $\{ h_n\}''$ is increasing. Then 
$$\rb(N \subset M, \varphi)=(N' \cap M) \vee \{h\}''= \bigvee_{n \in \N} (N' \cap M) \vee \{h_n\}'' = \bigvee_{n \in \N}\rb(N \subset M, \varphi_n).$$

Fix $n \in \N$. Write $h_n$ in the form $h_n=\sum_{k \in \N} \lambda_k p_k$ where $(p_k)_{k \in \N}$ is a partition of unity in $N$, and $(\lambda_k)_{k \in \N}$ is a strictly decreasing sequence of positive numbers. Then $N_{\varphi_n} = \bigoplus_k p_kNp_k$ and $N_{\varphi_n}' \cap M=\bigoplus_n  p_n(N' \cap M)p_n$. Thus $N_{\varphi_n}' \cap M=(N' \cap M) \vee  \{ h_n\}''=\rb(N \subset M, \varphi_n)$, hence
$$\rb(N \subset M, \varphi)= \bigvee_{n \in \N} (N_{\varphi_n}' \cap M).$$
Now, take $x \in\rB(N \subset M, \varphi)$. Take $\cV$ a closed convex neighborhood of $x$ in the weak* topology. There exists $\delta > 0$ such that $uxu^* \in \cV$ for all $u \in \cU(N)$ with $\| u \varphi u^*-\varphi \| \leq \delta$. Take $n \in \N$ large enough so that $\| \varphi_n-\varphi \| \leq \frac{\delta}{2}$. Then we get $uxu^* \in \cV$ for every $u \in \cU(N_{\varphi_n})$. Since $\cV$ is a weak* closed convex subset of $M$, and $N_{\varphi_n} \subset M$ has the weak Dixmier property, this implies that $\cV$ contains an element of $N_{\varphi_n}' \cap M$. Since $\cV$ was arbitrary, we conclude that $x$ is in the weak* closure of $\bigcup_{n \in N} (N_{\varphi_n}' \cap M)$ hence $x \in\rb(N \subset M, \varphi)$ as we wanted. This concludes the proof in the semifinite case.

Now, suppose that $N$ is an arbitrary von Neumann algebra with no type $\III_1$ summand. Take $Q$ as in Corollary \ref{no III1 semifinite reduction} (with $\varphi$ instead of $\psi$). Then we have $\rb(N \subset M,\varphi)=\rb(Q \subset M, \varphi|_Q)$. By the semifinite case, we know that $\rb(Q \subset M, \varphi|_Q) =\rB(Q \subset M, \varphi|_Q)$ and since $Q \subset N$, we have $\rB(N \subset M, \varphi) \subset \rB(Q \subset M, \varphi|_Q)$. We conclude that $\rB(N \subset M, \varphi) \subset \rb(Q \subset M, \varphi|_Q)=\rb(N \subset M,\varphi)$ as we wanted.
\end{proof}

\begin{prop} \label{transition extends alg anal}
Let $N \subset M$ be an inclusion of von Neumann algebras with expectations. Suppose that $N$ is of type $\III_1$. Let $\varphi, \psi \in \cP(N)$ be two strictly semifinite weights. Then the isomorphism
$$ \beta^{\psi,\varphi} : \rB(N \subset M, \varphi) \rightarrow \rB(N \subset M,\psi)$$
is an extension of the isomorphism
$$ \rb^{\psi,\varphi} : \rb(N \subset M, \varphi) \rightarrow \rb(N \subset M,\psi).$$
In particular, the flow
$$\beta^\varphi : \R^*_+ \curvearrowright \rB(N \subset M,\varphi)$$ is an extension of the flow 
$$\rb^\varphi : \R^*_+ \curvearrowright \rb(N \subset M,\varphi).$$
The same is true if we replace $M$ by $c(M)$.
\end{prop}
\begin{proof}
Use Proposition \ref{transition map is in ucp} and Theorem \ref{ucp map transition III1}.
\end{proof}

\begin{prop} \label{conjecture with intermediate}
Let $N \subset M$ be an inclusion of von Neumann algebras with expectations. Suppose that $N$ is of type $\III_1$. Then $N \subset M$ satisfies the bicentralizer conjecture if and only if $N$ satisfies the bicentralizer conjecture and $\rB^\sharp(N \subset M)=\rb^\sharp(N \subset M)$.
\end{prop}
\begin{proof}
The only if direction is obvious. We prove the converse. We may assume that $N$ is countably decomposable. Let $\varphi \in N_*$ be a faithful state. Take $T \in \cP(M,N)$. Take the conditional expectation $E \in \cE( c(M), \rB(N \subset c(M),\varphi))$ induced by the conditional expectation of Proposition \ref{conditional expectation is in ucp}. Since $c_T(N) \cdot (N' \cap c(M))$ is dense in $\rb^\sharp(N \subset c(M))$ and $$E(c_T(N))=\rB(N \subset c_T(N),\varphi)=\rB(N,\varphi) \vee \{ \varphi \circ T\}''=\cZ(N) \vee \{ \varphi \circ T\}'',$$
we conclude that
$$\rB(N \subset c(M),\varphi) = E(\rb^\sharp(N \subset c(M)))= \cZ(N) \vee \{ \varphi \circ T\}'' \vee (N' \cap c(M))=\rb(N \subset c(M),\varphi).$$
\end{proof}

\begin{prop} \label{conjecture iff fixed point}
Let $N \subset M$ be a an inclusion of von Neumann algebras with expectations. Suppose that $N$ is of type $\III_1$. Let $\varphi \in \cP(N)$ be a strictly semifinite weight. Then the inclusion $N \subset M$ satisfies the bicentralizer conjecture if and only if $$\rB(N \subset c(M),\varphi)^{\beta^\varphi}= N' \cap c(M).$$
\end{prop}
\begin{proof}
The only if direction follows from Theorem \ref{ucp map transition III1}.(5) and Proposition \ref{transition extends alg anal}.

For the if direction, suppose that $\rB(N \subset c(M),\varphi)^{\beta^\varphi}= N' \cap c(M)$. Pick $T \in \cP(M,N)$. For all $(\lambda,t) \in \R^*_+ \times \R$, we have $$\beta^{\varphi}_\lambda( (\varphi \circ T)^{\ri t})=\lambda^{\ri t}(\varphi \circ T)^{\ri t}.$$
Then Theorem \ref{dual action crossed product} shows that $\rB(N \subset c(M),\varphi)$ is generated by $\rB(N \subset c(M),\varphi)^{\beta^\varphi}= N' \cap c(M)$ and $\{ \varphi \circ T \}''$. We conclude that $\rB(N \subset c(M),\varphi)=\rb(N \subset c(M),\varphi)$ by definition.
\end{proof}

\begin{prop}
Let $N \subset M$ be a an inclusion of von Neumann algebras with expectations. Suppose that there exists $\varphi \in \cP_s(N)$ such that $\varphi$ is almost periodic. Then $N \subset M$ satisfies the bicentralizer conjecture.
\end{prop}
\begin{proof}
We may assume that $N$ is of type $\III_1$. For $\lambda \in \R^*_+$ and $a \in N_{\lambda \varphi,\varphi}$, we have $ax=\beta^{\varphi}_\lambda(x)a$ for every $x \in \rB(N \subset c(M),\varphi)$. In particular, $a$ commutes with $\rB(N \subset c(M),\varphi)^{\beta^\varphi}$. Since $\varphi$ is almost periodic, $N$ is generated by such elements $a$. We conclude that  $\rB(N \subset c(M),\varphi)^{\beta^\varphi} \subset N' \cap c(M)$ and therefore that $N \subset M$ satisfies the bicentralizer conjecture thanks to Proposition \ref{conjecture iff fixed point}.
\end{proof}

\begin{lem} \label{fixed point contained in commutant}
Let $N \subset M$ be an inclusion of von Neumann algebras with expectations. Suppose that $N$ is of type $\III_1$. Let $P \subset N$ a von Neumann subalgebra with conditional expectation $E \in \cE(N,P)$ and suppose that the inclusion $P \subset M$ satisfies the bicentralizer conjecture.
\begin{enumerate}
\item For every strictly semifinite weight $\varphi \in \cP(N)$ such that $\varphi =\varphi \circ E$ we have
$$\rB(N \subset c(M),\varphi)^{\beta^{\varphi}} \subset P' \cap c(M).$$
\item If $\psi \in \cP(N)$ is another strictly semifinite weight such that $\psi=\psi \circ E$, then 
$$\forall x \in\rB(N \subset c(M),\varphi)^{\beta^{\varphi}}, \quad \beta^{\psi,\varphi}(x)=x.$$ In particular, we have
$$\rB(N \subset c(M),\varphi)^{\beta^{\varphi}}=\rB(N \subset c(M),\psi)^{\beta^{\psi}}.$$
\end{enumerate}

\end{lem}
\begin{proof}
(1) If $z$ is the projection in $\cZ(P)$, then
$$\rB(N \subset c(M),\varphi)^{\beta^\varphi} \subset z\rB(N \subset c(M),\varphi)^{\beta^\varphi} \oplus z^{\perp}\rB(N \subset c(M),\varphi)^{\beta^\varphi}$$
and
$$  z\rB(N \subset c(M),\varphi)^{\beta^\varphi}=\rB(zNz \subset c(zMz),\varphi_{z})^{\beta^{\varphi_{z}}}$$
Therefore, we may reduce the problem to direct summands of $P$ and assume that $P$ is either of type $\III_1$ or has no type $\III_1$ summand.

If $P$ is of type $\III_1$, we have nothing to do because $$\rB(N \subset c(M),\varphi)^{\beta^\varphi} \subset \rB(P \subset c(M), \varphi|_P)^{\beta^{\varphi|_P}}=P' \cap c(M)$$
by Proposition \ref{conjecture iff fixed point}.

Now, assume that $P$ has no type $\III_1$ summand. Up to taking an amplification, we may assume that $\varphi$ has infinite multiplicity. By Corollary \ref{no III1 semifinite reduction}, we can find $Q \subset P$ with a $\varphi$-preserving $E_Q \in \cE(P,Q)$ such that $Q$ is semifinite, $Q' \cap P \subset Q$ and $P$ is generated by $Q$ and a unitary $u \in P$ such that $uQu^*=Q$. Choose a trace $\tau \in \cP(Q)$ and consider $\phi=\tau \circ E_Q \circ E$. Since $\varphi|_Q$ commutes with $\tau$, then $\varphi = \varphi|_Q \circ E_Q \circ E$ commutes with $\phi$. Therefore $\rB(N \subset c(M),\varphi)^{\beta^{\varphi}}=\rB(N \subset c(M),\phi)^{\beta^{\phi}}$ by Lemma \ref{weight commute fixed point}. We also know that $u \phi u^*$ commutes with $\phi$ because they are tracial on $Q$ and preserved by $E_Q$. Therefore, by Lemma \ref{weight commute fixed point} again, we know that $\beta^{u\phi u^*,\phi}(x)=x$ for all $x \in\rB(N \subset c(M),\phi)^{\beta^\phi}$. But $\beta^{u\phi u^*,\phi}(x)=uxu^*$ by definition. Thus $u$ commutes with $\rB(N \subset c(M),\phi)^{\beta^\phi}$. Since $\rB(N \subset c(M),\phi)^{\beta^\phi}$ also commutes with $Q \subset N_\phi$, we conclude that $\rB(N \subset c(M),\phi)^{\beta^\phi}$ commutes with $P=Q \vee \{u\}''$. This concludes the proof of item (1).

(2) Consider $\tilde{P}=P \otimes M_2(\C)$, $\tilde{N}=N \otimes M_2(\C)$ and $\tilde{M}=M \otimes M_2(\C)$. Note that $\tilde{P} \subset \tilde{M}$ stills satisfies the bicentralizer conjecture. Let $\tilde{E}=E \otimes \id \in \cE( \tilde{N}, \tilde{P})$ the conditional expectation induced by $E$. On $\tilde{N}$, we consider the diagonal weight $\varphi \oplus \psi$ which is preserved by $\tilde{E}$. Let $u=\begin{pmatrix} 
	0 & 1  \\
	1 & 0 \\
	\end{pmatrix}$. Then we have $u(\varphi \oplus \psi)u^*=(\psi \oplus \varphi)$. Therefore, we have
	$$ \forall x \in \rB(\tilde{N} \subset c(\tilde{M}), \varphi \otimes \psi), \quad  \beta^{(\psi \oplus \varphi),(\varphi \oplus \psi)}(x)=uxu^*.$$
	Now, since $u \in \tilde{P}$ and 
$$ \rB(\tilde{N} \subset c(\tilde{M}), \varphi \otimes \psi)^{\beta^{\varphi \otimes \psi}} \subset \tilde{P}' \cap c(\tilde{M})$$
by the first part of the theorem, we get
$$ \forall x \in \rB(\tilde{N} \subset c(\tilde{M}), \varphi \otimes \psi)^{\beta^{\varphi \otimes \psi}}, \quad  \beta^{(\psi \oplus \varphi),(\varphi \oplus \psi)}(x)=x.$$
Finally, cutting by the projection $e=\begin{pmatrix} 
	1 & 0  \\
	0 & 0 \\
	\end{pmatrix}$ and using item (3) of Theorem \ref{bicentralizer equivariance and flow}, we conclude that $\beta^{\psi,\varphi}(x)=x$ for all $x \in\rB(N \subset c(M),\varphi)^{\beta^{\varphi}}$.
\end{proof}

\begin{thm} \label{stability under generation}
Let $N \subset M$ be an inclusion of von Neumann algebras with expectations. Let $\varphi \in \cP(N)$ be strictly semifinite. Suppose that we have a family $(N_i)_{i \in I}$ of von Neumann subalgebras of $N$ such that :
\begin{itemize}
\item For every $i \in I$, there exists a $\varphi$-preserving conditional expectation $E_i \in \cE(N,N_i)$.
\item $N$ is generated by the subalgebras $(N_i)_{i \in I}$.
\item The inclusion $N_i \subset M$ satisfies the bicentralizer conjecture for every $i \in I$.
\end{itemize}
Then the inclusion $N \subset M$ satisfies the bicentralizer conjecture.
\end{thm}
\begin{proof}
We may assume that $N$ is of type $\III_1$. By Lemma \ref{fixed point contained in commutant}, we have
$$\rB(N \subset c(M), \varphi)^{\beta^\varphi} \subset \bigcap_{i \in I} N_i' \cap c(M)=N' \cap c(M).$$
Therefore, by Proposition \ref{conjecture iff fixed point} again, the inclusion $N \subset M$ satisfies the bicentralizer conjecture.
\end{proof}

The following is the main theorem of \cite{Ha85}. We briefly explain how to extend it to the non-factorial non-separable case.
\begin{thm} \label{haagerup dixmier}
Let $N$ be a type $\III_1$ von Neumann algebra. Then $N$ satisfies the bicentralizer conjecture if and only if $N \subset c(N)$ has the weak Dximier property.
\end{thm}
\begin{proof}
The only if direction follows from Proposition \ref{conditional expectation is in ucp}. Conversely, suppose that $N \subset c(N)$ has the weak Dixmier property and let us show that $N$ satisfies the bicentralizer conjecture. 

If $N$ has separable predual, we use successively Proposition \ref{desintegration dixmier}, Proposition \ref{dixmier for dominant and core} then \cite[Theorem 3.1]{Ha85} and finally Proposition \ref{desintegration bicentralizer}.

Now assume that $N$ is only countably decomposable. Let $\varphi \in N_*$ be a faithful state. Take a countably generated $N_0$ a $\sigma^\varphi$-invariant subalgebra. Using the separability of $(N_0)_*$ together with item (3) of Proposition \ref{equivalences dixmier}, we can find a countably generated globally $\sigma^\varphi$-invariant subalgebra $N_1$ that contains $N_0$ and such that for every $\xi \in c(N_0)_*$ with $\xi|_{N_0}=0$, the norm-closed convex hull of $\{ \xi \circ E_{ c(N_0)} \circ \Ad(u) \mid u \in \cU(N_1) \}$ contains $0$. We can repeat the same procedure to construct inductively an increasing sequence of $\sigma^\varphi$-invariant countably generated subalgebras $(N_k)_{k \in \N}$ such that for every $\xi \in c(N_k)_*$ with $\xi|_{N_k}=0$, the norm-closed convex hull of $\{ \xi \circ E_{ c(N_k)} \circ \Ad(u) \mid u \in \cU(N_{k+1}) \}$ contains $0$. Let $N_\infty=\bigvee_{k \in \N} N_k$ which has separable predual. One can check that the norm-closed convex hull of $\{ \xi  \circ \Ad(u) \mid u \in \cU(N_{\infty}) \}$ will contain $0$ for every $\xi \in c(N_\infty)_*$ such that $\xi|_{N_\infty}=0$. Using the Hahn-Banach theorem, this means that we can find an element of the Dixmier semigroup $\cD(N_\infty \subset c(N_\infty))$ whose range is contained in $N_\infty$. Since $N_\infty$ itself has the weak Dixmier property, we conclude that $N_\infty \subset c(N_\infty)$ has the weak Dixmier property, hence satisfies the bicentralizer conjecture. Since $N_0$ can be taken arbitrarily large, we can easily conclude that $N$ satisfies the bicentralizer conjecture.

The non-countably decomposable case reduces trivially to the countably decomposable case.
\end{proof}

The following theorem is a generalization of \cite[Theorem 3.1]{Ha85} (and Theorem \ref{haagerup dixmier}) but with a completely different proof relying on Theorem \ref{invariance bicentralizer ucp}. However, we do not obtain a new proof of Theorem \ref{haagerup dixmier} because we actually use Theorem \ref{haagerup dixmier} in the proof of Theorem \ref{conjecture iff dixmier}.
\begin{thm} \label{conjecture iff dixmier}
Let $N \subset M$ be an inclusion of von Neumann algebras with expectations. Then $N \subset M$ satisfies the bicentralizer conjecture if and only if the inclusion $N \subset c(M)$ has the weak Dixmier property.
\end{thm}
\begin{proof}
If $N$ has no type $\III_1$ summand, we can apply Proposition \ref{dixmier for no type III1} and Theorem \ref{conjecture for no type III1}.

We may now assume that $N$ is of type $\III_1$. Suppose that $N \subset M$ satisfies the bicentralizer conjecture. Take $\varphi \in \cP_s(N)$ with infinite multiplicity. By Proposition \ref{conjecture iff fixed point}, we have $\rB(N \subset c(M),\varphi)^{\beta^\varphi}=N' \cap c(M)$. By Proposition \ref{conditional expectation is in ucp}, we know that $\cD(N \subset c(M))$ contains a conditional expectation $E$ onto $\rB(N \subset c(M),\varphi)$ and by Proposition \ref{transition map is in ucp}, we know that $\beta^\varphi_\lambda \circ E \in \cD(N \subset c(M))$ for all $\lambda \in \R^*_+$. Since $\R^*_+$ is amenable, we can average the maps $(\beta^\varphi_\lambda \circ E)_{\lambda \in \R^*_+}$ to obtain a conditional expectation onto $\rB(N \subset c(M),\varphi)^{\beta^\varphi}=N' \cap c(M)$. This means that $N \subset c(M)$ has the weak Dixmier property.

Suppose that $N \subset c(M)$ has the weak Dixmier property. By Theorem \ref{haagerup dixmier}, $N$ satisfies the bicentralizer conjecture. Take $\psi$ a dominant weight on $N$. Then $N_\psi \subset M$ also has the weak Dixmier property. Take $g \in \cV^\psi(N \subset M)$ a conditional expectation onto $N_\psi' \cap M$. Take $f \in \cV^{\varphi,\psi}(N \subset M)$ and $h \in \cV^{\psi,\varphi}(N \subset M)$. Then $q=f \circ g \circ h \in \cV^\varphi(N \subset M)$. Moreover, $q|_{\rB(N,\varphi)}$ is normal, because $\rB(N,\varphi)=\cZ(N)$. Therefore, by Theorem \ref{invariance bicentralizer ucp}, we have $q(x) =x$ for all $x \in\rB(N \subset M, \varphi)$. Since $g(M)=N_\psi' \cap M$ and $f$ sends $\rb(N\subset M, \psi)=N_\psi' \cap M$ onto $\rb(N \subset M,\varphi)$, we conclude that $\rB(N \subset M, \varphi) \subset\rb(N \subset M, \varphi)$. 
\end{proof}

\begin{lem} \label{ucp for ergodicity}
Let $N \subset M$ be an inclusion of von Neumann algebras with expectations. Suppose that $N$ is properly infinite and let $\psi$ be a dominant weight on $N$. Take $u \in \cU(N)$ such that $u\psi u^*=\mu \psi$ for some $\mu \in \R^*_+ \setminus \{ 1\}$. Then there exists $f \in \cD(N_\psi \subset M)$ such that $f(x) \in N_\psi' \cap M$ for every $x \in \{u\}' \cap M$.
\end{lem}
\begin{proof}
Take $T \in \cP(M,N)$ and use it to view $c(N)$ as a subalgebra with expectations of $c(M)$. Inside $c(M) \rtimes_\theta \R^*_+$, let $(u_\lambda)_{\lambda \in \R^*_+}$ be the unitaries implementing the action $\theta$. Then we have to show that there exists $f \in \cD( c(N) \subset c(M) \rtimes_\theta \R^*_+)$ such that $f(x) \in c(N)' \cap (c(M) \rtimes_\theta \R^*_+)$ for every $x \in \{u_\mu\}' \cap (c(M) \rtimes_\theta \R^*_+)=c(M)^{\theta_\mu} \rtimes_\theta \R^*_+$.

Observe that $N \subset c(M)^{\theta_\mu}$ is with expectation. Indeed, $N \subset M$ is with expectation by assumption and $M \subset c(M)^{\theta_\mu}$ is also with expectation because $M$ the fixed point of $c(M)^{\theta_\mu}$ by an action of the compact group $\R^*_+/\mu^\Z$. Therefore, by Theorem \ref{expectation implies dixmier}, we can find $f_0 \in \cD(N \subset c(M)^{\theta_\mu})$ that is a conditional expectation onto $N' \cap c(M)^{\theta_\mu}$. Extend $f_1=f_0 \rtimes_\theta \id \in \cD(N \subset  c(M)^{\theta_\mu} \rtimes_\theta \R^*_+)$. Then $f_1$ is a conditional expectation from $c(M)^{\theta_\mu} \rtimes_\theta \R^*_+$ onto $(N' \cap c(M)^{\theta_\mu}) \rtimes_\theta \R^*_+=N' \cap (c(M)^{\theta_\mu} \rtimes_\theta \R^*_+)$. 

Since $\R$ is amenable, we can average by the unitaries $(\psi^{\ri t})_{t \in \R}$ in $\cU(c(N))$ to obtain a conditional expectation $g$ from $c(M)^{\theta_\mu} \rtimes_\theta \R^*_+$  onto $\{ \psi^{\ri t} \}'  \cap c(M)^{\theta_\mu} \rtimes_\theta \R^*_+$. This conditional expectation will send $N' \cap (c(M)^{\theta_\mu} \rtimes_\theta \R^*_+)$ onto $c(N)' \cap (c(M)^{\theta_\mu} \rtimes_\theta \R^*_+)$. Thus by letting $f_2=g \circ f_1$, we obtain $f_2 \in \cD( c(N) \subset c(M)^{\theta_\mu} \rtimes_\theta \R^*_+)$ that is conditional expectation onto $c(N)' \cap (c(M)^{\theta_\mu} \rtimes_\theta \R^*_+)$. Finally, we can extend $f_2$ to some $f \in \cD( c(N) \subset c(M) \rtimes_\theta \R^*_+)$ as we wanted.
\end{proof}

\begin{thm} \label{conjecture almost periodic part}
Let $N \subset M$ be an inclusion of von Neumann algebras with expectations. Let $\varphi \in \cP_s(N)$ and $T \in \cP(M,N)$. Suppose that $N$ is of type $\III_1$ and that it satisfies the bicentralizer conjecture. Then the following holds :
\begin{enumerate}
\item $\rB(N \subset M, \varphi)^{\beta_\lambda^\varphi}=\rb(N \subset M, \varphi)^{\beta_\lambda^\varphi}$ for every $\lambda \in \R^*_+ \setminus \{1 \}$.
\item $\rB(N \subset M, \varphi)^{\sigma_t^{\varphi \circ T}}=\rb(N \subset M, \varphi)^{\sigma_t^{\varphi \circ T}}$ for every $t \in \R \setminus \{0 \}$.
\end{enumerate}
\end{thm}
\begin{proof}
Without loss of generality, we may assume that $\varphi$ is a state (use Proposition \ref{anal bic corner reduction} and Theorem \ref{bicentralizer equivariance and flow}). We may also assume that $T=E$ is a conditional expectation.

(1) Take $\psi$ a dominant weight on $N$ and take $u \in \cU(N)$ such that $u\psi u^*=\lambda \psi$. Thanks to Lemma \ref{ucp for ergodicity}, we can find $g \in \cV(N_\psi \subset M) \subset \cV^\psi(N \subset M)$ such that $g|_{\{u\}' \cap M }$ is a conditional expectation onto $N_\psi' \cap \{u\}' \cap M $. Take $f \in \cV^{\varphi,\psi}(N \subset M)$ and $h \in \cV^{\psi,\varphi}(N \subset M)$. Observe that $h_n= \Ad(u^*)^n \circ h \circ \beta_{\lambda^n}^\varphi \circ E_{\rB(N \subset M,\varphi)}$ is still in $\cV^{\psi,\varphi}(N \subset M)$ for every $n \in \Z$ thanks to Proposition \ref{transition map is in ucp} and Proposition \ref{conditional expectation is in ucp}. Hence up to replacing $h$ by an accumulation point of $\frac{1}{n} \sum_{k=1}^n h_k$ when $n \to \infty$, we may assume that $\Ad(u) \circ h = h \circ  \beta_{\lambda}^\varphi \circ E_{\rB(N \subset M,\varphi)}$.

Now, put $q=f \circ g \circ h \in \cV^\varphi(N \subset M)$. Take $x \in\rB(N \subset M, \varphi)^{\beta_\lambda^\varphi}$.  By Theorem \ref{invariance bicentralizer ucp}, we have $q(x) =x$. Since $(h \circ \beta_\lambda^\varphi)(x)=(\Ad(u) \circ h)(x)$, we get $uh(x)u^*=h(x)$, i.e.\ $h(x) \in \{u\}' \cap M$. Therefore, $g(h(x)) \in N_\psi' \cap M$. Since $f$ takes $\rb(N \subset M,\psi)=N_\psi' \cap M$ onto $\rb(N \subset M, \varphi)$, we get $x=q(x)=f(g(h(x))) \in\rb(N \subset M, \varphi)$. We conclude that $\rB(N \subset M, \varphi)^{\beta_\lambda^\varphi} \subset\rb(N \subset M, \varphi)$ as we wanted.

(2)  Take $\psi$ a dominant weight on $N$. Use the conditional expectation $E$ to extend $\varphi$ and $\psi$ to $M$. Observe that the inclusions $N_\psi \subset M_\psi$ and  $M_\psi \subset M^{\sigma_t^\psi}$ are with expectation. Therefore, $N_\psi \subset M^{\sigma_t^\psi}$ is with expectation, hence it has the weak Dixmier property. Therefore, there exists $g \in \cD(N_\psi \subset M) \subset \cV^\psi(N \subset M)$ such that $g|_{M^{\sigma_t^\psi}}$ is a conditional expectation onto $N_\psi' \cap M^{\sigma_t^\psi}$. Take $f \in \cV^{\varphi,\psi}(N \subset M)$ and $h \in \cV^{\psi,\varphi}(N \subset M)$. Then $q=f \circ g \circ h \in \cV^\varphi(N \subset M)$. Therefore, $q(x) =x$ for all $x \in\rB(N \subset M, \varphi)$.

Take $x \in\rB(N \subset M, \varphi)^{\sigma_t^\varphi}$. Since $h \circ \sigma_t^\varphi=\sigma_t^\psi \circ h$, we have $h(x) \in M^{\sigma_t^\psi}$. Therefore, $g(h(x)) \in N_\psi' \cap M^{\sigma_t^\psi} \subset N_\psi' \cap M$. Since $f$ takes $\rb(N \subset M,\psi)=N_\psi' \cap M$ onto $\rb(N \subset M, \varphi)$, we get $x=q(x)=f(g(h(x))) \in\rb(N \subset M, \varphi)$. We conclude that $\rB(N \subset M, \varphi)^{\sigma_t^\varphi} \subset\rb(N \subset M, \varphi)$ as we wanted.
\end{proof}

The following corollary was conjectured in \cite[Open question]{AHHM18}.
\begin{cor} \label{bicentralizer flow ergodic}
Let $N \subset M$ be an inclusion of von Neumann algebras with expectations. Suppose that $N$ is of type $\III_1$ and satisfies the bicentralizer conjecture. Then for every strictly semifinite weight $\varphi \in \cP(N)$, we have $\rB(N \subset M,\varphi)^{\beta^\varphi}=N' \cap M$.
\end{cor}

The following corollary was announced in \cite[Remark 4.4]{MV23}.
\begin{cor} \label{bicentralizer weakly mixing}
Let $N \subset M$ be an irreducible inclusion of type $\III_1$ factors with expectation $E \in \cE(M,N)$. Suppose that $N$ satisfies the bicentralizer conjecture. Then the following are equivalent:
\begin{enumerate}
\item $c(N)' \cap c(M)=\C$
\item  The modular flow $\sigma^{\varphi \circ E}$ is ergodic on $\rB(N \subset M,\varphi)$ for some faithful state $\varphi \in N_*$.
\item There exists a faithful state $\varphi  \in N_*$ such that $\sigma^{\varphi \circ E}$ is ergodic on $M$.
\end{enumerate}
\end{cor}

\begin{lem} \label{flow of weights pmp}
Let $N_1$ and $N_2$ be two von Neumann algebras. Suppose that the flow of weights $\R^*_+ \curvearrowright \cZ(c(N_1))$ is probability measure preserving. Then the inclusion $N_2 \subset N_1' \cap c(N_1 \ovt N_2)$ is with expectation.
\end{lem}
\begin{proof}
Identify $c(N_1 \ovt N_2)$ with 
$$ \{ x \in  c(N_1) \ovt c(N_2) \mid (\theta_1 \otimes \id)(x)=(\id \otimes \theta_2)(x)\} $$ where $\theta_i : \R^*_+ \curvearrowright c(N_i)$ are the trace scaling action. Let $\mu$ be a $\theta_1$-invariant faithful normal state on $\cZ(c(N_1))$ and $E=\mu \otimes \id$ the associated faithful normal conditional expectation from $N_1' \cap (c(N_1) \ovt c(N_2))=\cZ(c(N_1)) \ovt c(N_2)$ onto $c(N_2)$. Since $\mu$ is $\theta_1$-invariant, then for every $x \in N_1' \cap c(N_1 \ovt N_2)$, we have
$$ \theta_2(E(x)) = E((\id \otimes \theta_2)(x))=E((\theta_1 \otimes \id)(x))=E(x).$$
This shows that $E(x) \in N_2$ for every $x \in N_1' \cap c(N_1 \ovt N_2)$, i.e.\ $E$ is a faithful normal conditional expectation from $N_1' \cap c(N_1 \ovt N_2)$ onto $N_2$.
\end{proof}

\begin{thm}
Let $N_1$ and $N_2$ be two von Neumann algebras. Suppose that the flow of weights $\R^*_+ \curvearrowright \cZ(c(N_i))$ are probability measure preserving for each $i \in \{1,2\}$. Then every inclusion with expectation $N_1 \ovt N_2 \subset M$ satisfies the bicentralizer conjecture.
\end{thm}
\begin{proof}
We may assume that $M$ is countably decomposable. We have to show that $N_1 \ovt N_2 \subset c(M)$ has the weak Dixmier property.

Suppose first that $N_1$ has no type $\III_1$ summand. Then the inclusion $N_1 \subset c(M)$ has the weak Dixmier property. Thus we only have to show that $N_2 \subset N_1' \cap c(M)$ has the weak Dixmier property. Take $E \in \cE(M,N_1 \ovt N_2)$. By Lemma \ref{flow of weights pmp}, the inclusion $N_2 \subset N_1' \cap c(N_1 \ovt N_2)$ is with expectation. Since $c_E(N_1 \ovt N_2) \subset c(M)$ is with expectation, we also have that $N_1' \cap c_E(N_1 \ovt N_2) \subset N_1' \cap c(M)$ is wtih expectation. Therefore the inclusion $N_2 \subset N_1' \cap c(M)$ is with expectation and we conclude that it has the weak Dixmier property as we wanted.

The case where $N_2$ has no type $\III_1$ summand is similar. Therefore, the only case left is when both $N_1$ and $N_2$ are of type $\III_1$. Take $\varphi$ a faithful normal state on $N_1$. and let $E_2 = (\varphi \otimes \id) \circ E \in \cE(M,N_2)$. Then $c(E_2) \in \cE(c(M), c_{E_2}(N_2))$ restricts to a conditional expectation $F$ from $\rB(N_1  \subset c(M),\varphi)$ onto $c_{E_2}(N_2)$. Moreover, $F$ is equivariant with respect to to the bicentralizer flow $\beta^\varphi : \R^*_+ \curvearrowright\rB(N_1  \subset c(M),\varphi)$ whose restriction to $c_{E_2}(N_2)$ is the canonical scaling flow $\theta : \R^*_+ \curvearrowright c_{E_2}(N_2)$. It follows that $F$ restricts to a conditional expectation from $\rB(N_1 \subset c(M),\varphi)^{\beta^\varphi}$ onto $c_E(N_2)^\theta=N_2$. We proved that $N_2 \subset \rB(N_1 \subset c(M),\varphi)^{\beta^\varphi}$ is with expectation. But this implies that $N_2 \subset \rB(N_1 \subset c(M),\varphi)^{\beta^\psi}$ is also with expectation for every $\psi \in \cP_s(N_1)$, because the isomorphism $\beta^{\psi,\varphi}$ from $\rB(N_1 \subset c(M),\varphi)^{\beta^\varphi}$ onto $\rB(N_1 \subset c(M),\varphi)^{\beta^\psi}$ fixes $N_2$. By Theorem \ref{expectation implies dixmier}, we conclude that the inclusion $N_2 \subset\rB(N_1 \subset c(M),\varphi)^{\beta^\psi}$ has the weak Dixmier property. 

Take $\psi$ with infinite multiplicity. Then for every $x  \in c(M)$, the set $\conv \{ uxu^* \mid u \in \cU(N_1) \}$ intersects $\rB(N_1 \subset c(M),\varphi)^{\beta^\psi}$. Since $N_2 \subset \rB(N_1 \subset c(M),\varphi)^{\beta^\psi}$ has the weak Dixmier property  we get that the set $\conv \{ uxu^* \mid u \in \cU(N_1 \ovt N_2) \}$ intersects $N_2' \cap c(M)$. Now, by using the same argument as in the first case, we have that the inclusion $N_1 \subset N_2' \cap c(M)$ is wtih expectation, hence it has the the weak Dixmier property. Thus $\conv \{ uxu^* \mid u \in \cU(N_1 \ovt N_2) \}$ intersects $(N_1 \ovt N_2)' \cap c(M)$. We conclude that the inclusion $N_1 \ovt N_2 \subset c(M)$ has the weak Dixmier property as we wanted.
\end{proof}

\begin{thm} \label{finite bicentralizer}
Let $N \subset M$ be an inclusion of von Neumann algebras with expectations. Let $\varphi \in \cP(N)$ be some strictly semifinite weight. If $\rB(N \subset M, \varphi)$ is finite then $\rB(N \subset M, \varphi)=\rb(N \subset M, \varphi)$.
\end{thm}
\begin{proof}
We may assume that $N$ of type $\III_1$, that $\varphi$ is a state and that $\cE(M,N) \neq \emptyset$. Note that $N$ satisfies the bicentralizer conjecture because $\rB(N,\varphi)$ is finite.

 Take $E \in \cE(M,N)$. Let $Z$ be the center of $\rB(N \subset M, \varphi)$. Let $\tau$ be the unique faithful normal tracial state on $\rB(N \subset M, \varphi)$ such that $\tau|_Z = (\varphi \circ E)|_Z$. Since $(\varphi \circ E)|_Z$ is invariant under $\beta^\varphi$, so is $\tau$. Define a new state on $M$ by $\psi = \tau \circ E_{\rB(N \subset M, \varphi)}$. We have $\psi|_N=\varphi$. We claim that $\psi = \varphi \circ F$ for some $F \in \cE(M,N)$. Indeed, let $h$ be the Radon-Nikodym derivative of $(\varphi \circ E)|_{\rB(N \subset M, \varphi)}$ with respect to $\tau$. Since both of this states are invariant under $\beta^\varphi$, we know that $h$ is fixed by $\beta^\varphi$. Therefore $h$ is affiliated with $N' \cap M$ by Theorem \ref{conjecture almost periodic part}. This implies that $\sigma_t^\psi(x)=h^{\ri t} \sigma_t^{\varphi \circ E}(x) h^{- \ri t}=\sigma_t^{\varphi}(x)$ for every $x \in N$, hence $\psi=\varphi \circ F$ for some $F \in \cE(M,N)$. Now, we conclude that $$\rB(N \subset M, \varphi)=\rB(N \subset M, \varphi)^{\sigma^{\varphi \circ F}}=\rb(N \subset M,\varphi)^{\sigma^{\varphi \circ F}}$$ by Theorem \ref{conjecture almost periodic part}.
\end{proof}

\section{Finite index and quasiregular inclusions} \label{section quasiregular}

\begin{thm} \label{conjecture normalizer}
Let $N \subset M$ be an inclusion of von Neumann algebras with expectations. Let $P \subset N$ be a von Neumann subalgebra with a conditional expectation $E \in \cE(N,P)$ such that $N$ is generated by the normalizer 
$$\cN(E)=\{ u \in \mathcal{U}(N) \mid  \Ad(u) \circ E=E \circ \Ad(u) \}.$$ If $P \subset M$ satisfies the bicentralizer conjecture, then $N \subset M$ satisfies the bicentralizer conjecture.
\end{thm}
\begin{proof}
We may assume that $N$ is of type $\III_1$. Take a strictly semifinite $\varphi \in \cP(N)$ such that $\varphi =\varphi \circ E$. Take $u \in \cN(E)$ and let $\psi=u \varphi u^*$. Then we still have $\psi=\psi \circ E$. Therefore $uxu^*=\beta^{\psi,\varphi}(x)=x$ for all $x \in\rB(N \subset c(M),\varphi)^{\beta^\varphi}$ by Lemma \ref{fixed point contained in commutant}. This shows that  $\rB(N \subset c(M),\varphi)^{\beta^\varphi}$ commutes with $\cN(E)$, hence also with $N=\cN(E)''$. We conclude by Proposition \ref{conjecture iff fixed point} that $N \subset M$ satisfies the bicentralizer conjecture.
\end{proof}

Recall that an inclusion of von Neumann algebras $N \subset M$ has \emph{finite index} if there exists a conditional expectation $E \in \cE(M,N)$ and some $\kappa > 0$ such that $E(x) \geq \kappa x$ for all $x \in M_+$. The following theorem strengthens \cite[Theorem 5]{HP17}. 

\begin{thm}\label{conjecture finite index}
Let $N \subset M$ be an inclusion of von Neumann algebras with finite index. Then the following are equivalent :
\begin{enumerate}[\rm (i)]
\item $N$ satisfies the bicentralizer conjecture.
\item $M$ satisfies the bicentralizer conjecture.
\item $N \subset M$ satisfies the bicentralizer conjecture.
\end{enumerate}
\end{thm}
\begin{proof}
(i) $\Rightarrow$ (ii) + (iii). Let $E \in \cE(M,N)$ be a conditional expectation with finite index. Take $\varphi \in \cP(N)$ a strictly semifinite weight. Then $E|_{\rB(N \subset M, \varphi)} \in \cE(\rB(N \subset M,\varphi),\rB(N,\varphi))$ still has finite index. By assumption, $\rB(N,\varphi)=\rb(N,\varphi)$ is abelian. Therefore, $\rB(N \subset M,\varphi)$ must be finite (and of type I). By Theorem \ref{finite bicentralizer}, it follows that $\rB(N \subset M, \varphi)=\rb(N \subset M, \varphi)$. Since $\rB(N \subset M,\varphi)$ is finite, $\rB(M,\varphi \circ E) \subset\rB(N \subset M,\varphi)$ is also finite hence $\rB(M,\varphi \circ E)=\rb(M, \varphi \circ E)$.

(ii) $\Rightarrow$ (i). Suppose that $M$ satisfies the bicentralizer conjecture. Then so does $M'$. Since $M' \subset N'$ also has finite index, the first part shows that $N'$ satisfies the bicentralizer conjecture, and we conclude the same thing for $N$.

(iii) $\Rightarrow$ (i) is trivial.
\end{proof}

We say that an inclusion of von Neumann algebras $N \subset M$ has \emph{locally finite index} if there exists an increasing net of projections $(p_i)_{i \in I}$ in $N' \cap M$ with $\bigvee_{i \in I} p_i=1$ such that $Np_i \subset p_iMp_i$ has finite index for every $i \in I$. Using Proposition \ref{projection reduction conjecture}, we get immediately the following corollary from Theorem \ref{conjecture finite index}

\begin{cor}\label{conjecture locally finite index}
Let $N \subset M$ be an inclusion of von Neumann algebras with locally finite index. Then the following are equivalent :
\begin{enumerate}[\rm (i)]
\item $N$ satisfies the bicentralizer conjecture.
\item $M$ satisfies the bicentralizer conjecture.
\item $N \subset M$ satisfies the bicentralizer conjecture.
\end{enumerate}
\end{cor}

We say that an inclusion of von Neumann algebras $N \subset M$ is \emph{quasiregular} if the $N$-bimodule ${_N}\rL^2(M)_N$ decomposes as a direct sum of $N$-bimodules with finite index, or equivalently, if the inclusion $N \subset \langle M,N \rangle$ has locally finite index where $\langle M,N \rangle$ is the so-called basic construction. Note that since the inclusion $N \subset \langle M,N \rangle$ has locally finite index, it is with expectations. Therefore, $N \subset M$ is also with expectations.

\begin{cor} \label{conjecture quasiregular}
Let $N \subset M$ be a quasiregular inclusion. If $N$ satisfies the bicentralizer conjecture, then the inclusion $N \subset M$ also satisfies the bicentralizer conjecture. 
\end{cor}
\begin{proof}
Let $Q=\langle M,N \rangle$. Then $N \subset Q$ satisfies the bicentralizer conjecture by Corollary \ref{conjecture locally finite index}. Therefore, for every strictly semifinite weight $\varphi$ on $N$, we have
$$\rB(N \subset M, \varphi)=\rB(N \subset Q, \varphi) \cap M=\rb(N \subset Q, \varphi) \cap M.$$
Since $\cP(Q,M) \neq \emptyset$, we conclude by Proposition \ref{alg bicentralizer for intermediate subalgebra}.(3) that
$$\rB(N \subset M, \varphi)=\rb(N \subset M, \varphi).$$
\end{proof}


\begin{thm} \label{stability under quasiregular extensions}
Let $N \subset M$ be an inclusion of von Neumann algebras with expectations. Let $P \subset N$ be a von Neumann algebra subalgebra with expectations such that $P' \cap N \subset P$ and $P \subset N$ is quasiregular. If $P \subset M$ satisfies the bicentralizer conjecture, then $N \subset M$ satisfies the bicentralizer conjecture.
\end{thm}
\begin{proof}
We may assume that $N$ is of type $\III_1$. Since $P' \cap N \subset P$, there is a uniqe $E_P \in \cE(N,P)$. 

Let $\cA=\rB(N \subset c(M),\varphi)^{\beta^\varphi}$ where $\varphi \in \cP(N)$ is any strictly semifinite weight such that $\varphi =\varphi \circ E_P$. By Lemma \ref{fixed point contained in commutant}, the algebra $\cA$ does not depend on the choice of $\varphi$. We have to show that $N$ commutes with $\cA$. 

Since $P \subset N$ is quasiregular and $P' \cap N=\cZ(P)$, we know that $N$ is generated by partial isometries $v \in N$ for which there exists a partial *-morphism $\theta : pPp \rightarrow qPq$ with $p,q \in P$ such that $\theta(x)v=vx$ for all $x \in pPp$ and $\theta(pPp) \subset qPq$ has finite index. We may assume that $p=v^*v$.

Let $Q$ be the von Neumann algebra generated by $P$ and $\theta(pPp)' \cap qNq$. Choose $\varphi \in \cP(N)$ strictly semifinite such that $\varphi=\varphi \circ E_P$, $q \in N_\varphi$ and $\theta(pPp) \subset qPq$ admits a $\varphi$-preserving normal conditional expectation. Then $Q$ admits a $\varphi$-preserving normal conditional expectation $E_Q \in \cE(N,Q)$, which is necessarily the unique element of $\cE(N,Q)$ since $Q' \cap N \subset P' \cap N \subset P \subset Q$.

Observe that $\cZ(P)q=(qPq)' \cap qNq \subset \theta(pPp)' \cap qNq$ has finite index. Therefore $\theta(pPp)' \cap qNq$ is finite. Since $P \subset M$ satisfies the bicentralizer conjecture and $Q$ is generated by $P$ and $P_\varphi \vee (\theta(pPp)' \cap qNq)$ (which is semifinite), we conclude by Theorem \ref{stability under generation}, that $Q \subset M$ satisfies the bicentralizer conjecture. 

Let $e=vv^*$ and $\psi_0=v\varphi v^* \in \cP(eNe)$.  Observe that $vPv^*$ is globally invariant under $\sigma^{\psi_0}$ and $\psi_0|_{vPv^*}$ is semifinite. Therefore, $\psi_0$ is preserved by the some conditional expectation $F \in \cE(eNe,vPv^*)$. Moreover, $(vPv^*)' \cap eNe=v(P' \cap N)v^* \subset vPv^*$. This implies that $F$ is the unique element of $\cE(eNe,vPv^*)$. Therefore, we necessarily have $F \circ E_Q|_{eNe}=F$ because $vPv^* \subset Q$. Now, take any $\psi_1 \in \cP(e^{\perp}Ne^{\perp})$ such that $\psi_1=\psi_1 \circ E_Q|_{e^\perp N e^\perp}$ and let $\psi=\psi_0 \oplus \psi_1 \in \cP(N)$. Then we have $\psi =\psi \circ E_Q$. By sontruction of $\psi$, we have $v \in N_{\psi,\varphi}$. Then, since $Q \subset M$ satisfies the bicentralizer conjecture, Theorem \ref{fixed point contained in commutant} tells us that $vx=\beta^{\psi,\varphi}(x)v=xv$ for every $x \in\rB(N \subset c(M), \varphi)^{\beta^\varphi}$. We proved that the intertwiner $v$ commutes with $\cA$. Since $N$ is generated by such intertwiners, we conclude that $N$ commutes with $\cA$, hence $N \subset M$ satisfies the bicentralizer conjecture by Proposition \ref{conjecture iff fixed point}.
\end{proof}

\begin{cor}
Let $N \subset M$ be an inclusion of von Neumann algebras with expectations. Let $P \subset N$ be a finite index subalgebra. If $P \subset M$ satisfies the bicentralizer conjecture, then $N \subset M$ satisfies the bicentralizer conjecture.
\end{cor}
\begin{proof}
Note that $P' \cap N$ is finite because $\cZ(P) \subset P' \cap N$ has finite index. Since $P$ satisfies the bicentralizer conjecture and $P' \cap N$ is finite, we know by Theorem \ref{stability under generation} that $Q  =P \vee (P' \cap N) \subset M$ also satisfies the bicentralizer conjecture. Moreover, $Q' \cap N \subset Q$ and $Q \subset N$ is of finite index, hence quasiregular. By Theorem \ref{stability under quasiregular extensions}, we conclude that $N \subset M$ satisfies the bicentralizer conjecture.
\end{proof}

\section{Local quantization and maximal abelian subalgebras} \label{section masa}

In \cite[A.1]{Po94} and \cite[A.1.2]{Po95}, building on the same technique of \cite{Po81}, Popa obtained the following \emph{local quantization principle}.

\begin{thm}[Popa]  \label{local quantization original}
Let $N \subset M$ be an inclusion of von Neumann algebras with expectation. Suppose that $N$ is of type $\II_1$ and let $\varphi \in M_*$ be a faithful state such that $N \subset M_\varphi$. Then for every finite set $F \subset M$ and every $\varepsilon > 0$, there exists a finite dimensional abelian subalgebra $A \subset N$ such that 
$$ \forall x \in F, \quad \| E_{A' \cap M}(x)-E_{N' \cap M}(x) \|_\varphi \leq \varepsilon$$
where we use the $\varphi$-preserving conditional expectations.
\end{thm}

In this section, we will consider a type $\III$ analog of Popa's local quantization principle by using the \emph{Effros-Mar\' echal topology} on the space $\mathrm{SA}(M)$ of von Neumann subalgebras of a given von Neumann algebra $M$. We refer the reader to \cite{HW98} for all details on this topology. Let us simply list the following facts :
\begin{enumerate}[(P1)]
\item The Effros-Mar\' echal topology on $\mathrm{SA}(M)$ is the weakest topology that makes the functions
$$ \mathrm{SA}(M) \ni A \mapsto \|\xi |_A\| \in \R_+$$
continuous for all $\xi \in M_*$. 
\item \label{convergence of expectations} When restricted to the subset $$\mathrm{SA}_\varphi(M)=\{ N \in \mathrm{SA}(M) \mid \forall t \in \R, \; \sigma_t^\varphi(N)=N \}$$
for some faithful state $\varphi \in M_*$, 
the Effros-Mar\'echal topology is equivalent to the topology of pointwise strong convergence of the $\varphi$-preserving conditional expectations.
\item \label{approx increase decrease} If $(A_i)_{i \in I}$ is an increasing (resp.\ decreasing) net of elements of $\mathrm{SA}(M)$ and $A = \bigvee_{i \in I} A_i$ (resp.\ $A=\bigcap_{i \in I} A_i$), then $\lim_{i \to \infty} A_i=A$.
\end{enumerate}

Inspired by Theorem \ref{local quantization original}, we make the following definition.
\begin{df}
Let $N \subset M$ be an inclusion of von Neumann algebras. We say that the inclusion $N \subset M$ has \emph{Popa's local quantization property} if $N' \cap M$ is in the closure of
$$ \{ A' \cap M \mid A \text{ is a finite dimensional abelian subalgebra of } N\}$$
in the Effros-Marechal topology of $\mathrm{SA}(M)$. We say that $N$ has the local quantization property if the inclusion $N \subset N$ has the local quantization property.
\end{df}

Concretely, the inclusion  $N \subset M$ has the local quantization property if and only if for every finite set $F \subset M_*$ and every $\varepsilon > 0$, there exists a finite dimensional abelian subalgebra $A \subset N$ such that 
$$ \forall \xi \in F, \quad \| \xi|_{A' \cap M} \| \leq \| \xi|_{N' \cap M} \| +\varepsilon.$$ 

The following theorem is a reformulation of Theorem \ref{local quantization original}. We simply explain how to extend it to the type $\II_\infty$ case.

\begin{thm}  \label{local quantization type II}
Let $N \subset M$ be an inclusion of von Neumann algebras with expectations. Suppose that $N$ is of type $\II$. Then the inclusion $N \subset M$ has the local quantization property.
\end{thm}
\begin{proof}
Suppose first that $N$ is of type $\II_1$ and that $N$ and $N' \cap M$ are countably decomposable. In this case, the local quantization property follows from Theorem \ref{local quantization original} and Property (P\ref{convergence of expectations}).

Now, suppose that $N' \cap M$ is countably decomposable but $N$ is arbitrary. Take $e \in N$ such that $eNe$ is fintie and countably decomposable. Then, by the first case, the closure of 
$$ \{ A' \cap M \mid A \text{ is a finite dimensional abelian subalgebra of } eNe \oplus \C(1-e) \}$$
contains $(N' \cap M)e \oplus M(1-e)$. When $e$ goes to $1$, the latter subalgebra clearly converges to $N' \cap M$. This shows that $N' \cap M$ is in the closure of
$$ \{ A' \cap M \mid A \text{ is a finite dimensional abelian subalgebra of } N\}$$
which means that $N \subset M$ has the local quantization property.

Finally, when $N' \cap M$ is not necessarily countably decomposable, take $F \subset M_*$ a finite subset. Then there exists a projection $e \in N' \cap M$ such that $e(N' \cap M)e$ is countably decomposable and the elements of $F$ are supported on $e$, i.e.\ $F \subset (eMe)_*$. By the previous case, for every $\varepsilon > 0$, we can find a finite dimensional abelian subalgebra $A \subset N$ such that
$$ \forall \xi \in F, \quad \| \xi|_{A' \cap N} \| =\| \xi|_{A' \cap eNe} \| \leq \| \xi|_{((eNe)' \cap eMe) }\| + \varepsilon \leq \| \xi|_{N' \cap M }\| + \varepsilon.$$
We conclude that $N \subset M$ has the local quantization property.
\end{proof}

\begin{prop} \label{local quantization restriction}
Let $N \subset M \subset P$ be inclusions of von Neumann algebras. Let $(A_i)_{i \in I}$ be a net of von Neumann subalgebras of $N$ such that $(A_i' \cap P)_{i \in I}$ converges to $N' \cap P$ in the Effros-Marechal topology. Then $(A_i' \cap M)_{i \in I}$ converges to $N' \cap M$.

In particular, if $N \subset P$ has the local quantization property, then $N \subset M$ has the local quantization property.
\end{prop}
\begin{proof}
Let $Q=\limsup_{i \in I} (A_i' \cap M)$ as in \cite{HW98}. Then $Q \subset M$ and $Q \subset \limsup_{i \in I} (A_i' \cap P)=N' \cap P$. Therefore $Q \subset N' \cap M$. Since we also have $N' \cap M \subset \liminf_{i \in I} A_i' \cap M$, we conclude that $\lim_{i \in I} A_i' \cap M=N' \cap M$.
\end{proof}

As in \cite{Po81}, the local quantization property can be used to solve Kadison's problem.

\begin{prop}
Let $N \subset M$ be an inclusion of von Neumann algebras with separable predual. Suppose that $pNp \subset pMp$ has the local quantization property for every projection $p \in N$. Then there exists a maximal abelian subalgebra $A \subset N$ such that $A' \cap M=A \vee (N' \cap M)$.
\end{prop}
\begin{proof}
Let $(\xi_n)_{n \in \N}$ be a dense sequence in $N_*$ and $(\zeta_n)_{n \in \N}$ a dense sequence in $M_*$. We will construct inductively an increasing sequence of finite dimensional abelian subalgebras $(A_n)_{n \in \N}$ of $N$ such that
\begin{enumerate}
\item $ \| \xi_n|_{A_n' \cap N} \| \leq \| \xi_n|_{A_{n-1} \vee \cZ(N) } \| + 2^{-n}$
\item $ \| \zeta_n|_{A_n' \cap M} \| \leq \| \zeta_n|_{A_{n-1} \vee (N' \cap M)} \| + 2^{-n}$.
\end{enumerate}
Suppose we have already constructed $A_{n-1}$. Let $\cN = A_{n-1}' \cap N$ and $\cM=A_{n-1}' \cap M$. By our assumptions, we know that $\cN \subset \cM$ has the local quantization property. Take $(B_k)_{k \in \N}$ a sequence of finite dimensional abelian subalgebras of $\cN$ containing $A_{n-1}$ such that $B_k' \cap \cM=B_k' \cap M$ converges to $\cN' \cap \cM=A_{n-1} \vee (N' \cap M)$. Then we also have that $B_k' \cap N$ converges to $\cZ(\cN)=A_{n-1} \vee \cZ(N)$ (because the map $Q \mapsto Q \cap N$ is upper semicontinuous on $\mathrm{SA}(M)$). We then let $A_{n}=B_k$ for $k$ large enough.

Now that $(A_n)_{n \in \N}$ is constructed, we let $A=\bigvee_{n \in \N} A_n$. Then for every $n \geq 1$, using property (1) above we have
$$ \| \xi_n|_{A' \cap N} \| \leq \| \xi_n|_{A_n' \cap N} \| \leq \| \xi_n|_{A_{n-1} \vee \cZ(N) } \| + 2^{-n} \leq \| \xi_n|_{A \vee \cZ(N)} \| +2^{-n}.$$
Since $(\xi_n)_{n \in \N}$ is dense in $N_*$, it follows that $\| \xi|_{A' \cap N} \| \leq \| \xi|_{A \vee \cZ(N)} \|$ for all $\xi \in N_*$. By the Hahn-Banach theorem, this implies that $A' \cap N \subset A \vee \cZ(N)$. Similarly, using property (2), we get $A' \cap M \subset A \vee (N' \cap M)$. Up to replacing $A$ by $A \vee \cZ(N)$, we thus obtain the maximal abelian subalgebra we wanted.
\end{proof}

Our goal is to generalize Theorem \ref{local quantization type II} to obtain the following theorems.

\begin{thm} \label{local quantization for one}
Let $N$ be a von Neumann algebra with no type $\I$ direct summand. Then $N$ has the local quantization property.
\end{thm}

\begin{thm} \label{local quantization no core type III}
Let $N \subset M$ be an inclusion of von Neumann algebras with expectations. Suppose that $N$ has no type $\I$ direct summand and $N$ satisfies the bicentralizer conjecture. Then the inclusion $N \subset M$ has the local quantization property.
\end{thm}
\begin{cor}
Let $N \subset M$ be an inclusion of von Neumann algebras with expectation and with separable preduals. Suppose that $N$ satisfies the bicentralizer conjecture. Then there exists a maximal abelian subalgebra $A \subset N$ such that $A' \cap M=A \vee (N' \cap M)$.
\end{cor}

\begin{thm} \label{local quantization type III}
Let $N \subset M$ be an inclusion of von Neumann algebras with expectations. Suppose that $N$ has no type $\I$ direct summand. Then $N \subset M$ satisfies the bicentralizer conjecture if and only if the inclusion $N \subset c(M)$ has the local quantization property.
\end{thm}
\begin{cor}
Let $N \subset M$ be an inclusion of von Neumann algebras with expectation and with separable preduals. Then $N \subset M$ satisfies the bicentralizer conjecture if and only if there exists a maximal abelian subalgebra $A \subset N$ such that $A' \cap c(M)=A \vee (N' \cap c(M))$.
\end{cor}

In order to prove Theorem \ref{local quantization for one} and Theorem \ref{local quantization no core type III}, we will actually need to prove Theorem \ref{local quantization type III} first and then use Proposition \ref{local quantization restriction}.

The following lemma is obtained by applying Theorem \ref{local quantization type II} to the asymptotic centralizer $N^\omega_{\varphi^\omega}$.
\begin{lem} \label{fd abelian into bicentralizer}
Let $N \subset M$ be an inclusion of von Neumann algebras with expectations. Suppose that $N$ has no type $\I$ summand. Let $\varphi$ be a faithful normal state on $N$. Then for every finite set $F \subset c(M)_*$ and every $\varepsilon > 0$, there exists a finite dimensional abelian subalgebra $A \subset N$ such that $ \| \varphi - \varphi \circ E_{A' \cap N} \| \leq \varepsilon$ and $$\forall \xi \in F, \quad \| \xi \circ E_{A' \cap c(M)} - \xi \circ E_{\rB(N \subset c(M), \varphi)} \| \leq \varepsilon $$
\end{lem}

With Lemma \ref{fd abelian into bicentralizer}, we already get very close to Theorem \ref{local quantization type III} because if $N \subset M$ satisfies the bicentralizer conjecture, then $\rB(N \subset c(M),\varphi)=\rb(N \subset c(M),\varphi)=N' \cap c(M) \vee \{ \varphi \}''$. Unfortunately, in order to get rid of the $\{ \varphi \}''$-part and close the gap between Lemma \ref{fd abelian into bicentralizer} and Theorem \ref{local quantization type III}, we only found a lengthy proof based on a reduction to the amenable case. So let us deal with the amenable case first. We will only use the amenability assumption in the following key observation. 
\begin{lem} \label{masa amenable}
Let $N$ be an amenable von Neumann algebra with separable predual. Then there exists a maximal abelian subalgebra $A \subset N$ such that 
$$A' \cap c(N)= A \vee \cZ(c(N)).$$
\end{lem}
\begin{proof}
Since $N$ is amenable, we can write $N=B \rtimes_{\sigma} \Z$ where $B$ is some cartan subalgebra of $N$ and $\sigma : \Z \curvearrowright B$ is a nonsingular action. Take $A=B^{\sigma} \rtimes_{\sigma} \Z \cong B^{\sigma} \ovt L(\Z)$. Then $A$ is maximal abelian in $N$. Moreover, we have $c(N)= c(B) \rtimes_{c(\sigma)} \Z$ where $c(\sigma) : \Z \curvearrowright c(B)$ is the Maharam extension. Thus 
$$L(\Z)' \cap c(N)= c(B)^{c(\sigma)} \rtimes_{c(\sigma)} \Z = \cZ(c(N)) \vee L(\Z).$$
It follows that $A' \cap c(N)=L(\Z)' \cap c(N)= A \vee \cZ(c(N))$.
\end{proof}

\begin{lem} \label{universal masa}
Let $N$ be an amenable von Neumann algebra with separable predual. Then there exists a maximal abelian subalgebra $A \subset N$ such that for every inclusion with expectations $N \subset M$ we have $$A' \cap \rb^{\sharp}(N \subset c(M))=A \vee (N' \cap c(M)).$$
\end{lem}
\begin{proof}
If $N$ is finite, then for any inclusion with expectation $N \subset M$, we have $$\rb^{\sharp}(N \subset c(M)) = N \vee (N' \cap c(M)) \cong N \ovt_{\cZ(N)} (N' \cap c(M)).$$
Thus, in this case, any maximal abelian subalgebra $A \subset N$ will work.

Now, suppose that $N$ is properly infinite. Let $\psi \in \cP(N)$ be a dominant weight and $u=(u_\lambda)_{\lambda \in \R^*_+}$ a $\psi$-scaling group. Let $\kappa=\Ad(u) : \R^*_+ \curvearrowright N_\psi$ be the action implemented by this group. Note that the inclusion $N^\kappa_\psi \subset N_\psi$ is isomorphic to the inclusion $N \subset c(N)$. By Lemma \ref{masa amenable}, we can find maximal abelian subalgebra $A_0 \subset N^\kappa_\psi$ that satisfies $A_0' \cap N_\psi =A_0 \vee \cZ(N_\psi)$. Let $B$ be the von Neumann subalgebra of $N$ generated by the unitaries $(u_\lambda)_{\lambda \in \R^*_+}$ and let $A=A_0 \vee B \cong A_0 \ovt B$. Then $A' \cap N=(A_0' \cap N_\psi)^\kappa \vee B=A_0 \vee B=A$. Thus $A$ is maximal abelian in $N$.

Now, take $N \subset M$ an inclusion with expectations. We have a decomposition
$$\rb^{\sharp}(N \subset c(M)) =N_\psi \vee (N_\psi' \cap c(M)) \vee B \cong \left( N_\psi \ovt_{\cZ(N_\psi)} (N_\psi' \cap c(M)) \right) \rtimes_{\kappa} \R^*_+.$$
where $\kappa$ still denotes the action $\Ad(u)$ implemented by the group $u=(u_\lambda)_{\lambda \in \R^*_+}$. Therefore, we get
\begin{align*}
A_0' \cap \rb^{\sharp}(N \subset c(M)) &= (A_0' \cap N_\psi) \vee (N_\psi' \cap c(M)) \vee B \\
&= A_0 \vee (N_\psi' \cap c(M))  \vee B\\
&\cong   (A_0 \vee (N_\psi' \cap c(M)))  \rtimes_\kappa \R^*_+.
\end{align*}
Therefore, we have
$$ B' \cap A_0' \cap \rb^{\sharp}(N \subset c(M))  =  (A_0 \vee (N_\psi' \cap c(M)))^\kappa \vee B.$$
Now, observe that $(A_0 \vee (N_\psi' \cap c(M)) = A_0 \vee (N' \cap c(M)) \vee  \{ \psi \circ T \}'' $ where $T \in \cP(M,N)$ and on this algebra, $\kappa$ is a dual action that fixes $A_0 \vee (N' \cap c(M))$ and scales $\psi \circ T$. It follows that 
$$ (A_0 \vee (N_\psi' \cap c(M)))^\kappa = A_0 \vee (N' \cap c(M)).$$
We conclude that 
$$ B' \cap A_0' \cap \rb^{\sharp}(N \subset c(M)) = A_0 \vee (N' \cap c(M)) \vee B.$$  
hence $$A' \cap   \rb^{\sharp}(N \subset c(M)) = A \vee (N' \cap c(M))$$
as we wanted.
\end{proof}


\begin{lem} \label{semilocal quantization}
Let $N \subset M$ be an inclusion of von Neumann algebras with expectations. Suppose that $N$ is amenable and has separable predual. Then for every $\xi \in c(M)_*$ and every $\varepsilon > 0$, there exists a finite dimensional abelian subalgebra $A \subset N$ such that 
$$\| \xi|_{A' \cap c(M)} \| \leq \| \xi |_{N \vee (N' \cap c(M)) } \| + \varepsilon.$$
\end{lem}
\begin{proof}
Let $\eta=\xi \circ E_{\rb^{\sharp}(N \subset c(M))}$.
By Lemma \ref{universal masa}, we can find a maximal abelian subalgebra $B \subset N$ such that 
$$B' \cap \rb^{\sharp}( N \subset c(M))=B \vee (N' \cap c(M)).$$
Then, thanks to property (P\ref{approx increase decrease}) of the Effros-Mar\'echal topology, we can find a finite dimensional abelian subalgebra $A_0 \subset B$ that is so large that
$$\| \eta|_{A_0' \cap c(M)} \| = \| \eta|_{A_0' \cap \rb^\sharp (N \subset c(M))} \| \leq \| \eta|_{N \vee (N' \cap c(M))} \| + \varepsilon.$$
By Lemma \ref{fd abelian into bicentralizer}, we can then find a finite dimensional abelian subalgebra $A \subset N$ with $A_0 \subset A$ such that
$$ \| (\xi- \eta)|_{A' \cap c(M)} \| \leq \varepsilon.$$
Then it follows that
$$ \| \xi|_{A' \cap c(M)} \| \leq \| \eta|_{A' \cap c(M)}\|+\varepsilon \leq \| \eta|_{A_0' \cap c(M)}\|+\varepsilon \leq \| \eta|_{N \vee (N' \cap c(M))} \| + 2\varepsilon.$$
\end{proof}

\begin{lem} \label{averaging type III1}
Let $N \subset M$ be an inclusion of von Neumann algebras with expectations. Suppose that $N$ is of type $\III_1$. Take $\varphi \in \cP(N)$. Then for every $\xi \in c(M)_*$ and every $\varepsilon > 0$, there exists $\eta \in c(M)_*$ such that 
$$\eta|_{N \vee (N' \cap c(M))} =\xi|_{N \vee (N' \cap c(M))} $$
and 
$$ \| \eta|_{\rb(N \subset c(M),\varphi)} \| \leq \| \xi|_{N' \cap c(M)} \|+ \varepsilon.$$
\end{lem}
\begin{proof}
We first assume that $\xi|_{N' \cap c(M)} =0$. Let $C$ be the convex hull of $$\{ \beta_\lambda^\sharp \circ E_{\rb^{\sharp}(N \subset c(M))} \mid \lambda \in \R^*_+ \}.$$ inside $\UCP(c(M))$. Since $\R^*_+$ is amenable, the closure of $C$ in $\UCP(c(M))$ contains some conditional expectation $E$ onto $\rb^{\sharp}(N \subset c(M))^{\beta^\sharp}=N \vee (N' \cap c(M))$. Observe that $E$ restricts to a conditional expectation from $\rb(N \subset c(M),\varphi))$ onto $\rb(N \subset c(M),\varphi))^{\beta}=N' \cap c(M)$. Thus, $(\xi \circ E)|_{\rb(N \subset c(M),\varphi)}=0$. This implies that $0$ is in the weak closure of
$$C_0= \{ (\xi \circ \Psi)|_{\rb(N \subset c(M),\varphi)}  \mid \Psi \in C\}.$$
Since $C_0$ is convex, it follows by the Hahn-Banach theorem that $0$ is in the norm closure of $C_0$, i.e.\ there exists $\Psi \in C$ such that $\eta :=\xi \circ \Psi$ satisfies
$$\| \eta|_{\rb(N \subset c(M), \varphi)} \| \leq \varepsilon.$$
And since all elements of $C$ restrict to the identity on $N \vee (N' \cap c(M))$, we also have $$\eta|_{N \vee (N \cap c(M))}=\xi|_{N \vee (N' \cap c(M))}$$
as we wanted.

Now we deal with the general case where $\xi$ does not necessarily vanish on $N' \cap c(M)$. We can find $\zeta \in c(M)_*$ such that $\zeta|_{N' \cap c(M)}=\xi|_{N' \cap c(M)}$ and $\| \zeta \| \leq \| \xi|_{N' \cap c(M)} \| +  \varepsilon$. Since $\xi- \zeta$ vanishes on $N' \cap c(M)$, by the first part, we can find $\rho \in c(M)_*$ such that 
$$\rho|_{N \vee (N' \cap c(M))} =(\xi - \zeta)|_{N \vee (N' \cap c(M))} $$ and 
$$\| \rho|_{\rb(N \subset c(M),\varphi)} \| \leq \varepsilon.$$
We let $\eta = \rho + \zeta$. Then we have 
$$\eta|_{N \vee (N' \cap c(M))} =\xi|_{N \vee (N' \cap c(M))} $$
and 
$$\| \eta|_{\rb(N \subset c(M),\varphi)} \|  \leq \| \rho|_{\rb(N \subset c(M),\varphi)} \| + \|\zeta\|  \leq  \| \xi|_{N' \cap c(M)} \|  + 2\varepsilon.$$
\end{proof}

We can now prove Theorem \ref{local quantization type III} when $N$ is amenable.

\begin{thm} \label{local quantization amenable}
Let $N \subset M$ be an inclusion of von Neumann algebras with expectations. Suppose that $N$ amenable and has no type $\I$ summand and $N_*$ is separable. Then $N \subset c(M)$ has the local quantization property.
\end{thm}
\begin{proof}
If $N$ is of type $\III_0$, we can write $N=\bigvee_{n \in \N} N_n$ where $(N_n)_{n \in \N}$ is an increasing sequence of type $\II$ subalgebras with expectation in $N$ (see \cite[Proposition 8.3]{HS90}). By Theorem \ref{local quantization type II}, $N_n' \cap c(M)$ is in the closure of 
$$ \Omega = \{ A' \cap c(M) \mid A \text{ is a finite dimensional abelian subalgebra of } N\}.$$
Since $\bigcap_{n \in \N} N_n' \cap c(M)=N' \cap c(M)$, it follows that $N' \cap c(M)$ is also in the closure of $\Omega$.

 Now, suppose that $N$ is of type $\III$ but has no type $\III_0$ summand. Take $F \subset c(M)_*$ and $\varepsilon > 0$. We want to find $A \subset N$ finite dimensional and abelian such that 
 $$ \forall \xi \in F, \quad \| \xi|_{A' \cap c(M)} \| \leq \|\xi|_{N' \cap c(M)} \| + \varepsilon.$$
Up to replacing the inclusion $N \subset M$ by the inclusion $N \subset M^{\oplus F}$, we may assume that $F$ is a singleton. Let $F=\{ \xi \}$ for some $\xi \in c(M)_*$. By Lemma \ref{averaging type III1}, we can find a faithful state $\varphi \in N_*$ and some $\eta \in c(M)_*$ such that 
$$\eta|_{N \vee (N' \cap c(M))}=\xi|_{N \vee (N' \cap c(M))}$$
and 
$$ \| \eta|_{\rb(N \subset c(M),\varphi)} \| \leq \| \xi|_{N' \cap c(M)} \|+ \varepsilon.$$
By Lemma \ref{fd abelian into bicentralizer}, we can find some finite dimensional abelian subalgebra $A_0 \subset N$ such that 
$$ \| \eta|_{A_0' \cap c(M)} \| \leq  \| \eta|_{\rb(N \subset c(M),\varphi)} \| + \varepsilon \leq  \| \xi|_{N' \cap c(M)} \| + 2 \varepsilon.$$
Since $(A_0' \cap N) \vee ((A_0' \cap N)' \cap c(M)) \subset N \vee (N' \cap c(M))$, we can apply Lemma \ref{semilocal quantization} to obtain some finite dimensional abelian subalgebra $A_0 \subset A \subset N$ such that 
$$ \| (\xi - \eta)|_{A' \cap c(M)} \| \leq \| (\xi-\eta)|_{N \vee (N' \cap c(M))} \| +\varepsilon =\varepsilon.$$
Therefore, we get
$$ \| \xi|_{A' \cap c(M)}\| \leq \| \eta|_{A' \cap c(M)}\| + \varepsilon \leq \| \eta|_{A_0' \cap c(M)}\| + \varepsilon \leq 3\varepsilon.$$
This shows that $N \subset c(M)$ has the local quantization property.
\end{proof}

In order to prove Theorem \ref{local quantization type III}, we make a reduction to the amenable case using the following theorem.

\begin{thm} \label{large amenable}
Let $N \subset M$ be an inclusion of von Neumann algebras with expectation and with separable predual. Suppose that $N \subset M$ satisfies the bicentralizer conjecture. Then there exists an amenable subalgebra with expectation $P \subset N$ such that $P' \cap c(M)=N' \cap c(M)$ and $P' \cap c(P)=N' \cap c(N)$.
\end{thm}

\begin{lem} \label{induction III1}
Let $N \subset M$ be an inclusion of von Neumann algebras with expectation and with separable predual. Suppose that $N$ is a type $\III_1$ factor and that $N \subset M$ satisfies the bicentralizer conjecture. Let $\varphi \in N_*$ be a faithful state. Take a finite dimensional subfactor $R \subset N$ with a $\varphi$-preserving conditional expectation $\rE_R$, some $\zeta \in N_*$ and some $\xi \in c(M)_*$ such that $\xi|_{\rb^\sharp(N \subset c(M))}=0$.

Then for every $\varepsilon > 0$, we can find a finite dimensional subfactor $F \subset R' \cap N$ and a faithful state $\psi \in F_*$ such that
\begin{enumerate}[\rm (i)]
\item $\| \varphi- \varphi \circ \rE_{F' \cap N} \| \leq \varepsilon$
\item $\| \zeta|_{F_\psi' \cap N} \| \leq \| \zeta|_R\| + \varepsilon$
\item $\| \xi|_{F_\psi' \cap c(M)} \| \leq \varepsilon$
\end{enumerate}
where $\rE_{F' \cap N} : N \rightarrow F' \cap N$ is the unique conditional expectation whose restriction to $F$ is equal to $\psi$. 

Moreover, for any given $0 < \mu, \nu < 1$ such that $\mu/\nu \notin \Q$, we can choose $(F,\psi)$ so that
\begin{equation} \label{eq form of state} (F,\psi) \cong (\mathbf M_{2^p}(\C), \tau) \otimes (\mathbf M_2(\C), \omega_\mu) \otimes (\mathbf M_2(\C), \omega_\nu)
\end{equation}
for some $p \in \N$.
\end{lem}
\begin{proof}
Inside $c(M^\omega)$, let $P=R' \cap N^\omega_{\varphi^\omega}=(R' \cap N)^\omega_{\varphi^\omega}$. We have 
$$P' \cap c(M^\omega) \subset R \vee \rB( R' \cap N \subset c(M), \varphi)^\omega \subset \rB^{\sharp}(N \subset c(M))^\omega$$ and
since $N \subset M$ satisfies the bicentralizer conjecture, it follows that
$$ \| \xi^\omega|_{P' \cap c(M^\omega)} \| \leq  \| \xi|_{\rB^\sharp(N \subset c(M))} \|=  \| \xi|_{\rb^\sharp(N \subset c(M))} \| = 0.$$ 
We also have
 $P' \cap N^\omega \subset R \vee \rB(R' \cap N,\varphi)^\omega=R$ because $R' \cap N \cong N$ is of type $\III_1$. This shows that 
 $$ \| \zeta^\omega|_{P' \cap N^\omega} \| \leq \| \zeta|_R\|.$$
 By Theorem \ref{local quantization type II}, we can find a finite dimensional abelian subalgebra $A \subset P$ such that $$\| \xi^\omega|_{A' \cap c(M^\omega)} \| \leq \varepsilon/2 \quad  \text{ and } \quad  \| \zeta^\omega|_{A' \cap N^\omega} \| \leq \| \zeta|_R\| + \varepsilon/2.$$
Up to perturbing $A$, we may assume that the sizes of its minimal projections are multiples of $2^{-p}$ for some $p \in \N$. This property implies that $A$ is contained in a finite dimensional subfactor $\cF_0 \subset P$ isomorphic to $M_{2^p}(\C)$. By adding to $\cF_0$ some $\sigma^{\varphi^\omega}$-invariant copy of $(\mathbf M_2(\C), \omega_\mu) \otimes (\mathbf M_2(\C), \omega_\nu)$ inside $(R' \cap \cF_0' \cap N)^\omega$, we obtain  a globally $\sigma^{\varphi^\omega}$-invariant subfactor $\cF_0 \subset \cF \subset (R' \cap N)^\omega$ such that $\psi=\varphi^{\omega}|_{\cF}$ satisfies
$$ (\cF,\psi) \cong  (\mathbf M_{2^p}(\C), \tau) \otimes (\mathbf M_2(\C), \omega_\mu) \otimes (\mathbf M_2(\C), \omega_\nu).$$
By \cite[Lemma 2.1]{AHHM18}, we can find a sequence of unital *-morphisms $\theta_n : \cF \rightarrow R' \cap N$ such that $x=(\theta_n(x))^\omega$ for all $x \in \cF$. Moreover, if $E$ is the conditional expectation from $N^\omega$ onto $\cF' \cap N^\omega$ induced by $\psi$ and $E_n$ is the conditional expectation from $N$ onto $\theta_n(\cF)' \cap N$ induced by $\psi$, then we have $\varphi^\omega = (\varphi \circ E_n)^\omega$. We also check easily that
$$\| \zeta^\omega|_{A' \cap N^\omega} \| = \| \zeta^\omega \circ E_{A' \cap N^\omega} \| =\lim_{n \to \omega} \| \zeta \circ E_{\theta_n(A)' \cap N} \| =  \lim_{n \to \omega} \| \zeta|_{\theta_n(A)' \cap N} \|$$
and 
$$\| \xi^\omega|_{A' \cap c(M^\omega)} \|  = \| \zeta^\omega \circ E_{A' \cap c(M^\omega)} \| =\lim_{n \to \omega} \| \zeta \circ E_{\theta_n(A)' \cap c(M)} \| =  \lim_{n \to \omega} \| \xi|_{\theta_n(A)' \cap c(M)} \|.$$
We conclude by taking $F=\theta_n(\cF)$ for $n$ large enough along the ultrafilter $\omega$.
\end{proof}

With a similar proof, but much easier because we don't have to use ultrafilters and we don't have to perturb the state, we obtain the following type $\III_\lambda$ analog.
\begin{lem} \label{induction III lambda}
Let $N \subset M$ be an inclusion of von Neumann algebras with expectation and with separable predual. Suppose that $N$ is a type $\III_\lambda$ factor with $0 < \lambda < 1$. Let $\varphi$ be a $\lambda$-trace on $N$ so that $N_\varphi' \cap N=\C$. Take a finite dimensional subfactor $R \subset N$ with a $\varphi$-preserving conditional expectation $\rE_R$, some $\zeta \in N_*$ and some $\xi \in c(M)_*$ such that $\xi|_{\rb^\sharp(N \subset c(M))}=0$.

Then for every $\varepsilon > 0$, we can find a finite dimensional subfactor $F \subset R' \cap N$ that is globally invariant under $\sigma^\varphi$ such that
\begin{enumerate}[\rm (i)]
\item $\| \zeta|_{F_\varphi' \cap N} \| \leq \| \zeta|_R\|+ \varepsilon$
\item $\| \xi|_{F_\varphi' \cap c(M)} \| \leq \varepsilon$
\end{enumerate}
where $\rE_{F' \cap N} : N \rightarrow F' \cap N$ is the $\varphi$-preserving conditional expectation. Moreover, we can choose $F$ such that
$$ (F,\varphi|_F) \cong (\mathbf M_{2^p}(\C), \tau) \otimes (\mathbf M_2(\C), \omega_\lambda)$$
for some $p \in \N$.
\end{lem}

\begin{proof}[Proof of Theorem \ref{large amenable}]
Using a standard desintegration and measurable selection argument, we may reduce the problem to the case where $N$ is a factor. We treat all cases separately according to the type of $N$.

\textbf{Semifinite case.} In this case, we already know by Popa's theorem \cite{Po81} that there exists a semifinite amenable subfactor with expectation $P \subset N$ such that $P' \cap M=N' \cap M$ and $P' \cap N=\C$. Since $N$ is semifinite, this automatically implies that $P' \cap c(M)=N' \cap c(M)$ and $P' \cap c(N)=P' \cap c(P)$. 

\textbf{Type $\III_1$ case.}
Take a dense sequence $(\zeta_n)_{n \in \N}$ in  $N_*$ and a dense sequence $(\xi_n)_{n \in \N}$ in  $\{ \xi \in c(M)_* \mid \xi|_{\rb^\sharp(N \subset c(M))}=0 \}$.

Using Lemma \ref{induction III1}, we can construct by induction a sequence $(F_n)_{n \in \N}$ of finite dimensional subfactors of $N$ with faithful states $\psi_n \in (F_n)_*$ of the form \ref{eq form of state} that satisfy the following properties: $F_{n} \subset (R_n)' \cap N$ for all $n \in \N$ where $R_0=\C$ and $R_n=F_0 \vee F_1 \vee \cdots \vee F_{n-1}$ for all $n \geq 1$, and 
\begin{enumerate}[(i)]
\item $ \| \varphi \circ \rE_{R_{n+1}' \cap N} - \varphi \circ \rE_{R_{n}' \cap N} \| \leq 2^{-n}$
\item $\| \zeta_n|_{(F_{n})_{\psi_n}' \cap N} \| \leq \| \zeta_n|_{R_n} \| + 2^{-n}$
\item $ \| \xi_n|_{(F_{n})_{\psi_n}'  \cap c(M)} \|  \leq 2^{-n}$
\end{enumerate}

where $\rE_{R_n'\cap N}$ is the conditional expectation on $R_n' \cap N$ induced by the state $\psi_0 \otimes \psi_1 \otimes \dots \otimes \psi_n \in (R_n)_*$ and $\rE_{R_n}$ is the conditional expectation on $R_n$ that preserves $\varphi_n=\varphi \circ \rE_{R_n' \cap N}$.

The condition $(\rm i)$ implies that the sequence of states $\varphi_n=\varphi \circ \rE_{R_n' \cap N}$ is a Cauchy sequence in $N_*$ that converges to some state $\phi \in N_*$ satisfying  $\phi =\phi \circ \rE_{R_n' \cap N}$ for all $n \in \N$. Let $e \in N$ be the support projection of $\phi$. Then for all $n \in \N$, we  have $e \in R_n' \cap N$, hence $\varphi(e)=\varphi_n(e)$. Taking the limit when $n \to \infty$, we get $\varphi(e)=\phi(e)=1$. Since $\varphi$ is faithful, this means that $e=1$, i.e.\ $\phi$ is faithful.

By construction, all the algebras $R_n$ are globally invariant under $\sigma^{\phi}$. It follows that their union generates an AFD type $\III_1$ subfactor $P=\bigvee_n R_n$ that is globally invariant under $\sigma^{\phi}$.
We claim that $P_{\phi}' \cap N=\C$. Indeed, item $(\rm ii)$ implies that 
$$\| \zeta_n|_{P_{\phi}' \cap N} \|  \leq \| \zeta_n|_{(F_{n})_{\psi_n}' \cap N} \| \leq \| \zeta_n|_{R_n} \| + 2^{-n} \leq \| \zeta_n|_P \| + 2^{-n}$$ for all $n \in \N$. Since $(\zeta_n)_n$ is dense in $N_*$, we obtain $\| \zeta|_{P_{\phi}' \cap N} \| \leq \| \zeta|_P\|$ for all $\zeta \in N_*$. We conclude that $P_{\phi}'  \cap N \subset P$, hence $P_{\phi}' \cap N =P_{\phi}' \cap P=\C$. Moreover, item $(\rm iii)$ shows that $\| \xi_n |_{P_{\phi}' \cap c(M)} \| \leq 2^{-n}$ for all $n \in \N$. By density, it follows that  $\xi |_{P_{\phi}' \cap c(M)}=0$ for all $\xi \in c(M)_*$ such that $\xi|_{\rb^{\sharp}(N \subset c(M))}=0$. This means that  $P_{\phi}' \cap c(M) \subset \rb^{\sharp}(N \subset c(M))$ hence $$P_{\phi}' \cap c(M) = (P_{\phi}' \cap N) \vee \rb(N \subset c(M), \phi)= \rb(N \subset c(M),\phi).$$
Since $\phi|_P$ is almost periodic, we conclude that
$$ P' \cap c(M)=P' \cap \rb(N \subset c(M),\phi)=\rb(N \subset c(M),\phi)^{\beta^{\phi}}=N' \cap c(M).$$

\textbf{Type $\III_\lambda$, $0 < \lambda < 1$ case.} Take $\varphi$ a $\lambda$-trace on $N$. Following an inductive procedure similar to the type $\III_1$ case, we can construct a type $\III_\lambda$ subfactor $P$ that is globally invariant under $\sigma^\varphi$ such that $P_\varphi' \cap N=P_\varphi' \cap P=\C$ and $P_\varphi' \cap c(M) =\rb(N \subset c(M),\varphi)=N_{\varphi}' \cap c(M)$. But $N$ is generated by $N_\varphi$ and $P$, hence
$$ P' \cap c(M)=P' \cap P_\varphi' \cap c(M)=P' \cap N_\varphi' \cap c(M)=N' \cap c(M)$$
as we wanted. Moreover, since $P_\varphi' \cap N=\C$, we have $P' \cap c(N) \subset P_\varphi' \cap c(N) \subset c(P)$, hence $P' \cap c(N)=P' \cap c(P)$.

\textbf{Type $\III_0$ case.} The proof is exactly the same as in \cite[Corollary 3.4]{Po85}. Indeed, by \cite[5.3.6]{Co72}, we can find a discrete decomposition $N=Q \rtimes_\theta \Z$ where $Q$ is a type $\II_\infty$ von Neumann algebra and $\theta \in \Aut(Q)$ acts as an odometer on the center of $Q$. Since $Q$ is semifinite and with expectation in $M$, it is also with expectation in $c(M)$. Therefore, we can find an abelian subalgebra with expectation $A \subset Q$ such that $A' \cap c(M)=A \vee (Q' \cap c(M))$. Now, the exact same proof as in \cite{Po85} shows that we can actually construct $A$ together with some hyperfinite $\II_\infty$ subalgebra with expectation $Q_0 \subset Q$ and a unitary $u \in Q$ such that $A$ is a cartan subalgebra of $Q_0$, $\cZ(Q_0)=\cZ(Q)$ and $\theta'$ normalizes $Q_0$ and $A$. Let $P = Q_0 \rtimes_{\theta'} \Z \subset N$. Then $Q_0' \cap c(M)=\cZ(Q_0) \vee (Q' \cap c(M))=Q' \cap c(M)$. Since $N$ is generated by $P$ and $Q$, it follows that $P' \cap c(M)=N' \cap c(M)$. Moreover, since $P$ contains with expectation a maximal abelian subalgebra of $N$, we also have $P' \cap c(N)=P' \cap c(P)$ as we wanted.
\end{proof}

\begin{thm} \label{irreducible amenable}
Let $N \subset M$ be an inclusion of von Neumann algebras with expectation and with separable preduals. Suppose that $N$ satisfies the bicentralizer conjecture. Then there exists an amenable subalgebra with expectation $P \subset N$ such that $P' \cap M=N' \cap M$.
\end{thm}
\begin{proof}
Using a desintegration argument, we may assume that $N$ is a factor. If $N$ is not of type $\III_1$, we can apply Theorem \ref{large amenable}. If $N$ is of type $\III_1$ with trivial bicentralizer, we can combine Corollary \ref{bicentralizer flow ergodic} and \cite{AHHM18}.
\end{proof}

\begin{lem} \label{separable step}
Let $N \subset M$ be an inclusion of von Neumann algebras with a faithful normal state $\varphi$ on $M$ such that $N$ is globally $\sigma^\varphi$-invariant.

Then for every countable subset $X \subset M$, we can find a countably generated $\sigma^\varphi$-invariant subalgebra $P \subset N$ such that $E_{\rB(P \subset M,\varphi)}(x)=E_{\rB(N \subset M,\varphi)}(x)$ for all $x \in X$.
\end{lem}
\begin{proof}
We may assume that $X=\{x\}$ for some $x \in M$. For every $n \in \N$, we can find a finite set $F_n \subset \cU(N)$ such that $\|u\varphi u^*-\varphi\| \leq 2^{-n}$ for all $u \in F_n$ and 
$$\| \sum_{u \in F_n} \lambda(u) uxu^*-E_{\rB(N \subset M,\varphi)}(x) \|_\varphi \leq 2^{-n}$$
for some probability measure $\lambda \in \mathrm{Prob}(F_n)$. We take $P$ to be the von Neumann algebra generated by $\bigcup_{t \in \Q} \bigcup_{n \in \N} \sigma^\varphi_t(F_n)$.
\end{proof}

\begin{lem} \label{separable reduction}
Let $N \subset M$ be an inclusion of von Neumann algebras with a faithful normal state $\varphi$ on $M$ such that $N$ is globally $\sigma^\varphi$-invariant. Suppose that the inclusion $N \subset M$ satisfies the bicentralizer conjecture.

Then for every pair of countably generated subalgebras $P_0 \subset N$ and $Q_0 \subset M$, we can find a pair of countably generated $\sigma^\varphi$-invariant subalgebras $P \subset N$ and $Q \subset M$ such that $P \subset Q$ and the inclusion $P \subset Q$ satisfies the bicentralizer conjecture.
\end{lem}
\begin{proof}
Without loss of generality, we may assume that $P_0 \subset Q_0$ and that $P_0$ and $Q_0$ are $\sigma^\varphi$-invariant. We construct recursively a sequence $(P_n,Q_n)$ of $\sigma^\varphi$-invariant subalgebras with $P_n \subset Q_n$ such that 
$$E_{\rB(P_n \subset M,\varphi)}(x)=E_{\rB(N \subset M,\varphi)}(x)$$
for all $x \in Q_{n-1}$. 
Suppose that we have already constructed $(P_{n-1},Q_{n-1})$. By applying Lemma \ref{separable step} to a dense countable subset of the unit balle of $Q_{n-1}$, we obtain the desired $P_n$. Then we simply put $Q_n=P_n \vee Q_{n-1}$.

Now, ler $P=\bigvee_{n \in \N} P_n$ and $Q=\bigvee
_{n \in \N} Q_n$. Then for all $x \in Q_{n}$, we have
$$ E_{\rB(P \subset M,\varphi)}(x) = E_{\rB(P \subset M,\varphi)}( E_{\rB(P_{n+1} \subset M,\varphi)}(x))=E_{\rB(P \subset M,\varphi)}( E_{\rB(N \subset M,\varphi)}(x))=E_{\rB(N \subset M,\varphi)}(x).$$
Since this holds for every $x \in Q_{n}$ and for every $n \in \N$, we get
$$ E_{\rB(P \subset M,\varphi)}(x)=E_{\rB(N \subset M,\varphi)}(x)=E_{\rb(N \subset M,\varphi)}(x)$$
for all $x \in Q$. By taking $x \in \rB(P \subset Q,\varphi)$, we obtain $\rB(P \subset Q,\varphi) \subset \rb(N \subset M,\varphi) \subset \rb(P \subset M,\varphi)$, hence $\rB(P \subset Q,\varphi)=\rb(P \subset Q,\varphi)$.
\end{proof}

\begin{proof}[Proof of Theorem \ref{local quantization type III}]
We first prove the if direction. Suppose that $N \subset c(M)$ has the local quantization property. Take $(A_i)_{i \in I}$ a net of finite dimensional abelian subalgebras of $N$ such that $A_i' \cap c(M)$ converges to $N' \cap c(M)$. Since $A_i$ is abelian, we have $E_{A_i' \cap c(M)} \in \cD(N \subset c(M))$ for all $i \in I$. Let $\Phi \in \cD(N \subset c(M))$ be an accumulation point of the net $(E_{A_i' \cap c(M)})_{i \in I}$. Take $\xi \in c(M)_*$ such that $\xi|_{N' \cap c(M)}=0$. Then
$$ \| \xi \circ E_{A_i' \cap c(M)} \| = \| \xi|_{A_i' \cap c(M)} \| \to \| \xi|_{N' \cap c(M)}\|=0$$
when $i \to \infty$. In particular, $\xi \circ E_{A_i' \cap c(M)} $ converges weakly to $0$, hence $\xi \circ \Phi=0$. Since this holds for every $\xi \in c(M)_*$ that vanishes on $N' \cap c(M)$, the Hahn-Banach theorem implies that the image of $\Phi$ is contained in $N' \cap c(M)$. We conclude that $N \subset c(M)$ has the weak Dixmier property, hence $N \subset M$ satisfies the bicentralizer conjecture by Theorem \ref{conjecture iff dixmier}.

Now, we deal with the only if direction. If $M_*$ is separable, then by Theorem \ref{large amenable}, we can find an amenable subalgebra with expectation $P \subset N$ such that $P' \cap c(M)=N' \cap c(M)$. Then the local quantization property for $N \subset c(M)$ follows immediately from the local quantization for $P \subset c(M)$, which holds by Theorem \ref{local quantization amenable}.

We now deal with the case where $M$ is countably decomposable but not necessarily countably generated. Take $F \subset c(M)_*$ a finite set. Fix a faithful normal state $\varphi$ on $M$ such that $N$ is $\sigma^\varphi$-invariant. Using Property (P\ref{approx increase decrease}), take $P \subset N$ a countably generated $\sigma^\varphi$-invariant subalgebra that is so large that $\| \xi|_{P' \cap c(M)} \| \leq \| \xi_{N' \cap c(M)} \| +\varepsilon$ for all $\xi \in F$. Take $Q \subset M$ a countably generated $\sigma^\varphi$-invariant subalgebra that contains $P$ and that is so large that $\| \xi -\xi \circ c(E) \| \leq \varepsilon$ for all $\xi \in F$, where $E \in \cE(M,Q)$ is the $\varphi$-preserving conditional expectation and $c(E) \in \cE(c(M),c_E(Q))$ is the core extension of $E$. Thanks to Lemma \ref{separable reduction}, we may assume that the inclusion $P \subset Q$ satisfies the bicentralizer conjecture. Since $Q$ has separable predual, by the first part of the proof, we can find a finite dimensional abelian subalgebra $A \subset P$ such that for all $\xi \in F$, we have
$$\| \xi|_{A' \cap c_E(Q)} \| \leq \| \xi|_{P' \cap c_E(Q)} \| +\varepsilon.$$
But since $\| \xi- \xi \circ c(E) \| \leq \varepsilon$, we also have
$$ \| \xi|_{A' \cap c(M)} \| \leq \| (\xi \circ c(E))|_{A' \cap c(M)} \| +\varepsilon =  \| \xi|_{A' \cap c_E(Q)} \|+ \varepsilon$$
and 
$$\| \xi|_{P' \cap c_E(Q)} \| = \| (\xi \circ c(E))|_{P' \cap c(M)} \|  \leq \| \xi|_{P' \cap c(M)} \|  + \varepsilon.$$
Combining all these inequalities, we obtain
$$  \| \xi|_{A' \cap c(M)} \| \leq \| \xi|_{P' \cap c(M)} \|  + 3\varepsilon.$$
Finally, since $\| \xi|_{P' \cap c(M)} \| \leq \| \xi|_{N' \cap c(M)}\| \leq \varepsilon$ by the choice of $P$, we conclude that 
$$  \| \xi|_{A' \cap c(M)} \| \leq \| \xi|_{N' \cap c(M)} \|  + 4\varepsilon.$$

This proves the local quantization property when $M$ is countably decomposable. The case where $M$ is not countably decomposable follows easily by cutting with arbitrarily large countably decomposable projections $e \in N$ and $f \in N' \cap M$ as in the proof of Theorem \ref{local quantization type II}.
\end{proof}

\begin{proof}[Proof of Theorem \ref{local quantization no core type III}] Suppose first that $M_*$ is separable. By Theorem \ref{irreducible amenable}, we can find $P \subset N$ an amenable subalgebra with expectation such that $P'\cap M=N' \cap M$. By Theorem \ref{local quantization type III}, the inclusion $P \subset c(M)$ has the local quantization property. By Proposition \ref{local quantization restriction}, it follows that $P \subset M$ also has the local quantization property. Since $P' \cap M=N' \cap M$, we conclude that $N \subset M$ has the local quantization property.

We now deal with the case where $M$ is countably decomposable but not necessarily countably generated. Take $F \subset c(M)_*$ a finite set. Fix a faithful normal state $\varphi$ on $M$ such that $N$ is $\sigma^\varphi$-invariant. Take $P \subset N$ a countably generated $\sigma^\varphi$-invariant subalgebra that is so large that $\| \xi|_{P' \cap M} \| \leq \| \xi_{N' \cap M} \| +\varepsilon$ for all $\xi \in F$. Thanks to Lemma \ref{separable reduction}, we may assume that the inclusion $P$ satisfies the bicentralizer conjecture. Take $Q \subset M$ a countably generated $\sigma^\varphi$-invariant subalgebra that contains $P$ and that is so large that $\| \xi -\xi \circ E \| \leq \varepsilon$ for all $\xi \in F$, where $E \in \cE(M,Q)$ is the $\varphi$-preserving conditional expectation. Since $Q$ has separable predual, by the first part of the proof, we can find a finite dimensional abelian subalgebra $A \subset P$ such that for all $\xi \in F$, we have
$$\| \xi|_{A' \cap Q} \| \leq \| \xi|_{P' \cap Q} \| +\varepsilon.$$
But since $\| \xi- \xi \circ E \| \leq \varepsilon$, we also have
$$ \| \xi|_{A' \cap M} \| \leq \| (\xi \circ E)|_{A' \cap M} \| +\varepsilon =  \| \xi|_{A' \cap Q} \|+ \varepsilon$$
and 
$$\| \xi|_{P' \cap Q} \| = \| (\xi \circ E)|_{P' \cap M} \|  \leq \| \xi|_{P' \cap M} \|  + \varepsilon.$$
Combining all these inequalities, we obtain
$$  \| \xi|_{A' \cap M} \| \leq \| \xi|_{P' \cap M} \|  + 3\varepsilon.$$
Finally, since $\| \xi|_{P' \cap M} \| \leq \| \xi|_{N' \cap M}\| \leq \varepsilon$ by the choice of $P$, we conclude that 
$$  \| \xi|_{A' \cap M} \| \leq \| \xi|_{N' \cap M} \|  + 4\varepsilon.$$

This proves the local quantization property when $M$ is countably decomposable. The case where $M$ is not countably decomposable follows easily by cutting with arbitrarily large countably decomposable projections $e \in N$ and $f \in N' \cap M$ as in the proof of Theorem \ref{local quantization type II}.
\end{proof}

The following theorem is proved by reducing to the factor case and then using \cite[Theorem 3.2]{Po81} (for the type $\II$ and type $\III_\lambda, \lambda \in ]0,1[$ cases), \cite[Corollary 3.4]{Po85} (for the type $\III_0$ case), and \cite[Theorem D]{Ma18} (for the type $\III_1$ case).
\begin{thm} \label{irreducible amenable one}
Let $N$ be a von Neumann algebra with separable predual. Then there exists an amenable subalgebra with expectation $P \subset N$ such that $P' \cap N=\cZ(N)$.
\end{thm}

\begin{proof}[Proof of Theorem \ref{local quantization for one}] It is sufficient to prove the theorem when $N_*$ is separable. By Theorem \ref{irreducible amenable one}, we can find an amenable subalgebra with expectation $P \subset N$ such that $P' \cap N =\cZ(N)$. Since $N$ has no type $\I$ summand, the same holds for $P$. Then Theorem \ref{local quantization no core type III} gives the desired local quantization property.
\end{proof}

\begin{lem} \label{fd abelian masa}
Let $N$ be a von Neumann algebra that satisfies the bicentralizer conjecture. Let $\varphi \in N_*$ be a faithful state. Then for every $\xi \in N_*$ and every $\varepsilon > 0$, there exists a finite dimensional abelian subalgebra $A \subset N$ such that $$\| \xi |_{A' \cap N} \| - \| \xi |_A\| \leq \varepsilon \quad \text{ and } \quad  \| \varphi - \varphi \circ E_{A' \cap N} \| \leq \varepsilon.$$
\end{lem}
\begin{proof}
Since $\rb(N,\varphi)$ is abelian and thanks to property (P\ref{approx increase decrease}), we can take a finite dimensional abelian subalgebra $A_0 \subset \rb(N,\varphi)$ such that $\|\xi|_{\rb(N,\varphi)} \| \leq  \| \xi|_{A_0}\| + \varepsilon$. By applying Lemma \ref{fd abelian into bicentralizer} to $A_0' \cap N$ and the state $\varphi|_{A_0' \cap N}$, we find a finite dimensional abelian subalgebra $A_1 \subset A_0' \cap N$ such that 
$$ \| \xi|_{A_1' \cap A_0' \cap N} \| \leq \| \xi |_{\rb(A_0' \cap N, \varphi|_{A_0' \cap N}) } \| +\varepsilon \quad  \text{ and } \quad  \| (\varphi \circ E
_{A_1' \cap A_0' \cap N} -\varphi)|_{A_0' \cap N} \| \leq \varepsilon .$$
Observe that $\rb(A_0' \cap N, \varphi|_{A_0' \cap N}) = \rb(N,\varphi)$ and
$$  \| (\varphi \circ E
_{A_1' \cap A_0' \cap N} -\varphi)|_{A_0' \cap N} \|  =  \| (\varphi \circ E
_{A_1' \cap A_0' \cap N} -\varphi) \circ E_{A_0' \cap N} \| = \| \varphi \circ E_{A_1' \cap A_0' \cap N} -\varphi \|.$$
Therefore, if we let $A=A_0 \vee A_1$, we obtain 
$$ \| \xi|_{A' \cap N} \| \leq \| \xi|_{\rb(N,\varphi)} \| + \varepsilon \leq \| \xi|_{A_0} \| + 2\varepsilon \leq \| \xi|_A\| + 2 \varepsilon$$
and 
$$ \| \varphi \circ E_{A' \cap N} - \varphi \| \leq \varepsilon$$
as we wanted.
\end{proof}

\begin{thm} \label{masa with expectation if and only if}
Let $N \subset M$ be an inclusion of von Neumann algebras with expectation and with separable preduals. Then $N \subset M$ satisfies the bicentralizer conjecture if and only if there exists a maximal abelian subalgebra with expectation $A \subset N$ such that $$A' \cap c(M)=c(A) \vee (N' \cap c(M)).$$
Moreover, for any given faithful state $\varphi \in N_*$ and any $\varepsilon > 0$, we can choose $A$ such that $\| \varphi \circ E_A - \varphi \| \leq \varepsilon$.
\end{thm}
\begin{proof} 
First, we prove the if direction. Suppose that such a maximal abelian subalgebra $A$ exists. Then $N$ satisfies the bicentralizer conjecture. Let $\varphi \in N_*$ be a faithful state such that $A \subset N_\varphi$. Then $$\rB(N \subset c(M)) \subset N_\varphi' \cap c(M) \subset A' \cap c(M) =c(A) \vee (N' \cap c(M)) \subset \rb^\sharp(N \subset c(M)).$$
It follows that $\rB^\sharp(N \subset c(M))=\rb^\sharp(N \subset c(M))$ hence $\rB^\sharp(N \subset M)=\rb^\sharp(N \subset M)$. We conclude by Proposition \ref{conjecture with intermediate}.

Let us now prove the only if direction. Let $\varphi \in N_*$ be a faithful state and let $\varepsilon > 0$.  Let $(\xi_n)_{n \in \N}$ be a dense sequence in $c(M)_*$ and $(\zeta_n)_{n \in \N}$ a dense sequence in $N_*$. By applying alternatively Lemma \ref{fd abelian into bicentralizer} and Lemma \ref{fd abelian masa}, we construct inductively an increasing sequence $(A_n)_{n \in \N}$ of finite dimensional abelian subalgebras  of $N$ with such that 
\begin{enumerate}
\item $\| \varphi \circ E_{A_n' \cap N}- \varphi \circ E_{A_{n-1}' \cap N} \| \leq 2^{-n} \varepsilon $ for all $n \in \N$.
\item $\| \zeta_n |_{A_n' \cap N} \| - \| \zeta_n |_{A_n}\| \leq 2^{-n}$ when $n$ is odd.
\item $\| \xi_n |_{A_n' \cap c(M)}\| \leq \| \xi_n|_{\rb^{\sharp}(N \subset c(M))} \| +  2^{-n}$ when $n$ is even.
\end{enumerate}
Put $A=\bigvee_{n \in \N} A_n$.

 Condition (1) implies that the sequence of states $\varphi_n=\varphi \circ E_{A_n' \cap N}$, $n \in \N$ is a Cauchy sequence. Thus, it converges to some state $\varphi_\infty \in N_*$ such that $\| \varphi_\infty - \varphi \| \leq \varepsilon$ and $\varphi_\infty = \varphi_\infty \circ E_{A_n' \cap N}$ for all $n \in \N$. Let $e$ be the support projection of $\varphi_\infty$. Then, for all $n \in \N$, we have $e \in A_n' \cap N$, hence $\varphi(e)=\varphi_n(e)$. Taking the limit when $n \to \infty$, we obtain $\varphi(e)=\varphi_\infty(e)=1$. Since $\varphi$ is faithful, we conclude that $e=1$, i.e.\ $\varphi_\infty$ is faithful.
 
Observe that $A_n \subset N_{\varphi_\infty}$ for all $n \in \N$ and therefore $A \subset N_{\varphi_\infty}$, hence there exists a $\varphi_\infty$-preserving conditional expectation $E_A \in \cE(N,A)$.
 
 Condition (2) implies that $\| \zeta_n|_{A' \cap N} \| - \| \zeta_n |_A \| \leq \| \zeta_n|_{A_n' \cap N} \| - \| \zeta_n |_{A_n} \|  \leq 2^{-n}$ for all odd $n \in \N$. Therefore, by density of $(\zeta_n)_{n \in \N}$ in $N_*$, we get $\| \zeta |_{A' \cap N} \| = \| \zeta  |_A \| $ for all $\zeta \in N_*$, which means that $A'\cap N=A$ by the Hahn-Banach theorem.

  Condition (3) implies that $$\| \xi_n |_{A' \cap c(M)}\|  \leq \| \xi_n |_{A_n' \cap c(M)}\| \leq \| \xi_n|_{\rb^{\sharp}(N \subset c(M))} \| +  2^{-n}$$
  for all even $n \in \N$. Since $(\xi_n)_{n \in \N}$ is dense in $c(M)_*$, it follows that 
$$\| \xi |_{A' \cap c(M)}\|  \leq \| \xi|_{\rb^{\sharp}(N \subset c(M))} \| $$
 for every $\xi \in c(M)_*$, hence $A' \cap c(M) \subset \rb^{\sharp}(N \subset c(M))$ by the Hahn-Banach theorem. 
 
Let us now show that $A' \cap c(M)=c(A) \vee (N' \cap c(M))$. Let $F$ be the unique element of $\cE(c(M),A' \cap c(M))$. We have $F(c(N))=A' \cap c(N)=c(A)$ and $F$ fixes $N' \cap c(M)$. Since $c(N) \cdot (N' \cap c(M))$ is dense in $\rb^\sharp(N \subset c(M))$, it follows that $$F(\rb^\sharp(N \subset c(M)))=c(A) \vee (N' \cap c(M)).$$
 Since $A' \cap c(M) \subset \rb^\sharp(N \subset c(M))$, we conclude that
$$ A' \cap c(M) = c(A) \vee (N' \cap c(M)).$$
\end{proof}

\begin{cor}
Let $N \subset M$ be an inclusion of von Neumann algebras with expectation and with separable preduals. Suppose that $N \subset M$ satisfies the bicentralizer conjecture and $N' \cap c(M) \subset c(N)$. Then there exists a maximal abelian subalgebra with expectation $A \subset N$ such that $A$ is maximal abelian in $M$.
\end{cor}
\begin{proof}
By Theorem \ref{masa with expectation if and only if}, take  $A$ a maximal abelian subalgebra with expectation in $N$ as in such that $A' \cap c(M) =c(A) \vee (N' \cap c(M))$. Since $N' \cap c(M) \subset c(N)$, this means that $A' \cap c(M) \subset c(N)$, hence $A' \cap M \subset N$. We conclude that $A$ is maximal abelian in $M$.
\end{proof}

\bibliographystyle{plain}

\begin{thebibliography}{AHHM18}

\bibitem[Aa68]{Aa68} {\sc J. F. Aarnes}, {\it On the Mackey-topology for a von Neumann algebra.} Math. Scan. {\bf 22} (1968), 87--107.




\bibitem[AH12]{AH12} {\sc H. Ando, U. Haagerup}, {\it Ultraproducts of von Neumann algebras.} J. Funct. Anal. {\bf 266} (2014), 6842--6913.

\bibitem[AHHM18]{AHHM18} {\sc H. Ando, U. Haagerup, C. Houdayer and A. Marrakchi}, {\it Structure of bicentralizer algebras and inclusions of type III factors.} Math. Ann. {\bf 376} (2020), 1145--1194.


\bibitem[Bi24]{Bi24} {\sc P. Bikram}, {\it Connes' Bicentralizer Problem for Mixed $q$-deformed Araki-Woods Algebras.}  {\tt arXiv:2410.09490}.
 

\bibitem[BMO19]{BMO19} {\sc J. Bannon, A. Marrakchi, N. Ozawa}, {\it Full factors and co-amenable inclusions.} Comm. in math. phys. {•bf 378 (2)}, 1107--1121.

\bibitem[Ca22]{Ca22} {\sc M. Caspers}, {\it On the isomorphism class of $q$-Gaussian W*-algebras for infinite variables.} To appear in Comptes Rendus de l'Académie des Sciences. {\tt arXiv:2210.11128}.

%

\bibitem[Co72]{Co72} {\sc A. Connes}, {\it Une classification des facteurs de type ${\rm III}$.} Ann. Sci. \'{E}cole Norm. Sup. {\bf 6} (1973), 133--252.

%


\bibitem[Co75b]{Co75b} {\sc A. Connes}, {\it Classification of injective factors. Cases ${\rm II_1}$, ${\rm II_\infty}$, ${\rm III_\lambda}$, $\lambda \neq 1$.} Ann. of Math. {\bf 74} (1976), 73--115.

\bibitem[Co80]{Co80} {\sc A. Connes}, {\it Classification des facteurs.} In ``Operator algebras and applications, Part 2 (Kingston, 1980)'' Proc. Sympos. Pure Math. {\bf 38} Amer. Math. Soc., Providence, 1982, pp.\ 43--109.

\bibitem[Co85]{Co85} {\sc A. Connes}, {\it Factors of type $\III_1$, property $\rL_{\lambda}'$, and closure of inner automorphisms.} J. Operator Theory {\bf 14} (1985), 189--211.

\bibitem[CS78]{CS78} {\sc A. Connes, E. St\o rmer}, {\it Homogeneity of the state space of factors of type $\III_1$}, J. Funct. Anal. {\bf 28} (1978), 187--196.
%
\bibitem[CT76]{CT76} {\sc A. Connes, M. Takesaki}, {\it The flow of weights of factors of type ${\rm III}$.} Tohoku Math. Journ. {\bf 29} (1977), 473--575.

\bibitem[DKEP23]{DKEP23} {\sc C. Ding, S. Kunnawalkam Elayavalli,  J. Peterson}, {\it Properly proximal von Neumann algebras.} Duke Math. J. {\bf 172} (2023), 2821--2894.

\bibitem[DP23]{DP23} {\sc C. Ding,  J. Peterson}, {\it Biexact von Neumann algebras.} preprint {\tt arXiv:2309.10161}



\bibitem[FT01]{FT01} {\sc T. Falcone, M. Takesaki}, {\it The Non-commutative Flow of Weights on a Von Neumann Algebra.} J. Funct. Anal. {\bf 182} (2001), no. 1, 170--206.

\bibitem[GP98]{GP98} {\sc L. Ge, S. Popa}, {\it On some decomposition properties for factors of type II$_1$}, Duke
Math. J. {\bf 94} (1998), 79--101.

%

\bibitem[Ha75]{Ha75} {\sc U. Haagerup}, {\it The standard form of von Neumann algebras}, Math. Scand. \textbf{37} (1975), 271--283.
%
\bibitem[Ha77a]{Ha77a} {\sc U. Haagerup}, {\it Operator valued weights in von Neumann algebras}, I. J. Funct. Anal. {\bf 32} (1979), 175--206.
%
\bibitem[Ha77b]{Ha77b} {\sc U. Haagerup}, {\it Operator valued weights in von Neumann algebras}, II. J. Funct. Anal. {\bf 33} (1979), 339--361.
%

\bibitem[Ha79]{Ha79} {\sc U. Haagerup}, {\it $\rL^p$-spaces associated with an arbitrary von Neumann algebra.} In Alg\`ebres d’op\'erateurs et leurs applications en physique math\' ematique (Proc. Colloq., Marseille, 1977), volume 274 of Colloq. Internat. CNRS, pages 175–184. CNRS, Paris, 1979.

\bibitem[Ha84]{Ha84} {\sc U. Haagerup}, {\it A new proof of the equivalence of injectivity and hyperfiniteness for factors on a separable Hilbert space.} J. Funct. Anal. {\bf 62} (1985), 160--201.

\bibitem[Ha85]{Ha85} {\sc U. Haagerup}, {\it Connes' bicentralizer problem and uniqueness of the injective factor of type ${\rm III_1}$.} Acta Math. {\bf 69} (1986), 95--148.
%
\bibitem[HI15]{HI15} {\sc C. Houdayer, Y. Isono}, {\it Unique prime factorization and bicentralizer problem for a class of type ${\rm III}$ factors.} Adv. Math. {\bf 305} (2017), 402--455.

\bibitem[HI20]{HI20} {\sc C. Houdayer, Y. Isono}, {\it Connes' bicentralizer problem for q-deformed Araki-Woods algebras.} Bull. Lond. Math. Soc. {\bf 52} (2020), 1010--1023.


%
%
\bibitem[HP17]{HP17} {\sc C. Houdayer, S. Popa}, {\it Singular masas in type ${\rm III}$ factors and Connes' bicentralizer property.} To appear in Proceedings of the 9th MSJ-SI ``Operator Algebras and Mathematical Physics''. {\tt arXiv:1704.07255}


\bibitem[HS88]{HS88} {\sc Haagerup, E. St\o rmer},
{\it Pointwise inner automorphisms of von Neumann algebras.} With an appendix by C.\ Sutherland. J. Funct. Anal. {\bf 92} (1990), 177--201.

\bibitem[HS90]{HS90} {\sc Haagerup, E. St\o rmer},
{\it Equivalence of normal states on von Neumann algebras and the flow of weights.} Adv. Math. {\bf 83} (1990), 180--262.

\bibitem[HU15]{HU15} {\sc C. Houdayer, Y. Ueda}, {\it Asymptotic structure of free product von Neumann algebras.} Math. Proc. Cambridge Philos. Soc. {\bf 161} (2016), 489--516.

\bibitem[HW98]{HW98} {\sc U.Haagerup, C. Winsl\o w}, {\it The Effros-Mar\' echal Topology in the Space of Von Neumann Algebras} Amer. J. of Math., {\bf 120(3)} (1998), 567--617.


\bibitem[IM19]{IM19} {\sc Y. Isono, A. Marrakchi}, {\it Tensor product decompositions and rigidity of full factors.} Ann. Sci. \' Ec. Norm. Sup\'e r. {\bf 55(4)}, (2022), 109-139.

\bibitem[Is23]{Is23} {\sc Y. Isono}, {\it Haagerup and St\o rmer's conjecture for pointwise inner automorphisms.} preprint {\tt arXiv:2309.05279}.

\bibitem[IT23]{IT23} {\sc A. Ioana, H. Tan}, {\it On existentially closed $\II_1$ factors.} preprint {\tt arXiv:2306.00474}.

\bibitem[IV15]{IV15} {\sc A. Ioana, S. Vaes}, {\it Spectral gap for inclusions of von Neumann algebras.} Appendix to the article {\it Cartan subalgebras of amalgamated free product $\II_1$ factors} by {\sc A. Ioana} in Ann. Sci. \'{E}cole Norm. Sup. {\bf 48} (2015), 71--130.

%
\bibitem[Ka67]{Ka67} {\sc R.V. Kadison}, {\it Problems on von Neumann algebras.} Baton Rouge Conference, 1967 (unpublished).
%
%


\bibitem[Ma17]{Ma17} {\sc A. Marrakchi}, {\it Stability of products of equivalence relations.} Compositio Mathematica, {\bf 154(9)} (2017), 2005--2019.

\bibitem[Ma18]{Ma18} {\sc A. Marrakchi}, {\it Full factors, bicentralizer flow and approximately inner automorphisms.} Invent. Math. {\bf 222(1)} (2018), 375--398.

\bibitem[Ma19]{Ma19} {\sc A. Marrakchi}, {\it On the weak relative Dixmier property.} Proc. London Math. Soc. {\bf 122(1)} (2019), 118--123.

\bibitem[Mac49]{Mac49} {\sc G. W. Mackey}, {\it A theorem of Stone and von Neumann.} Duke Math. J. {\bf 16} (1949), 313--326.

\bibitem[Mas03]{Mas03} {\sc T. Masuda}, {\it An analogue of Connes-Haagerup approach for classification of subfactors of type ${\rm III_{1}}$.} J. Math. Soc. Japan {\bf 57} (2005), 959--1001.

\bibitem[Mas20]{Mas20} {\sc T. Masuda}, {\it On the Relative Bicentralizer Flows and the Relative Flow of Weights of Inclusions of Factors of Type $\III_1$.} Publ. Res. Inst. Math. Sci. {\bf 56} (2020), no. 2, pp. 391–400.

\bibitem[MV23]{MV23} {\sc A. Marrakchi, S. Vaes}, {\it Ergodic states on type $\III_1$ factors and ergodic actions.} (2023), {\tt arXiv:2305.14217}.

\bibitem[NT81]{NT81} {\sc M. Nagisa, J. Tomiyama}, {\it Completely positive maps in the tensor products of von Neumann algebras.} J. Math. Soc. Japan. {\bf 33(3)} (1981), 539--550.

\bibitem[Oc85]{Oc85} {\sc A. Ocneanu}, {\it Actions of discrete amenable groups on von Neumann algebras.} Lecture Notes in Mathematics, {\bf 1138}. Springer-Verlag, Berlin, 1985. iv+115 pp.

\bibitem[Ok21]{Ok21} {\sc R. Okayasu}, {\it A note on injectivite factors with trivial bicentralizer.} (2021), {\tt arXiv:2106.05464}.

\bibitem[Oz10]{Oz10} {\sc N. Ozawa}, {\it A comment on free group factors.} Noncommutative harmonic analysis with applications to probability II, 241—245, Banach Center Publ., 89, Polish Acad. Sci. Inst. Math., Warsaw, 2010.

%

\bibitem[Po81]{Po81} {\sc S. Popa}, {\it On a problem of R.V.\ Kadison
on maximal abelian $\ast$-subalgebras in factors.} Invent. Math. {\bf 65} (1981), 269--281. 


\bibitem[Po85]{Po85} {\sc S. Popa}, {\it Hyperfinite subalgebras normalized by a given automorphism and related problems.} Operator algebras and their connections with topology and ergodic theory (Bu\c{s}teni, 1983), 421--433, Lecture Notes in Math., {\bf 1132}, Springer, Berlin, 1985. 

\bibitem[Po94]{Po94} {\sc S. Popa}, {\it Classification of amenable subfactors of type II.} Acta Math. {\bf 172} (1994), 163-255.

\bibitem[Po95]{Po95} {\sc S. Popa}, {\it Classification of subfactors and their endomorphisms.} CBMS Regional Conference Series in Mathematics, {\bf 86}. Published for the Conference Board of the Mathematical Sciences, Washington, DC; by the American Mathematical Society, Providence, RI, 1995. x+110 pp.
%

\bibitem[Po21]{Po21} {\sc S. Popa}, {\it On Ergodic Embeddings of Factors.} Commun. Math. Phys. {\bf 384} (2021), 971--996.

%

\bibitem[SZ99]{SZ99} {\sc S. Str\v{a}til\v{a}, L. Zsid\' o}, {\it The Commutation Theorem for Tensor Products over von Neumann Algebras.} Journal of Functional Analysis {\bf 165(2)} (1999), 293--346.

\bibitem[Ta03]{Ta03} {\sc M. Takesaki}, {\it Theory of operator algebras. ${\rm II}$.} Encyclopaedia of Mathematical Sciences, {\bf 125}. Operator Algebras and Non-commutative Geometry, 6. Springer-Verlag, Berlin, 2003. xxii+518 pp.




\end{thebibliography}

\end{document}